\documentclass[reqno]{amsart}
\usepackage[margin=1.2in]{geometry}

\usepackage{amssymb,mathrsfs}
\usepackage{amsmath}
\usepackage{amsthm}
\usepackage{xcolor}
\usepackage{graphicx}
\usepackage{verbatim}
\usepackage{enumerate}
\usepackage[inline,shortlabels]{enumitem}%
  \setlist{topsep=3pt, itemsep=0pt,partopsep=0pt}
\usepackage{amscd}
\usepackage{epic}
\usepackage{paracol}
\usepackage{multirow}
\usepackage[all]{xy}
\usepackage{tikz-cd}
\usepackage[colorlinks,citecolor=blue,urlcolor=black,pdfencoding=auto,psdextra]{hyperref}

\let\p\partial 
\mathchardef\O="24F %
\def\ZZ{\mathbb Z}\def\RR{\mathbb R}
\def\QQ{\mathbb Q}

\def\PP{\mathbb P}\def\AA{\mathbb A}

\let\geL\succeq 
\def\geQ{\geL_{\mathbb Q}} 
\def\simQ{\sim_{\mathbb Q}}

\def\NFun#1#{\NFunto#1}\def\NFunto#1#2{\def#1{\mathop{\rm#2}\nolimits}}
\NFun\Tr{Tr}
\NFun\Hom{Hom}\NFun\Ker{Ker}
\NFun\Pic{Pic}\NFun\Alb{Alb}\NFun\alb{alb}
\NFun\im{im}\NFun\Im{Im}
\NFun\rk{rank}
\NFun\id{id}
\NFun\codim{codim}
\NFun\coker{coker}\NFun\Coker{Coker}
\NFun\Ann{Ann}
\NFun\Der{Der}
\NFun\Spec{Spec}
\NFun\Supp{Supp}
\NFun\calSpec{\mathcal Spec}
\NFun\Proj{Proj}\NFun\bfProj{\mathbf{Proj}}
\NFun\Sym{Sym}\NFun\Lie{Lie}
\NFun\Aut{Aut}\NFun\Isom{Isom}\NFun\varIsom{\underline{Isom}}
\NFun\pdeg{pdeg}
\NFun\Var{Var}
\NFun\Char{char}
\NFun\ord{ord}
\NFun\Sing{Sing}
\NFun\Frob{Frob}
\newtheorem{thm}{Theorem}[section]
\newtheorem{lem}[thm]{Lemma}
\newtheorem{cor}[thm]{Corollary}
\newtheorem{prop}[thm]{Proposition}

\theoremstyle{definition}

\newtheorem{examp}[thm]{Example}%

\newtheorem{ques}{Question}
\newtheorem{rmk}[thm]{Remark}

\theoremstyle{remark}

\newtheorem{assumption}[thm]{Assumptions}

\theoremstyle{plain}

\date{\today}

\author{Jingshan Chen}
\email{chjingsh@hbmzu.edu.cn}
\address{School of Mathematics and Statistics, Hubei Minzu University, Enshi 445000, P.R.China.}

\author{Chongning Wang}
\email{chnwang@mail.ustc.edu.cn}
\address{School of Mathematical Science, University of Science and Technology of China, Hefei 230026, P.R.China.}

\author{Lei Zhang}
\email{zhlei18@ustc.edu.cn}
\address{School of Mathematical Science, University of Science and Technology of China, Hefei 230026, P.R.China.}

\begin{document}
\title[Irregular threefolds with numerically trivial canonical divisor]{Irregular threefolds with numerically trivial canonical divisor}

\begin{abstract}
  In this paper, we classify irregular threefolds with numerically trivial canonical divisors in positive characteristic.
  For a threefold, if its Albanese dimension is not maximal, then the Albanese morphism will induce a fibration which either maps to a curve or is fibered by curves.
  In practice, we treat arbitrary dimensional irregular varieties with either one dimensional Albanese fiber or one dimensional Albanese image.
  We prove that such a variety carries another fibration transversal to its Albanese morphism (a ``bi-fibration'' structure), which is an analog structure of bielliptic or quasi-bielliptic surfaces.
  In turn, we give an explicit description of irregular threefolds with trivial canonical divisors.
\end{abstract}
\maketitle
\tableofcontents

\section{Introduction}\label{sec:intro}
Over the field $\mathbb C$ of complex numbers, Enriques and Kodaira's classification of surfaces with numerically trivial canonical divisors ($K$-trivial), has four basic classes: abelian, bielliptic, K3, and Enriques surfaces. This classification is achieved by considering their Albanese morphisms: for a $K$-trivial surface $X$, the Albanese morphism $a_X\colon X\to A$, according to $\dim a_X(X) = 2,1,0$, is respectively an isomorphism, an elliptic fibration to an elliptic curve, or a trivial morphism. When $\dim a_X(X) = 1$, $X$ has another elliptic fibration $g\colon X\to \mathbb{P}^1$, thus $X$ has a so called bielliptic structure; more precisely, there is an \'etale cover $A'\to A$ such that $X\times_A A' \cong A'\times F$, where $F$ is a general fiber of $a_X$. Bielliptic surfaces were fully classified in \cite{B-DF} (see \cite[List~VI.20]{Beau96}). For higher dimensional $K$-trivial varieties, Bogomolov and Beauville (\cite{Bog75, Beau83} for smooth cases) and Kawamata (\cite[Section~8]{Kaw85} for varieties with canonical singularities) proved that the Albanese morphism $a_X\colon X\to A$ is a fibration and there exists an \'etale covering $A'\to A$ such that $X\times_A A' \cong F\times A'$, where $F$ is a general fiber.
Especially, for an irregular $K$-trivial variety $X$, when $\dim a_X(X) = \dim X$, we have $X\cong A$.
When the Albanese image of $X$ has an intermediate dimension, namely $0<\dim a_X(X)<\dim X$, we can reduce the study of $X$ to the study of lower dimensional $K$-trivial varieties. Here, by an \emph{irregular variety} we mean a normal projective variety $X$ with \emph{irregularity} $q(X):=\dim \mathrm{Pic}^0(X) >0$, or equivalently, the Albanese morphism $a_X\colon X\to A$ is not trivial.

Over a ground field of positive characteristic, Bombieri and Mumford \cite{BM77, BM76} classified $K$-trivial irregular surfaces $X$: if $\dim a_X(X) =2$, then $X=A$; if $\dim a_X(X) =1$, then a general fiber of the Albanese morphism is either an elliptic curve or a rational curve with a cusp, and $X$ has a bielliptic or a quasi-bielliptic structure accordingly.
Note that, in Bombieri and Mumford's classification of (quasi-)bielliptic surfaces, a key step is to show that there is a rational pencil of elliptic curves on $X$ that is transversal to the Albanese morphism of $X$, that is,
$X$ carries a ``bi-fibration'' structure:
$$\xymatrix@R=4ex@C=4ex{&X\ar[d]_{a_X}\ar[r]^>>>>{g} & Z\cong \mathbb{P}^1.\\
&A  &
}$$
Here, ``transversal'' means that a general fiber of $g$ is dominant and finite over $A$ under the morphism $a_X$.

For higher dimensional $K$-trivial varieties in positive characteristic, we know that if $X$ is of maximal Albanese dimension, then $X$ is birational to an abelian variety by \cite{HPZ19}. In the past decade, a series of progresses have been made in understanding the positivity of the direct image of (pluri)canonical sheaves in positive characteristics (\cite{Pat14, Eji17, Eji19}, etc.). The application of the powerful positivity engine to the Albanese morphism $a_X\colon X\to A$ induces some remarkable results for higher dimensional $K$-trivial (or more generally $-K$ nef) varieties. A pivotal result by \cite{PZ19} establishes that if $X$ is weakly ordinary (meaning the Frobenius pullback $F^*\colon H^d(\mathcal{O}_X) \to H^d(\mathcal{O}_X)$ is a $\sigma$-linear isomorphism, or equivalently, $X$ is globally $F$-splitting), then $X$ admits a Beauville-Bogomolov type decomposition. Roughly speaking, there is an isogeny $A' \to A$ such that $X\times_AA' \cong A' \times F$, where $F$ is a general fiber of $a_X$. Recently, Ejiri and Patakfalvi \cite{EP23} prove that when $-K$ is nef, under certain conditions on singularities, the Albanese morphism $a_X$ is surjective, and the intermediate variety $Y$ arising from Stein factorization of $a_X$ is either purely inseparable over $A$ or isomorphic to $A$.
 Later, the authors \cite{CWZ23} prove that, in case $a_X$ is of relative dimension one, $a_X\colon X\to A$ is a fibration. Under additional conditions that $\dim a_X(X)=1$ or that the generic geometric fiber is strongly $F$-regular, in \cite{EP23} and \cite{Eji23} the authors prove that the fibers of $f$ are isomorphic to each other.

 It is worth mentioning that the assumptions of the decomposition theorems above avoid a ``bad phenomenon'' in positive characteristic: the general fiber of the Albanese morphism might have bad singularities, and sometimes it is non-reduced (see for example \cite{Moret-Bailly79} or \cite[Section~3]{Schroer04}).
 A natural question arises:
\begin{ques}
Let $X$ be an irregular $K$-trivial variety. Is the Albanese morphism $a_X\colon X\to A$ a fibration? Does there exist an isogeny $A' \to A$ of abelian varieties such that $X\times_A A' \cong A'\times F$?
\end{ques}

In this paper, we focus on irregular $K$-trivial threefolds and treat this ``bad phenomenon''.
Our approach applies to not only threefolds but also arbitrary dimensional $X$ whose Albanese morphism $a_X \colon X\to A$ satisfies one of the following conditions:
\begin{itemize}
  \item the Albanese image $a_X(X)$ is of dimension one (Section~\ref{sec:str-thms-albdim=1});
  \item $a_X\colon X\to a_X(X)$ is  of relative dimension one (Section~\ref{sec:str-thms-reldim=1}).
\end{itemize}
In both cases we derive a ``bi-fibration'' structure 
\begin{equation}\label{eq:JZBU}
  \vcenter{\xymatrix@R=4ex@C=4ex{&X\ar[d]_{a_X}\ar[r]^{g} & Z \\ &A  & }}
\end{equation}
which is a crucial step to obtain an explicit structure of $X$.

We explain our strategy as follows.
In the first case, $\dim a_X(X) = 1$, we follow the strategy earlier used in \cite{PZ19, EP23, Eji23}: applying the positivity engine to construct a semi-ample divisor $D$ on $X$, which is relatively ample over $A$ and has $\nu(D) = \dim X-1$ (Theorem~\ref{thm:pos-EP}). In the second case, $f=a_X\colon X\to A$ is a fibration with the generic fiber $X_{\eta}$ being a curve of arithmetic genus one, thus we treat the following three cases separately:
\begin{itemize}
  \item[(C1)]
    $X_{\eta}$ is smooth over $k(\eta)$. In this case, $f\colon X\to A$ is an elliptic fibration with $\mathrm{var}(f) = 0$ by \cite[Theorem~2.14]{CZ15}, so we can apply the Isom functor developed in \cite{PZ19}.
\end{itemize}
  If $X_{\eta}$ is not smooth over $k(\eta)$, we can show that there is a natural movable divisor which induces the fibration $g\colon X\to Z$ as follows:
\begin{itemize}
\item[(C2)]
  $X_{\bar{\eta}}$ is reduced but not smooth. Then $a_X\colon X \to A$ is fibred by quasi-elliptic curves, we prove that the divisor supported on the singular locus $\Sigma$ of the fibers is movable as required (Theorem~\ref{thm:hor-map-1}).
\item[(C3)]
  $X_{\bar{\eta}}$ is not reduced, which means that $a_X\colon X \to A$ is inseparable. Then the required movable divisor arises from global sections of  $\Omega_A$ when doing Frobenius base change (Section \ref{sec:base-change}). This should be attributed to Ji and Waldron's observation \cite{JW21} ({see \cite[Proposition~3.4]{CWZ23}}).
\end{itemize}
We see that in the cases (C2) and (C3), the ``bad phenomenon'' that $X_{\bar{\eta}}$ is singular becomes an advantage.
\smallskip
Finally, we apply the two fibrations (\ref{eq:JZBU}) to derive the explicit structure of $X$.

For threefolds, we have a precise description as follows.
\begin{thm}[{=\thinspace Theorem~\ref{thm:3fold}}]\label{thm:main}
  Let $k$ be an algebraically closed field of characteristic $p>0$.
  Let $X$ be a normal $\QQ$-factorial projective threefold over $k$ with $K_X\equiv0$. 
  Denote the Albanese morphism of $X$ by $a_X\colon X\to A$, and assume $\dim a_X(X) >0$.
  Then the following statements hold.
  \begin{itemize}
    \item[\rm(A)] If $\dim a_X(X) =3$, then $X=A$.
    \item[\rm(B)] If $\dim a_X(X) =1$, under the condition that 
      \begin{itemize}
        \item either {\rm(i)} $X$ is strongly $F$-regular and $K_X$ is $\mathbb Z_{(p)}$-Cartier;
        \item or {\rm(ii)} $X$ has at most terminal singularities and $p \geq 5$,
      \end{itemize}
      then $a_X$ is a fibration and there exists an isogeny of elliptic curves $A'\to A$, such that  $X\times_A A' \cong A'\times F$, where $F$ is a general fiber of $a_X$. More precisely, $X\cong A'\times F/H$,
      where $H$ is a finite group subscheme of $A'$ acting diagonally on $A'\times F$.
    \item[\rm(C)] If $\dim a_X(X) =2$, then $a_X$ is a fibration and $X$ falls into one of the following three cases:
    \begin{itemize}
      \item[\rm(C1)] the generic fiber $X_{\eta}$ of $a_X$ is smooth.
        Then there exists an isogeny of abelian surfaces $A'\to A$, such that  $X\times_A A' \cong A'\times E$, where $E$ is an elliptic curve appearing as a general fiber of $a_X$.
        More precisely, $X\cong A'\times E/H$, where $H$ is a finite group subscheme of $A'$ acting diagonally on $A'\times E$, with a complete classification as in Section~\ref{sec:explicite-of-C1}.
      \item[\rm(C2)] $X_\eta$ is non-smooth but geometrically reduced. Then $p=2$ or $3$.
    Denote by $F_{A/k}\colon A_1:=A^{(-1)}\to A$ the relative Frobenius over $k$, $X_1$ the normalization of $X \times_{A} A_1$ and $f_1\colon X_1\to A_1$ the induced morphism.
   Then $f_1\colon X_1\to A_1$ is a smooth fibration fibred by rational curves, which falls into one of the following specific cases:
      \begin{itemize}
        \item[\rm(2.1)] In this case, $f_1\colon X_1\to A_1$ admits a section, and
        \begin{itemize}
          \item[\rm(2.1a)] either $X_1\cong \mathbb{P}_{A_1}(\mathcal{O}_{A_1}\oplus \mathcal{L})$, where $\mathcal{L}^{\otimes p+1}\cong \mathcal{O}_X$; or
          \item[\rm(2.1b)] $X_1\cong \mathbb{P}_{A_1}(\mathcal{E})$, where $\mathcal{E}$ is a unipotent vector bundle of rank two, and
          there exists an \'etale cover $\mu\colon A_2\to A_1$ of degree $p^v$ for some $v\le 2$ such that $\mu^*F^{(2-v)*}_{A_1/k}\mathcal{E}$ is trivial.
        \end{itemize}
        \item[\rm(2.2)] In this case, $p=2$, and there exists a purely inseparable isogeny $A_2 \to A_1$ of degree two, such that $X_2 :=X_1\times_{A_1}A_2$ is a projective bundle over $A_1$ described as follows
        \begin{itemize}
          \item[\rm(2.2a)] $X_2\cong \mathbb{P}_{A_2}(\mathcal{O}_{A_2}\oplus \mathcal{L})$, where $\mathcal{L}^{\otimes 4}\cong \mathcal{O}_{A_2}$; or
          \item[\rm(2.2b)] $X_2\cong \mathbb{P}_{A_2}(\mathcal{E})$, where $\mathcal{E}$ is a unipotent vector bundle of rank two, and there exists an \'etale cover $\mu\colon A_3\to A_2$ of degree $p^v$  for some $v\le 2$ such that $\mu^*F^{(2-v)*}_{A_1/k}\mathcal{E}$ is trivial.
        \end{itemize}
      \end{itemize}
    \item[\rm(C3)] $X_{\eta}$ is not geometrically reduced.
      In this case, we also have $p=2$ or $3$.
      Let $X_1$ be the normalization of $(X \times_{A} A_1)_{\rm red}$, where $A_1:=A^{(-1)}\to A$ is the relative Frobenius. Then the projection $X_1 \to A_1$ is a smooth morphism, and either
      \begin{itemize}
        \item[\rm(3.1)] $X_1=A_1\times\mathbb P^1$ and $X=A_1\times\mathbb P^1/\mathcal F$ for some smooth rank~one foliation $\mathcal F$ which is described concretely in Section~\ref{sec:descrip-C31}; or
        \item[\rm(3.2)] $p=2$, and there exists an isogeny of abelian surfaces $\tau\colon A_2 \to A_1$ such that $X_1\times_{A_1}A_2 \cong A_2\times\mathbb P^1$, where either
        \begin{itemize}
          \item[\rm(3.2a)] $\tau\colon A_2 \to A_1$ is an \'etale of degree two; or
          \item[\rm(3.2b)] $\tau=F_{A_1/k}\colon A_2 := A^{(-2)} \to A_1$ is the relative Frobenius.
        \end{itemize}
      \end{itemize}
    \end{itemize}
  \end{itemize}
\end{thm}

\begin{rmk}%
  (1) Theorem~\ref{thm:main} shows that the Albanese fibration $f\colon X\to A$ splits into the product after a sequence of ``normalized'' Frobenius base changes and an étale base change; say \[
  \xymatrix{
    A_2 \times F \ar[r]\ar[d] & X_1\ar[r]\ar[d] & X\ar[d]^f \\ A_2 \ar[r]^{\text{ét}} & A_1 \ar[r]^{F_{A/k}^n} & A \rlap{\,,}
  }
\]
where $F$ is the normalization of a general fiber of $f$, $X_1 = (X\times_A A_1)_{\rm red} ^\nu$ and the left square is Cartesian.
By Remark~\ref{rmk:sep-foliation-lift}, the same holds if we first do an étale base change following by Frobenius base changes, namely we have the following sequence of normalized base changes: \[
  \xymatrix{
    B_1 \times F \ar[r]\ar[d] & X_B\ar[r]\ar[d] & X\ar[d]^f \\ B_1 \ar[r]^{F_{B/k}^n} & B \ar[r]^{\text{ét}} & A \rlap{\,.}
  }
\]

(2) In Case (C), we get a full classification for the cases (C1) and (C3.1), and provide examples for the remaining cases (Section \ref{sec:3-folds}).
\end{rmk}

\subsection{Effectivity of the pluricanonical map of threefolds}%
Over the field of complex numbers, for terminal $K$-trivial threefolds $X$, Kawamata \cite{Kaw86} showed that there exists a positive calculable integer $m_0$ such that $m_0 K_X \sim 0$; and by \cite{Beauville83Rmk,Morrison86} and finally \cite{Oguiso93}, the smallest $m_0$ is $2^5 \cdot 3^3 \cdot 5^2 \cdot 7 \cdot 11 \cdot 13 \cdot 17 \cdot 19$.
It is natural to ask the following question.
\begin{ques}\label{ques:Effectivity}
  Does there exist a positive integer $N$ such that $N K_X \sim 0$ for all terminal $\QQ$-factorial $K$-trivial threefolds over an algebraically closed field of characteristic $p>0$?
\end{ques}

Applying the structure theorem \ref{thm:main}, we can prove the following effectivity result when $\dim a_X(X) = 2$.
\begin{cor}[see Section~\ref{sec:effectivity}]\label{cor:effectivity}
  Let $X$ be a terminal $\QQ$-factorial threefold such that $K_X \equiv 0$ and $q = 2$, then $(2^4 \cdot 3^3) K_X \sim 0$.
\end{cor}
For the case $\dim a_X(X) = 1$, we may break the effectivity problem into the following two questions.
\begin{ques}\label{ques:eff}   
  Let $f\colon X \to A$ be a fibration belonging to Case (B) of Theorem~\ref{thm:main}. Denote by $F$ a general fiber of $f$. Then there exists an integer $N_F>0$ such that $N_FK_F \sim 0$.

  (1) Is there a uniform bound of $N_F$? Equivalently, is there a positive integer $N$ such that $NK_F \sim 0$ {holds for every fibration} $f\colon X \to A$ in Case (B)?
  \smallskip

  Let $H$ be a finite subgroup scheme of an elliptic curve. Assume that $H$ acts on $F$ as in Case (B).
  Then there is a natural group homomorphism $H \to \mathrm{GL}(H^0(F, N_FK_F)) \cong \mathbb G_m$; we denote its image by $\overline{H}$, which is a finite group scheme over $k$.

  (2) Is there a uniform bound of the order of $\overline{H}$?
\end{ques}

\subsection{Notation and conventions}
\begin{itemize}
\item 
  By a \emph{variety} we mean an integral quasi-projective scheme over a field.
  By a \emph{log pair} $(X,\Delta)$, we mean a pair consisting of a variety $X$ and an effective $\QQ$-divisor $\Delta$ such that $K_X+\Delta$ is $\QQ$-Cartier.
  We denote by $\nu\colon X^\nu \to X$ the normalization morphism of a variety $X$.
\item
By a \emph{fibration} we mean a projective morphism $f\colon X \to Y$ of normal varieties such that $f_*\mathcal{O}_X = \mathcal{O}_Y$.
\item
Let $X$ be a normal variety and $D$ a Weil divisor on $X$. Denote by $\mathcal{O}_X(D)$ the reflexive sheaf associated with $D$.
Note that if $D_1$,$D_2$ are Weil divisors, then $\mathcal O_X(D_1 + D_2) \cong (\mathcal O_X(D_1)\otimes\mathcal O_X(D_2))^{\vee\vee}$.
\item
For a projective morphism $f\colon X \to Y$ of normal varieties, a Weil divisor $D$ on $X$ is called {\it $f$-exceptional} if $f(\Supp D) \subset Y$ has codimension $\ge2$, {\it $f$-vertical} if $f(\Supp D)$ has codimension $\ge1$, and {\it $f$-horizontal} if each irreducible component of $D$ is dominant over $Y$.
\item
For a morphism $\sigma\colon Z \to X$ of varieties and a divisor $D$ on $X$ such that the pullback $\sigma^*D$ is well defined, we often use $D|_Z$ to denote $\sigma^*D$ for simplicity.
\item
Let $K=\mathbb Z,\mathbb Q$ or $\mathbb R$.
Let $D$ be a $K$-divisors on $X$, namely $D\in N^1(X)\otimes K$. We say that $D$ is effective, with the notation $D\ge0$, if all coefficients are non-negative.
By $D\succeq_K0$, we mean that there exists an effective $K$-divisor $D'$ such that $D'\sim_K D$.
When $K=\mathbb Z$, we also denote $D\succeq_{\mathbb Z}0$ by $D\geL0$ for simplicity.
\item
For every effective integral divisor $E$ on $X$, the inclusion $\mathcal{O}_X(D) \subseteq \mathcal{O}_X(D+E)$ allows us to regard $H^0(X,\mathcal{O}_X(D))$ as a subspace of $H^0(X,\mathcal{O}_X(D+E))$. This space coincides with $H^0(X,\mathcal{O}_X(D))\otimes 1_E$, where $1_E \in H^0(X, \mathcal{O}_X(E))$ denotes the section corresponding to the constant function $1 \in K(X)$.
\end{itemize}

{Throughout this paper, we let $k$ be an algebraically closed field of characteristic $p>0$, and unless otherwise mentioned we assume varieties are defined over $k$.}

\medskip
{\small \noindent\textit{Acknowledgments.}
The authors would like to thank the referee for giving many helpful comments to improve the presentation and the proof.
This research is partially supported by National Key R and D Program of China (No. 2020YFA0713100), CAS Project for Young Scientists in Basic Research (No. YSBR-032) and NSFC (No.12122116 and No. 12471495).
The first author is also supported by Hubei Minzu University (Grant No. XN24040).
}

\section{Preparations}\label{sec:prelim}
In this section, we collect some basic notions and facts which we will use in the sequel.

\subsection{Frobenius morphisms}\label{sec:Frob-mor}%
Let $f\colon X\to \Spec k$ be a variety over $k$.
We denote by $F_X\colon X\to X$ the absolute Frobenius morphism of $X$.
Set $X^{(1)} := X \times_{k,F_k} k$. We denote by $F_{X/k}\colon X \to X^{(1)}$ the {\em relative Frobenius} of $X$ over $k$, which fits into the following commutative diagram
\[
  \xymatrix{
    X \ar@/^8mm/[rr]^{F_X} \ar[r]^{F_{X/k}} \ar[rd]_f & X^{(1)} \ar[d]\ar[r] & X \ar[d]^f \\
                                            & \Spec k \ar[r] ^{F_k} & \Spec k\,.     
  }
\]
Note that since $k$ is perfect, the morphism $X^{(1)} \to X$, though not $k$-linear, is an isomorphism as schemes.
For this reason, we also denote the relative Frobenius by $F_{X/k}\colon X^{(-1)}\to X$.

\subsection{Foliations and purely inseparable morphisms}\label{sec:foliation}

Let $Y$ be a normal variety over $k$, and denote by $\mathcal T_Y:=\Omega_{Y/k}^{\vee}$ the tangent sheaf.
A {\it foliation} on $Y$ is a saturated subsheaf $\mathcal F\subseteq \mathcal T_Y$, which is $p$-closed ($\mathcal F^p \subseteq \mathcal F$) and involutive ($[\mathcal F,\mathcal F] \subseteq \mathcal F$).
The subsheaf $\mathop{\rm Ann} \mathcal F \subseteq \mathcal O_Y$ is a subring containing $\mathcal O^p_Y$, and thus gives a natural morphism $\pi\colon Y \to Y/\mathcal F:=\mathrm{Spec}( \mathop{\rm Ann}\mathcal F)$ over $k$. By the construction, the  relative Frobenius morphism $F_{Y/k}\colon Y \to Y^{(1)}$ factors through $\pi\colon Y \to Y/\mathcal F$, thus $\pi\colon Y \to Y/\mathcal F$ is a purely inseparable morphism of height one.
In fact, there is a one-to-one correspondence (\cite{Eke87} or \cite[Proposition~2.9]{PW22}): $$
    \newcommand\amsatop[2]{\genfrac{}{}{0pt}{}{#1}{#2}}
    \left\{\amsatop{\text{foliations}}{\mathcal F\subseteq \mathcal T_Y}\right\}
    \leftrightarrow
    \left\{\vcenter{\hbox{finite purely inseparable morphisms $\pi\colon Y\to X$}%
                   \hbox{over $k$ of height one with $X$ normal}}\right\}.
$$
which is given by 
\[
  \mathcal F ~\mapsto ~\pi\colon Y \to Y/\mathcal F \text{ \ and \ } \pi\colon Y\to X~\mapsto \mathcal F_{Y/X},
\]
where $\mathcal F_{Y/X}$ is the subsheaf of $\mathcal{T}_Y$ annihilated by $\mathrm{im }(\pi^*\Omega_X^1 \to \Omega_Y^1)$.
Recall (see for example \cite[Proposition~2.10]{PW22}) the following formula 
\begin{equation}\label{eq:pullback-cano}
    \pi^*K_X \sim K_Y - (p-1)\det \mathcal F_{X/Y}\sim K_Y + (p-1)\det \Omega_{X/Y}^1.
\end{equation}

If $Y$ is a smooth variety, we call a foliation $\mathcal F$ on $Y$ a {\it smooth foliation} if $\mathcal F \subseteq \mathcal T_Y$ is a subbundle, namely both $\mathcal F$ and $\mathcal T_Y/\mathcal F$ are locally free.
In this case by \cite[Proposition~2.4]{Eke87} or \cite[Proposition~3.1.9]{MP97}, the quotient $Y/\mathcal F$ is smooth if and only if $\mathcal F$ is a smooth foliation.

\subsection{``Pushing down'' and pullback foliations}\label{sec:etaleLift}

\subsubsection{``Pushing-down'' foliations along a fibration}%
Let $f\colon X\to S$ be a fibration of normal varieties over $k$, and let $\mathcal F$ be a foliation on $X$.
We recall the ``pushing-down'' foliation of $\mathcal F$ constructed in \cite[Section~3.1.1]{CWZ23} as follows.
By results of the previous section, we have the following commutative diagram
$$  \xymatrix{
  X \ar[r]^<<<<<\pi \ar[d]_f & \bar X = X/\mathcal F \ar[d]_{\bar f}\ar[r] &X^{(1)}\ar[d] \\  S \ar[r]^<<<<<<<<\sigma & \bar S  \ar[r] & S^{(1)},
}$$
where $\bar{f}\colon \bar{X}\to \bar{S}$ arises from the  Stein factorization of  $\bar{X} \to X^{(1)} \to S^{(1)}$, and hence $\bar{S}$ is obviously between $S$ and $S^{(1)}$. The purely inseparable morphism $\sigma\colon S \to \bar{S}$ corresponds to a foliation $\mathcal G$ on $S$ such that $\bar{S}=S/\mathcal{G}$. The following is another characterization of $\mathcal{G}$.

\begin{prop}[{\cite[Lemma~3.3]{CWZ23}}]\label{prop:push-foliation}
Let notation be as above.
Assume moreover that $S$ is regular.
Let $\mathcal F \subseteq \mathcal T_X \buildrel\eta\over\to f^*\mathcal T_S$ be the natural homomorphisms.
Then
\begin{itemize}
  \item[\rm(1)] the sheaf $\mathcal G$ is the minimal foliation on $S$ such that $\eta(\mathcal F) \subseteq f^*\mathcal G$ holds generically;
  \item[\rm(2)] if $f$ is separable, then $\bar{f}$ is separable if and only if $\eta\colon \mathcal F\to f^*\mathcal G$ is generically surjective.
\end{itemize}
\end{prop}


\subsubsection{Pullback of a foliation}%
Let $\tau\colon Y\to X$ be a generically finite, separable and dominant morphism of normal varieties, and let $\mathcal F\subset \mathcal T_X$ be a foliation on $X$.
We can define the pullback foliation $\mathcal F_Y$ on $Y$ as follows.
The natural homomorphism $\mathcal T_Y \to \tau^*\mathcal T_X$ is generically isomorphic.
So over some open subset of $Y$, $\tau^*\mathcal F$ can be viewed as a subsheaf of $\mathcal T_Y$ under this isomorphism.
We define $\mathcal F_Y$ to be the saturation of $\tau^*\mathcal F$ in $\mathcal T_Y$.
One can check that $\mathcal F_Y$ is a foliation on $Y$.

\begin{lem}[{\cite[Lemma~3.0.7]{Posva24}}]\label{lem:etaleLift}
  Let $\tau\colon Y\to X$ be a finite étale morphism between normal varieties.
  Let $\mathcal F$ be a foliation on $X$ and let $\mathcal F_Y$ be the pullback foliation on $Y$.
  There is an étale morphism $\sigma\colon Y/\mathcal F_Y\to X/\mathcal F$ which give a Cartesian square \[
    \xymatrix@R=3ex{
      Y \ar[r] \ar[d]_{\tau} & Y/\mathcal F_Y \ar[d]^{\sigma}\\ X \ar[r] & X/\mathcal F \rlap.}
  \]
  Moreover, if $\tau\colon Y\to X$ is a Galois covering, then so is $\sigma$.
\end{lem}
\begin{proof}
  The first assertion is \cite[Lemma~3.0.7]{Posva24}. For the second, assume that $G$ is the group of automorphisms of $Y$ over $X$.
  Then there is a commutative diagram \[
    \xymatrix@R=3ex{
      Y/\mathcal F_Y \ar[d]^{\sigma}\ar[r] & Y \ar[r] \ar[d]_{\tau} & (Y/\mathcal F_Y)^{(1)} \ar[d]^{\sigma^{(1)}}\\
      X/\mathcal F \ar[r] & X \ar[r] & (X/\mathcal F)^{(1)} \rlap,
  }\]
  where the composition of the horizontal rows are relative Frobenius morphisms over $k$. 
  Moreover the left square is Cartesian, i.e., $Y/\mathcal F_Y \cong X/\mathcal F \times_X Y$, thus $G$ acts naturally on $Y/\mathcal F_Y$ which makes $\sigma$ a Galois covering.
\end{proof}

\begin{rmk}\label{rmk:sep-foliation-lift}
  Let $f\colon X\to S$ be a fibration of normal varieties and let $\mathcal F$ be a foliation on $X$.
  If $\tau\colon Y\to X$ is obtained by an étale base change, say $Y = X\times_S \tilde S$ where $\sigma\colon \tilde S\to S$ is finite étale, then $Y/\mathcal F_Y \cong X/\mathcal F \times_{S^{(1)}} \tilde S^{(1)}$ and the diagram
 \begin{equation}\label{eq:MYXX}
    \vcenter{\xymatrix@R3ex@C3ex{
      Y \ar[dd]\ar[dr]^\tau\ar[rr]^{\tilde{\pi}} && Y/\rlap{$\mathcal F_Y$}\ar[dr]^{\tau'}\ar@{-->}[dd]|\hole \\
      & X\ar[rr]^<<<<<\pi \ar[dd]^<<<<f && X/\rlap{$\mathcal F$}\ar[dd]^<<<<{f'} \\
      \tilde{S} \ar@{-->}[rr]|\hole \ar[rd]^\sigma  && \tilde{S}^{(1)} \ar@{-->}[dr]^{\sigma'} \\
                                                    & S \ar[rr]^{F} && S^{(1)} \rlap{\,,}
    }}
  \end{equation}
  is commutative, where $S\to S^{(1)}$ and $\tilde S\to \tilde S^{(1)}$ are the relative Frobenius morphisms over $k$ and $\sigma'$ is the natural étale morphism induced by $\sigma$.
\end{rmk}

\subsection{A property of fibred varieties under flat base changes}
\begin{lem}\label{lem:flat-base-change-integral}%
  Let $X,S,T$ be quasi-projective normal varieties over an arbitrary field. Let $f\colon X\to S$ be a separable fibration and $\sigma\colon T\to S$ a finite flat morphism.
  Then the fiber product $X\times_S T$ is integral, and it is normal if and only if its conductor divisor is zero.
\end{lem}
\begin{proof}
  Consider the tensor product of the function fields $R:= K(X) \otimes_{K(S)} K(T)$.
  Since $f$ is a fibration, $K(S)$ is algebraically closed in $K(X)$, so $\Spec R$ is irreducible by \cite[Proposition~4.3.2]{EGAIV.2}.
  Since $f$ is separable, $\Spec R$ is reduced by \cite[Proposition~4.3.5]{EGAIV.2}.
  Therefore, $\Spec R$ is integral, which implies that $Y := X \times_S T$ is generically integral (in the sense that it becomes integral restricting to certain open subsets of $X,S,T$).
  Since the base change morphism is flat, $Y$ satisfies Serre's condition $(S_2)$ by \cite[Corollary of Theorem~23.3]{MatsumuraCRT}.
  Consequently, $Y$ satisfies $(S_1)+(R_0)$, and thus it is reduced.
  Furthermore, since $Y$ is $(S_2)$ and generically irreducible, it is irreducible overall according to Hartshorne's connectedness lemma (\cite[\href{https://stacks.math.columbia.edu/tag/0FIV}{Tag~0FIV}]{stacks-project}).
  Therefore, $Y$ is integral. Since $Y$ is ($S_2$), it is normal if and only if it satisfies ($R_1$) condition, which is equivalent to the condition that the conductor divisor is zero.
\end{proof}

\subsection{Behavior of the relative canonical divisor under purely inseparable base changes}\label{sec:base-change}
We shall frequently encounter the following settings:
\begin{equation}\label{eq:base-change}
  \xymatrix{&Y\ar[r]^<<<<<\nu \ar@/^8mm/[rrr]|{\,\pi\,} \ar[rrd]_g\ar[r] &(X_T)_{\rm red}\ar[r] &X_T\ar[r]\ar[d]^{f_1} &X\ar[d]^{f}\\
  &&  &T\ar[r]^{\tau} &S\rlap{\,,} \\
}\end{equation}
where $X,S,T$ are normal quasi-projective varieties over $k$, $f\colon X \to S$ is a fibration and $\tau\colon T\to S$ is a finite purely inseparable morphism of height one.
For simplicity, we assume that $S$ and $T$ are regular so that all divisors on them are Cartier.
Let us recall the following formulas from \cite{CWZ23}.

(a) If $X_{K(T)}$ is integral, then, by \cite[Proposition~3.5]{CWZ23}, there exists an effective Weil divisor $E$ and a $g$-exceptional $\mathbb Q$-divisor $V$ on $Y$ such that
\begin{equation}\label{eq:compds}
  \pi^*K_{X/S} \simQ K_{Y/T} + (p-1)E +V.
\end{equation}

(b) 
If $S=A$ is an abelian variety, $T=A_1:=A^{(-1)}$ and $\tau\colon A^{(-1)} \to A$ is the relative Frobenius morphism over $k$, 
then $\Omega^1_{Y/X}$ is generically globally generated since $\Omega_{A_1/A}^1=\Omega_{A_1}^1 \to \Omega^1_{Y/X}$ is generically surjective. As a result, $\det (\Omega^1_{Y/X})$ has global sections (see \cite[Proposition~3.4]{CWZ23}).
By (\ref{eq:pullback-cano}), we have
\begin{equation}
  \label{eq:4C7W}
  \pi^* K_X \sim K_Y + (p-1) (E + V),\qquad E\ge 0, V\ge 0,
\end{equation}
where $E$ is a $g$-horizontal divisor and $V$ is a $g$-vertical divisor on $Y$.
If moreover, $f$ is separable, then $\omega_{X_{A_1}}$ is locally free restricted on the generic fiber of $f_1$, and thus $E$ can be chosen such that, on the generic fiber $Y_\eta$ of $g$, $(p-1)E$ coincides with the conductor divisor of the normalization of $Y_\eta$ (see \cite[Theorem~1.2]{PW22}).
If $f$ is inseparable, then $E$ contains a nontrivial movable part (see \cite[Theorem~1.1]{JW21} and \cite[Proposition~3.4]{CWZ23}).
Thus
\begin{equation}
  \label{eq:6QN0}
  \pi^* K_X \sim K_Y + (p-1) (\mathfrak M + \mathfrak F),
\end{equation}
where $\mathfrak M$ is the movable part, and $\mathfrak F$ is the fixed part.

\subsection{Numerical dimension}%

Let $X$ be a normal projective variety over $k$ and $D$ be an $\RR$-Cartier $\RR$-divisor on $X$.
The {\it numerical dimension of $D$} is
\[
  \kappa_\sigma(D)= \max\bigl\{\ell\in\mathbb N\mid \mathop{\rm lim\;inf}_{m\to\infty} (h^0(X,\mathcal O_X(\lfloor mD \rfloor+A))/m^\ell) > 0\bigr\},
\]
where $A$ is an ample divisor on $X$ and $\kappa_\sigma(D)=-\infty$ if no such $\ell$ exists.
If $D$ is nef, we use $\nu(D)$ to denote the largest natural number $j\ge0$ such that the cycle class $D^j$ is not numerically trivial.
Equivalently, $\nu(D)$ is the largest $j$ such that $(D^j\cdot H^{\dim X-j}) \ne 0$ for some ample Cartier divisor $H$ by \cite[Corollary~3.17]{Fulger-Lehmann-2017-Dual}.
Moreover, for any nef divisor $D$, we have $\kappa_\sigma(D)=\nu(D)$ by \cite[Remark~4.6]{CHMS14}.
\smallskip

The following lemma is probably well known to experts. In characteristic zero, it is a consequence of \cite[Proposition~2.1]{Kaw85p}, whose proof needs resolution of singularities. In characteristic $p$, we use smooth alteration instead.
\begin{lem}\label{lem:nu=1}%
  Let $X$ be a normal projective variety, and let $D$ be a nef $\QQ$-Cartier $\QQ$-divisor with $\nu(D) = \kappa(D) = 1$.
  Then $D$ is semi-ample.
\end{lem}

\begin{proof}
  First, we consider the case when $X$ is smooth.
  Since $\kappa(D) =1$, by replacing $D$ with its multiple, we can assume that the linear system $|D|$ contains a movable part. Write that $|D|= |M| + V$ where $|M|$ denotes the movable part and $V$ the fixed part. Since $X$ is smooth, intersections of divisors make sense.
  By restricting on the intersection of $\dim X-2$ general hypersurfaces, since $M$ is movable, from $D^2\equiv M(M+V) + V\cdot D\equiv 0$ we deduce that $M^2 \equiv 0$ and $M\cdot V\equiv 0$. From this we conclude that
  \begin{itemize}
    \item the linear system $|M|$ has no base point, hence it induces a fibration $f\colon X \to B$ where $B$ is a smooth projective curve; and
    \item the fixed part $V$ is contained in finitely many closed fibers of $f$.
  \end{itemize}
  Moreover since $D$ is nef, we see that $V$ is nef, and $V$ must be like $a_1F_1 + \cdots + a_rF_r$ where $F_i$ are closed fibers of $f$ and $a_i \in \mathbb{Q}^+$. In conclusion, there exists an effective $\mathbb{Q}$-divisor $D_B$ such that $D\sim_{\mathbb{Q}}f^*D_B$. Thus $D$ is semi-ample.

  In general, we can take a smooth alteration $\pi\colon Y\to X$ (\cite[Theorem 4.1]{deJ96}), so that $Y$ is smooth, projective and dominant over $X$.
  We see that $\pi^*D$ is nef, and $\nu(\pi^*(D)) = \kappa(Y, \pi^*(D))= \nu(D) =1$ (\cite[Remark~4.6]{CHMS14}). 
  In the previous paragraph, we have proved that $\pi^*D$ is semi-ample on $Y$, which implies that $D$ is semi-ample.
\end{proof}

\subsection{Covering theorem}%
\begin{thm}[{\cite[Theorem~10.5]{Iit82}}]\label{ct}
  Let $f\colon X \rightarrow Y$ be a proper surjective morphism between complete normal varieties over an algebraically closed field.
	If $D$ is a Cartier divisor on $Y$ and $E$ an effective $f$-exceptional divisor on $X$, then
	$$\kappa(X, f^*D + E) = \kappa(Y, D).$$
\end{thm}

\subsection{Adjunction formula and characterization of abelian varieties}%
\begin{prop}[{\cite[Proposition~4.5]{Kollar-Sing} and \cite[Theorem~4.1]{Das15}}]\label{adjunction}
  Let $X$ be a normal variety and $S$ be a prime Weil divisor of $X$.
  Let $S^\nu \to S$ be the normalization. Assume that $K_X + S$ is $\mathbb{Q}$-Cartier.
  Then
  \begin{itemize}
    \item[\rm(1)] There exists an effective $\mathbb Q$-divisor $\Delta_{S^\nu}$ on $S^\nu$ such that
      $$(K_X+ S)|_{S^{\nu}} \simQ K_{S^{\nu}} + \Delta_{S^\nu} .$$
    \item[\rm(2)] Let $V \subset S^\nu$ be a prime divisor, then $\mathrm{coeff}_V \Delta_{S^\nu} = 0$ if and only if $X,S$ are both regular at the generic point of $V$.
    \item[\rm(3)] If the pair $(S^\nu, \Delta_{S^\nu})$ is strongly $F$-regular, then $S$ is normal.
  \end{itemize}
\end{prop}

\begin{prop}[{\cite[Proposition~3.2]{EP23}, see \cite[Proposition~2.9]{CWZ23}}]\label{char-abel-var}
  Let $X$ be a normal projective variety of maximal Albanese dimension, then
\begin{itemize}
  \item[\rm(1)] $K_X \geQ 0$, and
  \item[\rm(2)] if $K_X \simQ 0$, then $X$ is isomorphic to an abelian variety.
\end{itemize}
\end{prop}

We give the following useful lemma which can be proved by the above two propositions.
\begin{lem}\label{lem:adj-charA}%
  Let $X$ be a normal $\QQ$-factorial projective variety.
  Let $D$ be a prime divisor of maximal Albanese dimension, and let $\Delta\ge0$ be a $\QQ$-divisor such that $D \not\subseteq\Supp \Delta$.
  Then
  \begin{itemize}
    \item[\rm(1)] $(K_X + D + \Delta)|_{D^\nu} \geQ 0$;
    \item[\rm(2)] if $(K_X + D+\Delta)|_{D^\nu} \simQ 0$, then $D$ is isomorphic to an abelian variety, $\Delta |_{D}\simQ 0$, and $X$ is regular at codimension-one points of $D$;
    \item[\rm(3)] if $D|_{D^\nu} \geQ 0$, and $(K_X + D + aD + \Delta)|_{D^\nu} \simQ 0$ for some $a\in\QQ_{>0}$, then $D|_{D^\nu} \simQ 0$; and as a consequence of {\rm(2)}, $D$ is isomorphic to an abelian variety, $\Delta |_{D}\simQ 0$, and $X$ is regular at codimension-one points of $D$.
  \end{itemize}
\end{lem}
\begin{proof}
  By Proposition \ref{char-abel-var} (1), we have $K_{D^\nu} \geQ 0$. 
  In turn, applying the adjunction formula (Proposition \ref{adjunction}) we have
  \[
  (K_X + D + \Delta)|_{D^\nu} \simQ K_{D^\nu} + \Delta_{D^\nu} + \Delta|_{D^\nu}\geQ 0, \]
  which is the assertion (1). 

  Next, assuming $(K_X + D+\Delta)|_{D^\nu} \simQ 0$, from the above equation we see that
  $$K_{D^\nu} \sim_{\mathbb{Q}} \Delta_{D^\nu} \sim_{\mathbb{Q}} \Delta|_{D^\nu} \sim_{\mathbb{Q}} 0.$$
  Then the assertion (2) follows immediately from Proposition~\ref{char-abel-var} (2) and Proposition~\ref{adjunction} (2).

  Finally, assuming $D|_{D^\nu} \geQ 0$, applying the adjunction formula (Proposition \ref{adjunction}) again, we have
  \[ 0\simQ (K_X + D + aD + \Delta)|_{D^\nu} \sim_{\mathbb{Q}} K_{D^\nu} + \Delta_{D^\nu} + aD|_{D^\nu} + \Delta|_{D^\nu}\geQ 0. \]
  It follows that
  $K_{D^\nu} \sim_{\mathbb{Q}} \Delta_{D^\nu}\sim_{\mathbb{Q}} D|_{D^\nu}\sim_{\mathbb{Q}} \Delta|_{D^\nu} \sim_{\mathbb{Q}} 0$.
  Then the assertion (3) follows by the same argument as above.
 \end{proof}

\subsection{Some results of elliptic fibrations}%
The following result appears as a middle step in the proof of \cite[Theorem~1.2]{CZ15}.
\begin{thm}[{\cite[Claim~3.2 and Remark~3.3]{CZ15}}]\label{rel-can-ellfib-sm}
  Let $f\colon X\to Z$ be an elliptic fibration from a normal variety $X$ onto a smooth variety $Z$. Then $\kappa(X,K_{X/Z})\ge 0$ and $\kappa(X) \geq \max\{\kappa(Z), \mathrm{Var}(f)\}$.
\end{thm}

Remark that in the setting of \cite{CZ15}, $X$ is assumed to be smooth, but the proof works when $X$ is normal.

\medskip
\begin{prop}\label{tri-rel-can-ellfib-isotrivial}
  Let $f\colon X \to C$ be an elliptic fibration from a normal projective surface to a smooth curve.
  If $K_{X/C} \simQ 0$, then there exists a finite flat morphism $\sigma\colon C_1\to C$ such that $X\times_C C_1 \cong C_1 \times F$, where $F$ is a closed fiber of $f$.
\end{prop}
\begin{proof}
  By doing a semi-stable reduction (see \cite[Proposition~10.2.33]{LiuAGAC}), we obtain the following commutative diagram
  \[
    \xymatrix{ 
      \widetilde Y \ar@/^8mm/[rrr]|{\,\pi\,}\ar[rrd]_g \ar[r]^{\mu}& Y \ar[r]^{\nu} \ar[rd] & X_{C'} \ar[r]^<<<{\tau'}\ar[d]^{f'} & X \ar[d]^f \\
                 &                                     & C' \ar[r]^{\tau}            & C \rlap{\,,}}
  \]
  where $C'$ is a smooth curve,
  $\tau$ is a finite morphism,
  $X_{C'} := X\times_C C'$ is the fiber product,
  $\nu$ is the normalization,
  $\mu$ is a minimal resolution,
  $\pi = \tau'\circ\nu\circ\mu$ is the composition, and
  $g$ is a semi-stable elliptic fibration.
  We have $K_{Y/C'} = (\tau'\circ\nu)^* K_{X/C} - \mathfrak C$ where $\mathfrak C$ is the conductor divisor of $\nu$, and $K_{\widetilde Y} = \mu^*K_Y - \widetilde E$ where $\widetilde E$ is an effective $\mu$-exceptional divisor.
It follows that \[
    K_{\widetilde Y/C'} = \pi^* K_{X/C} - \mu^*\mathfrak C - \widetilde E.
  \]
We have $K_{\widetilde Y/C'} \geQ 0$ by Theorem~\ref{rel-can-ellfib-sm}, in turn, combining with the assumption $K_{X/C} \simQ 0$ we can show that $K_{\widetilde Y/C'}\simQ0$ and $\mathfrak C =  \widetilde E = 0$.
It follows that  $Y$ has at most canonical singularities, and that $X_{C'}$ is normal by Lemma \ref{lem:flat-base-change-integral}, that is, $Y=X_{C'}$.
  Since $K_{\widetilde Y/C'}\simQ0$, we can apply \cite[Theorem~2.14]{CZ15} to obtain a finite morphism $C_1\to C'$ from a smooth curve $C_1$ such that $\widetilde Y \times_{C'} C_1\cong C_1 \times F$.
  In particular, all fibers of $g$ are irreducible, this implies that $\widetilde Y \to Y$ is an isomorphism. 
  We complete the proof by taking $\sigma$ to be the composition $C_1\to C'\to C$.
\end{proof}

 The following result can be derived directly from the proof of \cite[Theorem~4]{BM77}.
\begin{prop}[{cf.~\cite[Theorem~4]{BM77}}]\label{isotrivial-prod}
  Let $X$ be a quasi-projective normal variety which is equipped with a fibration $g\colon X\to Z$ and a morphism $f\colon X\to E$ to an elliptic curve such that $K(E)$ is algebraically closed in $K(X)$. Assume that all the closed fibers $C_z$ of $g$ are elliptic curves and that the induced morphisms $f|_{C_z}\colon C_z\to E$ are finite.
  Denote by $F$ a general fiber of $f$.
  Then there is an isogeny $\tau\colon E'\to E$ from an elliptic curve $E'$ such that $X \times_E E' \cong E' \times F$.
  Moreover, we have
  \[
    X \cong E'\times F/ G,
  \]
  where $G:=\ker\tau$ acts on $E' \times F$ diagonally: it acts on $E'$ by translation and on $F$ by some injective homomorphism of group functors $\alpha\colon G \to \Aut_F$.
\end{prop}


\subsection{Canonical bundle formula of relative dimension one}\label{sec:cbf-rel-one}
Let us recall some canonical bundle formulas for fibrations of relative dimension one.

\medskip
For a fibration $f\colon X\to S$ from $X$ to a variety of maximal Albanese dimension, we collect some useful results from \cite{CWZ23}.
\begin{thm}[{\cite[Theorem~1.3 and 7.3]{CWZ23}}]\label{thm:cb-formula}
  Let $(X,\Delta)$ be a normal $\QQ$-factorial projective pair.
  Let $f\colon X \to S$ be a fibration of relative dimension one, where $S$ is a normal variety of maximal Albanese dimension.
  Let $a_S\colon S \to A$ be the Albanese morphism of $S$.
  Assume that the pair $(X_{K(S)},\Delta_{K(S)})$ is klt and that $K_X +\Delta \simQ f^*D$ for some $\mathbb{Q}$-Cartier divisor $D$ on $S$.
  Then
  \begin{itemize}
    \item[\rm(1)] If $f$ or $a_S$ is separable, then $\kappa(X,K_X+\Delta) \geq \kappa(S)$.
    \item[\rm(2)] If $a_S$ is inseparable, then $\kappa(X,K_X+\Delta) \ge 0$; furthermore, if $\dim X = 3$, then $\kappa(X,K_X+\Delta) \geq 1$.
  \end{itemize}
\end{thm}

\begin{cor}[{\cite[Theorem~8.1]{CWZ23}}]\label{cor:fibration}
	Let $X$ be a normal projective $\mathbb Q$-factorial variety and $\Delta$ an effective $\mathbb{Q}$-divisor on $X$.
    Denote by $a_X\colon X\to A_X$ the Albanese morphism of $X$.
    Suppose that $-(K_X+\Delta)$ is nef, $X \to a_X(X)$ is of relative dimension one and $(X_{K(A)},\Delta_{K(A)})$ is klt. Then $a_X\colon  X \to A_X$ is a fibration.
\end{cor}

\subsection{Curves of small arithmetic genus}\label{sec:curve-pa1}
In this subsection, let $K$ denote an $F$-finite field of characteristic $p$. 
Let $X$ be a normal integral projective $K$-curve with $H^0(X, \mathcal O_X)=K$.
Let $p_a(X) =\dim_K H^1(X, \mathcal O_X)$ be the arithmetic genus of $X$.
Let $D =\sum_i a_i \mathfrak p_i$ be a Cartier divisor on $X$. Recall that the {\em degree} of $D$ is defined to be the integer \[
  \deg_K D= \sum_i a_i [\kappa(\mathfrak p_i):K],
\]
where $\kappa(\mathfrak p_i)$ denotes the residue field of $\mathfrak p_i$.

If $K_X \equiv 0$, applying Riemann-Roch formula (\cite[Theorem~7.3.17]{LiuAGAC}) we see that $K_X\sim 0$ and $p_a(X) = 1$.
Remark that when $p\geq 5$, such a curve $X$ is smooth over $K$ (\cite{Tate52}). When $p<5$, $X$ is possibly geometrically singular.
We will focus on the singular behavior of such curves under field extensions.

First, we recall the classification of curves with $p_a=0$.

\begin{prop}[{\cite[Theorem~9.10]{Tan21}}]\label{prop:ga0}
  Let $X$ be regular projective curve over $K$ with $H^0(X,\mathcal O_X) = K$ and $p_a(X) = 0$. Then
  \begin{itemize}
    \item[\rm(1)] $\deg_K K_X = -2$.
    \item[\rm(2)] $X$ is isomorphic to a conic in $\mathbb{P}^2_K$, and $X \cong \mathbb{P}^1_K$ if and only if it has a $K$-rational point.
    \item[\rm(3)] Either $X$ is a smooth conic or $X$ is geometrically non-reduced.
      In the latter case, we have $\mathop{\rm char} K =2$, and $X$ is isomorphic to the curve defined by a quadric $sx^2 + ty^2 + z^2 = 0$ for some $s,t\in K\setminus K^2$.
  \end{itemize}
\end{prop}

Next, we consider regular curves with $p_a = 1$. We collect some related results here and refer to \cite[Section~4]{CWZ23} for details.

\begin{prop}\label{prop:ga1-geo-reduced}
    Let $X$ be a regular projective curve over $K$ with $H^0(X,\mathcal O_X) = K$ and $p_a(X) = 1$. 
    Assume that $X$ is geometrically reduced and non-smooth. Then there exists a field extension $K \subset L \subseteq K^{1/p}$ such that $X _L$ is singular; and for any such $L$, the normalization $Y:= (X _ L)^\nu$ is a smooth curve of genus zero and the following statements hold.
  \begin{itemize}
    \item[\rm(1)]
      The non-smooth locus of $X$ is supported at a closed point $\mathfrak p$ and
      $\kappa(\mathfrak p)/K$ is purely inseparable of height one with $[\kappa(\mathfrak p):K]\le p^2$.
     In particular, there exists a unique point $\mathfrak q \in Y$ lying over $\mathfrak p$.
    \item[\rm(2)]
      If $p=3$, then $Y\cong\mathbb{P}^1_L$, $\pi^*\mathfrak p=3\mathfrak q$ (where $\pi\colon Y\to X$ is the induced morphism), $\mathfrak q$ is an $L$-rational point, and $[\kappa(\mathfrak p):K]=3$.
    \item[\rm(3)] Assume $p=2$.
      \begin{itemize}
      \item[\rm(a)] If the point $\mathfrak q$ is $L$-rational, then $Y\cong \mathbb{P}^1_L$, $\pi^*\mathfrak p=2\mathfrak q$ and $[\kappa(\mathfrak p):K]=2$.
      \item[\rm(b)] If the point $\mathfrak q$ is not $L$-rational, then $\deg_L \mathfrak q=2$, and
        \[
          \pi^*\mathfrak p=
          \begin{cases}
          \mathfrak q,  &\hbox{ if } \deg_K(\mathfrak p) = 2;\\
          2\mathfrak q, &\hbox{ if } \deg_K(\mathfrak p) = 4.
        \end{cases}\]
    \end{itemize}
  \end{itemize}
\end{prop}

\begin{prop}\label{prop:curve-nonreduced}
  Let $X$ be regular projective curve over $K$ with $H^0(X,\mathcal O_X) = K$ and $p_a(X) = 1$.
  Assume that $X$ is geometrically non-reduced. 
  Then
  \begin{itemize}
    \item[\rm(1)] There exists a height one field extension $K\subset L \subseteq K^{1/p}$ such that $X_L$ is integral but not normal,
      and for its normalization $Y$, we have $L \subsetneq K' := H^0(Y,\mathcal O_Y)\subseteq K^{\frac{1}{p}}$.
      Note that $Y \cong (X_{K'})_{\rm red}^\nu$ over $K'$.
      Denote by $\pi\colon Y\to X$ the induced morphism.
    \item[\rm(2)] We have $0\sim \pi^* K_X \sim K_Y + (p-1)C$ where $C$ is a Weil divisor such that $\deg_{K'} (p-1)C = 2$.
    \item[\rm(3)] The divisor $C$ is supported on either a single point $\mathfrak q\in Y$ or two points $\mathfrak q_1,\mathfrak q_2\in Y$ (this happens only when $p=2$).
      All the possibilities are listed as follows.
      \begin{itemize}
        \item[\rm(i)]
          $p=3$,  $X_L$ has a unique non-normal point, $ C= \mathfrak q$,  $Y\cong \mathbb{P}_{K'}^1$ and either $\pi^*\mathfrak p = \mathfrak q$ or $\pi^*\mathfrak p = 3\mathfrak q$.
        \item[\rm(ii)]
          $p=2$, and we fall into one of the following cases
          \begin{itemize}
            \item[\rm(a)] $ C= 2\mathfrak q$,  $\mathfrak q$ is a $K'$-rational point of $Y$ and $Y \cong \mathbb{P}_{K'}^1$;
            \item[\rm(b)] $ C= \mathfrak q_1+\mathfrak q_2$, and  $Y \cong \mathbb{P}_{K'}^1$;
            \item[\rm(c)] $ C= \mathfrak q$, $\kappa(\mathfrak q)/K'$ is an extension of degree two, and either
              \begin{itemize}
                \item[\rm(c1)] $Y\subset\mathbb{P}^2_{K'}$ is a smooth conic (possibly $\mathbb P^1_{K'}$), or
                \item[\rm(c2)] $Y$ is isomorphic to the curve defined by $sx^2 + ty^2 + z^2 = 0$ for some $s,t\in K'\setminus K'^2$ such that $[K'^2(s,t):K'^2]=4$.
              \end{itemize}
          \end{itemize}
        \end{itemize}
  \end{itemize}
\end{prop}

\begin{rmk}\label{rmk:pdeg2}%
  We note that if $\mathop{\rm pdeg} K := [K^{1/p}:K] = 2$, then {\it the case {\rm(c2)} in Proposition~\ref{prop:curve-nonreduced} does not occur}.
  This follows from Schröer's classification of regular genus-one curves \cite{Sch22}.
  Indeed, by \cite[Theorem~2.3]{Sch22}, the $p$-degree of $K$ is at least $r+1$, where $r$ is the geometric generic embedding dimension (defined as the embedding dimension of the local Artin ring $\mathcal O_{X,\eta} \otimes_K K^{\rm perf}$, see \cite[Section~1]{FS20}).
  If $Y = (X\times_K K^{1/p})^\nu_{\rm red}$ is not smooth, then by \cite[Theorem~2.3]{Sch22} the second Frobenius base-change $X\times_K K^{1/p^2}$ is isomorphic to some standard model $C^{(i)}_{r,F,\Lambda}$ with $i = 2$ (see \cite{Sch22} for the definition of $C^{(i)}_{r,F,\Lambda}$).
  Note that $i\le r$ by Schröer's construction (\cite[page~8]{Sch22}).
  Thus, the case (c2) occurs only when $\mathop{\rm pdeg} K \ge r + 1 \ge 3$.
\end{rmk}

\section{Constructions of semi-ample divisors}
This section concerns the construction of semi-ample divisors.
Let $X$ be a $K$-trivial variety and denote by $a_X\colon X\to A$ the Albanese morphism of $X$.
In case $\dim a_X(X) = 1$, we construct a semi-ample divisor which is relatively ample over $A$; and in case $a_X$ has relative dimension one, we prove a semi-ampleness criterion for divisors with numerical dimension one.
We use these results to derive the second fibration $g\colon X\to Z$ which is transversal to $a_X$.

\subsection{A construction of semi-ample divisors}
The construction we present here was recently developed in \cite{EP23, Eji23} by use of a powerful positivity engine.  Similar approaches were used to treat surfaces (\cite[Theorem 8.10]{BadescuAS}) and threefolds equipped with fibrations to elliptic curves  (\cite[Section 4.4]{Zh20}).
In the settings of \cite{Eji23, EP23}, a pair $(X,\Delta)$ is assumed to be strongly $F$-regular and the Cartier index of $K_X+\Delta$ is assumed to be indivisible by $p$ (i.e., $K_X+\Delta$ is a $\mathbb{Z}_{(p)}$-Cartier). By refining their argument, we can slightly relax the $F$-regularity condition and replace the second condition by that $X$ is $\QQ$-factorial.  Here, we only explain how to modify the argument.

First, we recall a positivity result due to Ejiri, which is crucial for proving Theorem~\ref{thm:pos-EP}, one of the main results of this subsection.
\begin{thm}[{\cite[Corollary~6.5, Examples~5.7, 5.8]{Eji24}}]\label{thm:pos-Eji}
  Let $f\colon X\to Y$ be a surjective morphism between normal projective varieties and let $a\colon Y \to A$ be a generically finite morphism to an abelian variety. Let $\Delta \geq 0$ be a $\mathbb{Z}_{(p)}$-divisor on $X$. Let $V'\subseteq A$ be an open dense subset, $V= a^{-1}V'$ and $U = f^{-1}V$. Assume that
  \begin{itemize}
    \item $(U, \Delta|_U)$ is $F$-pure;
    \item $K_U + \Delta|_U$ is $\mathbb{Q}$-Cartier and  $f$-ample; and
    \item $V\to V'$ is a finite morphism.
  \end{itemize}
  Let $H'$ be a free ample symmetric divisor on $A$ and $H= a^*H'$. Let $i$ be a positive integer such that $i(K_X + \Delta)$ is integral.
  Then there exists an integer $m_0$ such that for any $m\geq m_0$,
  \begin{itemize}
    \item[\rm(i)] $f_*\mathcal{O}_X(im(K_X + \Delta))\otimes \mathcal{O}_Y((\dim Y + 1)H)$ is globally generated over $V$; and
    \item[\rm(ii)] $\mathbb{B}_-(f_*\mathcal{O}_X(im(K_X + \Delta))) \subseteq Y\setminus V$ {\em (\em see \em\cite[Section~4.1]{Eji24} \em for the definition of $\mathbb B_-$\em)}.
  \end{itemize}
\end{thm}

We adapt \cite[Theorem~7.1]{EP23} to the following setting, which originally requires that $K_X + \Delta$ is a $\mathbb{Z}_{(p)}$-Cartier divisor and that $(X, \Delta)$ is $F$-pure.
\begin{cor}\label{cor:pos-nef}
  Let $X$ be a normal $\mathbb{Q}$-factorial projective variety of dimension $n$ and $f\colon X \to Y$ a fibration to a smooth curve $Y$ with $g(Y) \geq 1$. Let $\Delta\geq 0$ be a divisor on $X$ and $L$ be a nef Cartier divisor on $X$.
  Assume that $(X_{\eta}, \Delta_{\eta})$ is strongly $F$-regular and that $K_X + \Delta + L$ is $f$-nef. Then $K_X + \Delta + L$ is nef.
\end{cor}
\begin{proof}
We only need to prove that for every ample $\mathbb{Q}$-divisor $A$ on $X$, $K_X + \Delta + L + A$ is nef. Fix an ample divisor $A$ on $X$. By \cite[Lemma 3.15]{Pat14}, we can find an ample $\mathbb{Q}$-divisor $A'> 0$, such that
\begin{itemize}
  \item[\rm(i)] $A-A'$ is ample,
  \item[\rm(ii)] $K_X + \Delta'$ is a $\mathbb{Z}_{(p)}$-Cartier divisor, where $\Delta'= \Delta + A'$, and
  \item[\rm(iii)] $(X_{\eta}, \Delta'_{\eta})$ is strongly $F$-regular.
\end{itemize}
Here, we remark that in order to apply Theorem \ref{thm:pos-Eji}, we only require the pair $(X, \Delta)$ to be $F$-pure over an open subset of $Y$, which is guaranteed by condition (iii). Therefore, since $K_X + \Delta' + L$ is $f$-nef, we can apply the same proof of \cite[Theorem~7.1]{EP23} to the pair $(X,\Delta')$ and $L$, which yields that $K_X + \Delta' + L\simQ K_X + \Delta + L + A'$ is nef. From this, we conclude that $K_X + \Delta + L$ is nef.
\end{proof}

\begin{thm}\label{thm:pos-EP}
Let $X$ be a normal $\mathbb{Q}$-factorial projective variety of dimension $n$ and $f\colon X \to Y$ a fibration to a smooth curve $Y$ with $g(Y) \geq 1$. Let $\Delta\geq 0$ be a divisor on $X$. Assume that
\begin{itemize}
  \item[\rm a)] $-(K_{X}+\Delta)$ is nef;
  \item[\rm b)] the Cartier index of $K_{X_{\eta}}+\Delta_{\eta}$ is indivisible by $p$; and
  \item[\rm c)] $(X_{\eta}, \Delta_{\eta})$ is strongly $F$-regular.
\end{itemize}
Let $A$ be an ample divisor on $X$. Choose positive integers $a,b$ such that $(aA - bF)^n = a^nA^n - na^{n-1}bA^{n-1}F = 0$, where $F$ is a fiber of $f$.

Then
\begin{itemize}
  \item[\rm(i)] The divisor $D=aA - bF$ is nef with $v(D) = n-1$.
  \item[\rm(ii)] $g(Y)=1$.
  \item[\rm(iii)] For any sufficiently divisible integer $m>0$, the sheaf $f_*(\mathcal O_X(mD))$ is a numerically flat vector bundle, that is, both $f_*(\mathcal O_X(mD))$ and its dual $(f_*(\mathcal O_X(mD)))^\vee$ are nef. By \cite{Oda71}, there exists an isogeny $\pi\colon Z \to Y$ from an elliptic curve $Z$ such that $\pi^*f_*\mathcal{O}_X(mD) \cong \bigoplus_i L_i$, where $L_i \in \mathrm{Pic}^0(Z)$.
  \item[\rm(iv)] Assume moreover that
    \begin{itemize}
      \item[\rm d)] $(X, \Delta)$ is strongly $F$-regular, the Cartier index of $K_{X}+\Delta$ is indivisible by $p$ and $-(K_X + \Delta)$ is numerically semi-ample.
    \end{itemize}
Then there exists some $L \in \mathrm{Pic}^0 (Y )$ such that $D + f ^* L$ is semi-ample.
\end{itemize}
\end{thm}
\begin{proof}
First, we borrow the argument of \cite[Theorem~7.3]{EP23} to show that the divisor $D$ is nef as follows.
Take an ample $\mathbb{Q}$-divisor $H$ on $Y$. Let us show that $D+ f^*H$ is nef. By the construction of $D$, the divisor $D+ f^*H$ is big and $f$-ample. Take an effective divisor $ \Gamma \simQ D+ f^*H$ and a small rational number $\epsilon >0$ such that $(X, \Delta'= \Delta + \epsilon \Gamma)$ is strongly $F$-regular on $X_{\eta}$. Since $-(K_X +  \Delta)$ is nef, Corollary \ref{cor:pos-nef} applies and shows that
$$\epsilon (D+ f^*H)  \equiv K_X +  \Delta' + (-(K_X +  \Delta)),$$
is a nef divisor. Therefore, we conclude that $D$ is nef, and hence $\nu(D) = n-1$ by $D^n=0$. This proves (i).

Next, we prove (ii): $g(Y) = 1$. Otherwise, $H = K_Y$ is big. Since $D$ is nef and $f$-ample and $D-(K_{X/Y}+\Delta) -f^*K_Y = D-(K_{X}+\Delta) $ is nef, we can apply \cite[Theorem~1.5]{Zha19a} to show that $D$ is big, which contradicts $\nu(D) = n-1$.

Having proved $\nu(D) = n-1$ and $g(Y)=1$, by the same argument of \cite[Theorem~7.4 and Theorem~7.5]{EP23} we can show that $f_*\mathcal{O}_X(mD)$ is numerically flat. This proves (iii).

Finally, we can conclude the assertion (iv) from the proof of \cite[Theorem~6.1]{Eji23} under these conditions.
\end{proof}

\begin{rmk}
For the assertion (iv) in the theorem above, it seems not easy to drop the Cartier index condition in the proof of \cite[Theorem~6.1]{Eji23}.
\end{rmk}

When $\mathop{\rm char}k\geq 5$, we may run the minimal model program for $3$-dimensional klt pairs (\cite{HX15,Bir16,HW22}).   Taking advantage of this, the arguments of \cite[Section~4.4]{Zh20} yields the following theorem.
\begin{thm}\label{thm:3fold-rdim2-semiample}
  Assume $p=\mathop{\rm char}k\geq 5$. Let $(X, \Delta)$ be a normal projective $\mathbb{Q}$-factorial three-dimensional klt pair with $K_X + \Delta \equiv 0$. Let $f\colon X\to Y$ be a fibration to a curve with $g(Y) \geq 1$.
  Let $A$ be an ample divisor on $X$ and choose positive integers $a,b$ such that $(aA - bF)^3 = 0$, where $F$ is a fiber of $f$. 
  Then there exists some $L \in \mathrm{Pic}^0 (Y )$ such that $aA - bF+ f ^* L$ is semi-ample.
\end{thm}

\subsection{A semi-ampleness criterion}
If $X$ is a $K$-trivial variety equipped with a quasi-elliptic fibration $f\colon X\to A$ to an abelian variety, the non-smooth locus of $f$ provides a divisor. To prove this divisor is semi-ample, we abstract a semi-ampleness criterion from the argument of \cite[Section~4.2]{Zh20}, which proves a nonvanishing result up to a twist by a numerically trivial line bundle. In \cite[Section~4.2]{Zh20}, the author treated only threefolds and used results of minimal model program. But our situation is special, we can avoid running MMP in the argument.

We first prove the following nonvanishing result by use of \cite[Theorem~3.7]{Zh20} and a similar argument of \cite[Section~4.2]{Zh20}.

\begin{thm}\label{thm:non-van}
  Let $X$ be a normal projective variety equipped with a surjective morphism  $f\colon X\to A$ to an abelian variety $A$ of dimension $d$. Let $\Delta \geq 0$ be an effective $\QQ$-divisor on $X$ and $D$ a Cartier divisor on $X$. Assume that
\begin{itemize}
  \item[\rm(a)] $K_X+\Delta$ is a $\mathbb{Q}$-Cartier $\mathbb{Q}$-divisor with the Weil index of $K_X+\Delta$ being indivisible by $p$;
  \item[\rm(b)] the Cartier index of $(K_X+\Delta)|_{X_{\eta}}$  is indivisible by $p$;
  \item[\rm(c)] the divisor $D-(K_X+\Delta)$ is nef and relatively ample over $A$;
  \item[\rm(d)] $r=\dim_{K(\eta)}S_{\Delta}^0(X_{\eta}, D|_{X_{\eta}})> 0$ {\em(\em see \em\cite[page~10]{Zha19a} \em for the definition of $S_{\Delta}^0$\em)}.
\end{itemize}
Then
\begin{itemize}
  \item[\rm(1)] $V^0(f_*\mathcal{O}_X(D)) = \{\alpha \in \widehat{A}=\mathrm{Pic}^0(A)\mid h^0(f_*\mathcal{O}_X(D)\otimes \mathcal{P}_{\alpha}) > 0\} \neq \emptyset$, where $\mathcal{P}$ denotes the Poincar\'{e} line bundle over $A\times \widehat{A}$; and
  \item[\rm(2)] if $\dim V^0(f_*\mathcal{O}_X(D)) =0$, then there exist a subsheaf $\mathcal{F} \subseteq f_*\mathcal{O}_X(D)$ of rank $r$ such that $\mathcal{F}|_{X_{\eta}}=S_{\Delta}^0(X_{\eta}, D|_{X_{\eta}})$, an isogeny $\pi\colon A_1\to A$ of abelian varieties, some $P_1, \ldots, P_r \in \mathrm{Pic}^0(A_1)$ and a generically surjective homomorphism
$$\beta\colon  \bigoplus_iP_i \to \pi^*\mathcal{F}.$$
\end{itemize}
\end{thm}
\begin{proof}
By the assumption (a), we may fix an integer $g>0$ such that $(p^g - 1)\Delta$ is integral.
For an integer $e>0$ divisible by $g$, setting $D_e=(1-p^e)(K_X + \Delta)+ p^eD$, we have the trace map
$$\mathrm{Tr}^{e}_D: \mathcal{F}^{e}:=f_*(F_{X*}^{e}\mathcal O_X((1-p^e)(K_X + \Delta)) \otimes \mathcal O_X(D)) \cong f_*(F_{X*}^{e}\mathcal O_X(D_e)) \to f_*\mathcal O_X(D),$$
and we denote its image by $\mathcal{F}_0^e$.
Note that for positive integers $e'>e$ divisible by $g$, the trace map $\mathrm{Tr}^{e'}_D$ factors through the trace map (\cite[Section~2.7]{Zh20})
$$\mathrm{Tr}^{e'-e}_{D_e}: \mathcal{F}^{e'}=f_*(F_{X*}^{e'-e}\mathcal O_X((D_{e'})))   \to \mathcal{F}^{e}=f_*\mathcal O_X(D_e),$$
which implies that $\mathcal{F}^{e'}_0\subseteq \mathcal{F}^{e}_0$.
Therefore, there exists a positive integer $e_0$ such that for all $e \geq e_0$ divisible by $g$, the sheaf $\mathcal{F}_0^e$ has constant rank $r$, which means that
\begin{itemize}
\item[(c$'$)]\emph{the trace map
$\mathrm{Tr}^{e}_D\colon \mathcal{F}^{e} \to \mathcal{F} :=\mathcal{F}^{e_{0}}_0 $ is generically surjective.}
\end{itemize}

We shall apply \cite[Theorem~3.7]{Zh20}, and we need to verify the three conditions required there. We refer to  \cite[Section~3]{Zh20} for the related notions of Fourier-Mukai transform.

With the assumption that $D-(K_X+\Delta)$ is nef and relatively ample over $A$, the proof of the vanishing condition (C2) in \cite[Section~4.2]{Zh20}, if substituting the pair $(X,B)$ with $(X,\Delta)$ and the divisor $l(K_X+B)$ with $D$,  still works and yields the following:
\begin{quote}\it
  {\sc Claim}. If $H$ is an ample line bundle on $\widehat{A}$, then for any $i>0$ and sufficiently divisible integer $e>0$,
  $$H^i(A, \mathcal{F}^{e} \otimes \widehat{H}^*) = 0.$$
\end{quote}
By induction, we can find integers $e_0< e_1 < e_2 < \cdots <e_{d}$ divisible by $g$, and ample line bundles $H_0, H_1, \ldots, H_{d-1}$ on $\widehat{A}$ such that, if setting
$$  \mathcal{F}_0= \mathcal{F},\,
    \mathcal{F}_1=\mathcal{F}^{e_{1}},\,
    \mathcal{F}_2=\mathcal{F}^{e_{2}},\ldots,\, \mathcal{F}_d=\mathcal{F}^{e_{d}},$$
    we have
\begin{itemize}
\item[(a$'$)]
  \emph{for $0\leq l \leq d-1$ and every $i$, the sheaf $R^i\Phi_{\mathcal{P}}D_A(\mathcal{F}_{l}) \otimes H_l$ is globally generated, and if $j>0$ then $H^j(\widehat{A}, R^i\Phi_{\mathcal{P}}D_A(\mathcal{F}_{l}) \otimes H_l) = 0$; and}
\item[(b$'$)]
\emph{for $0\leq l < l' \leq d$, if $j>0$ then $H^j(A, \mathcal{F}_{l'}\otimes \widehat{H}_l^*) = 0$.}
\end{itemize}
We see that the conditions  (a,\,b,\,c) of \cite[Theorem~3.7]{Zh20} are guaranteed by the above conditions (a$'$,\,b$'$,\,c$'$) respectively. Therefore, we can apply \cite[Theorem~3.7]{Zh20}: 
\begin{itemize}
\item 
by \cite[Theorem~3.7 (i)]{Zh20} the homomorphism $\alpha_{\mathcal{F}}\colon  \mathcal{F}^{*} \to (-1)_A^*R^{0}\Psi_{\mathcal{P}}R^{0}\Phi_{\mathcal{P}}D_A(\mathcal{F})$ is injective, and in turn applying \cite[Proposition~3.4~(2)]{Zh20} we can show that
$$V^0(f_*\mathcal{O}_X(D)) =\mathrm{Supp}(-1)_{\widehat{A}}^*R^{0}\Phi_{\mathcal{P}}D_A(\mathcal{F}) \neq \emptyset;$$
\item
by \cite[Theorem~3.7 (ii)]{Zh20}, the second assertion of the theorem follows.
\end{itemize}
\end{proof}

Next, we prove the following semi-ample criterion.
\begin{thm}\label{thm:hor-map-1}
Let $X$ be a normal $\mathbb{Q}$-factorial projective variety equipped with a fibration $f\colon X \to A$ of relative dimension one onto an abelian variety. Assume that
\begin{itemize}
  \item[\rm(a)] $K_X \simQ 0$,
  \item[\rm(b)] there exists an irreducible $f$-horizontal divisor $B$ on $X$ such that $B|_{B^\nu} \equiv 0$.
\end{itemize}
Then $B$ is semi-ample, and the associated morphism $g\colon X\to C$ is a fibration to a curve.
\end{thm}
\begin{proof}
Since $B|_{B^\nu} \equiv 0$, $B$ is a nef divisor with numerical dimension $\nu(B) =1$.
To show the semi-ampleness of $B$, it suffices to find an integer $l>0$ and a numerically trivial line bundle $L$ such that $h^0(X, lB +L) >1$.
Indeed, granted this, by Lemma~\ref{lem:nu=1} the divisor $lB +L$ is semi-ample and induces a fibration $g\colon X\to C$ to a curve, then since $lB\equiv lB+L$ and $B$ is irreducible by assumption, we see that a multiple of $B$ coincides with a fiber of $g$, and hence $B$ is semi-ample.

\medskip
Since $(K_X+B)|_{B^\nu} \equiv 0$, by Lemma~\ref{lem:adj-charA}, $B$ is isomorphic to an abelian variety and thus $B\to A$ is a finite morphism.
Note that since $X$ has relative dimension one over $A$, the divisor $B$ is relatively big over $A$.
Let $\mathbb E(B)$ denote the relatively exceptional locus with respect to $B$, namely, the union of $f$-exceptional irreducible varieties on which the restriction of $B$ is not $f$-big.

\begin{quote}\it
  {\sc Claim}: The intersection $\mathbb E(B) \cap B = \emptyset$, and $B$ is relatively semi-ample over $A$. 
\end{quote}
{\it Proof of the claim.}
Since $B$ is finite over $A$, if there is an irreducible component $Z$ of $\mathbb E(B)$ intersecting $B$, then $Z\cap B$ is also finite over $f(Z)$, hence $B|_Z$ is big over $f(Z)$, a contradiction.
Therefore, we conclude that $\mathbb E(B)$ does not intersect $B$, and hence $B$ is $f$-semi-ample by \cite[Theorem~0.2]{Keel99}.
\qed\medskip

The $f$-relative semi-ample divisor $B$ induces a birational contraction $\sigma\colon X\to Y$, which is isomorphic both near $B$ and on the generic fiber $X_{\eta}$ of $f$, and there exists a $\mathbb{Q}$-Cartier divisor $B_Y$ on $Y$ such that $B=\sigma^*B_Y$.
\medskip

Since $B|_{X_{\eta}}$ is an ample divisor on $X_{\eta}$, we may take a sufficiently divisible integer $l>0$ such that $lB$ is Cartier and $r:=\dim_{K(\eta)}S^0(X_{\eta}, lB) > 1$. Set $\mathcal L = \mathcal{O}_X(lB)$.
By construction, we have  $B_Y$ is  relatively ample over $A$, and $K_Y \simQ 0$. If setting $D =lB_Y$ then $D-K_Y$ is relatively ample over $A$. Therefore, Theorem~\ref{thm:non-van} applies to $Y$ and $D =lB_Y$ and yields that $V^0(f_*\mathcal{L}) \neq \emptyset$.
\smallskip

If $\dim V^0(f_*\mathcal{L}) >0$, then we can apply the argument Step~1 of the proof of \cite[Theorem~4.2]{Zh20} to show that $\kappa(X, B+ f^*L)\geq 1$ for some $L \in \mathrm{Pic}^0(A)$, which is sufficient to conclude the proof.
\smallskip

Now, assume $\dim V^0(f_*\mathcal{L}) =0$. Then by Theorem~\ref{thm:non-van}~(2)
there exist a subsheaf $\mathcal{F} \subseteq f_*\mathcal{L}$ of rank $r$, an isogeny $\pi\colon A_1\to A$ of abelian varieties, some $P_1, \ldots, P_r \in \mathrm{Pic}^0(A_1)$ and a generically surjective homomorphism
$$\beta\colon  \bigoplus_iP_i \to \pi^*\mathcal{F}.$$
Applying the covering theorem as in Step 2 of the proof of \cite[Theorem~4.2]{Zh20}, we show that there exist an integer $m>0$ and some $L_1, \ldots, L_r \in \Pic^0(A)$ such that
\begin{itemize}
  \item $H^0(X, mB + f^*L_i) \neq 0$;
  \item the sub-linear system of $|(mB)_{K(\eta)}|$ corresponding to the subspace $\sum_iH^0(X, mB + L_i)\otimes_k K(\eta) \subseteq H^0(X_{\eta}, mB|_{X_{\eta}})$ defines a non-trivial map.
\end{itemize}
We may assume each $h^0(X, mB + f^*L_i)=1$, thus there exists a unique effective divisor $D_i \sim mB + f^*L_i$.


  Write $D_i = aB + D'$ such that $B\not\subseteq\Supp D'$. We have $D_i|_{B}\equiv 0$. By $\nu(B)=1$, we conclude that $\Supp B\cap \Supp D' = \emptyset$, thus $D'|_{B} \sim 0$.
  Moreover, by Lemma~\ref{lem:adj-charA}, $B|_{B} \simQ (K_X + B)|_{B} \geQ 0$. Then it follows that $D_i|_{B} \simQ 0$.

  Take $D_1 \neq D_2$. By $D_1 - D_2 \sim f^*(L_1-L_2)$, we conclude that $f^*(L_1-L_2) |_{B} \simQ 0$.
Since $B \to A$ is dominant, we have $L_1 \simQ L_2$, that is, there exists some $N>0$ such that $NL_1\sim NL_2$. But then $ND_1 \sim ND_2$. This shows $h^0(X, NmB+ NL_1) \geq 2$, which concludes the proof.
\end{proof}

\section{Structure theorems of $K$-trivial irregular varieties with $\dim a_X(X) = 1$}\label{sec:str-thms-albdim=1}
In this section, we treat $K$-trivial irregular varieties $X$ with $\dim a_X(X) = 1$. The main result is the following.
\begin{thm}\label{thm:q1}
  Let $X$ be a normal $\QQ$-factorial projective variety with $K_X\equiv 0$ and $\dim a_X(X)=1$, where $a_X\colon X\to A$ is the Albanese morphism of $X$.
  Assume moreover that, either 
  \begin{itemize}
    \item[\rm(a)] $X$ is strongly $F$-regular and the Cartier index of $K_X$ is indivisible by $p$, or 
    \item[\rm(b)] $\dim X = 3$, $p\ge5$, and $X$ has at worst terminal singularities.
  \end{itemize}
  Then 
  \begin{itemize}
    \item[\rm(1)] $q(X) = 1$, thus $E:=A$ is an elliptic curve and $f=a_X\colon X\to E$ is a fibration.
    \item[\rm(2)] there exists an isogeny $\pi\colon \bar E \to E$ of elliptic curves such that $X\times_E \bar E \cong \bar E\times F$, where $F$ is a fiber of $f$. 
      More precisely, there is a faithful action of $H= \mathrm{ker} (\pi)$ on $F$, such that \[
        X\cong \bar E\times F/H,
      \]
      where $H$ acts diagonally on $\bar E\times F$.
  \end{itemize}
\end{thm}

\begin{proof}
   Note that under the assumption (b), the generic fiber of $f$ is a regular surface. So in both cases (a) and (b), we can apply Theorem~\ref{thm:pos-EP} (i) to show the assertion (1). 
\smallskip

Next we claim that there exists a divisor $D$ such that
  \begin{itemize}
    \item[$(*)$] for sufficiently divisible $m>0$, $f_*\mathcal{O}_X(mD)$ is a numerically flat vector bundle on $E$, and there exists an isogeny $\tau_1\colon E_1 \to E$ such that $\tau_1^*f_*\mathcal{O}_X(mD) \cong \bigoplus_{i=1}^{r}L_i$ for some $L_i \in \mathrm{Pic}^0(E_1)$; and
    \item[$(**)$] $D$ is semi-ample and $f$-ample, and $\nu(D) = \dim X-1$.
  \end{itemize}
Such a divisor $D$ satisfying the condition $(*)$ exists by Theorem~\ref{thm:pos-EP}, and the condition $(**)$ is guaranteed by Theorem~\ref{thm:pos-EP} (iv) in case (a) and by Theorem~\ref{thm:3fold-rdim2-semiample} in case (b).

Now let $g\colon X\to Z$ be the fibration associated with $D$, where $Z$ is a normal projective variety with $\dim Z=\dim X-1$. To summarize, we obtain the following ``bi-fibration'' structure:
$$\xymatrix@R=4ex@C=4ex{&X\ar[d]_{f}\ar[r]^{g} & Z.\\
&E &
}$$
Here, for a general closed point $z\in Z$, the fiber $X_z$ of $g$ is a curve which is finite and dominant over $E$. Since  $K_{X_z}\equiv 0$, $X_z$ is in fact an elliptic curve.
Moreover, since $D$ is $f$-ample, we conclude that $g\colon X\to Z$ is equidimensional and every component of a fiber $X_z$ of $g$ over an arbitrary closed point $z\in Z$ is dominant over $E$.

In the following, we fix a sufficiently large $m$ such that for any $l>0$, the natural homomorphism $\eta\colon S^l(f_*\mathcal{O}_X(mD)) \to f_*\mathcal{O}_X(lmD)$ is surjective, where $S^l$ denotes  the $l$-th symmetric power.

\smallskip
Following the approach of \cite[Proposition 7.6]{EP23}, we can prove the following result.
\medskip

{\sc Lemma.}
  There exists an isogeny $\tau \colon \tilde{E} \to E$ of elliptic curves such that $\tilde{X}:=X\times_E\tilde{E} \cong \tilde{E} \times  F$.
\medskip

{\it Proof of the lemma.}
First we prove the following claim:
\begin{quote}\it
  {\sc Claim}. There exists an isogeny $\tau \colon \tilde{E} \to E$ of elliptic curves such that $\tau^*f_*\mathcal{O}_X(mD) \cong \bigoplus^r\mathcal{O}_{\tilde{E}}$, where $r:=\mathop{\rm rank}f_*\mathcal O_X(mD)$.
\end{quote}

{\it Proof of the claim.}
  By $(*)$ there exists an isogeny $\tau_1\colon E_1 \to E$, such that $\tau_1^*f_*\mathcal{O}_X(mD) \cong \bigoplus_{i=1}^{r}L_i$ for some $L_i \in \mathrm{Pic}^0(E_1)$.
  Let $X_1= (X\times_EE_1)^\nu$. Denote by $\pi_1\colon X_1 \to X$ and $f_1\colon  X_1\to E_1$ the natural projections. Then $|\pi_1^*mD - f_1^*L_i| \neq \emptyset$. The divisor $\pi_1^*D$, being semi-ample with $\nu(\pi_1^*D) =n-1$, induces a fibration  $g_1\colon X_1 \to Z_1$. Then for a general fiber $C_1$ of $g_1$, the linear system $|\pi_1^*mD - f_1^*L_i|_{C_1}   \neq \emptyset$. We conclude that $f_1^*L_i|_{C_1} \sim 0$, thus each $L_i$ is a torsion point in $\mathrm{Pic}^0(E_1)$. There exists a further isogeny of elliptic curves $\tau_2\colon  E_2 \to E_1$ such that $\tau_2^*L_i \cong \mathcal{O}_{E_2}$. We may take $\tau\colon  \tilde{E} = E_2 \to E_1 \to E$  to complete the proof of the claim.
\qed\smallskip

Denote by $\tilde{\pi}\colon \tilde{X}:=X\times_E\tilde{E}\to X$ and $\tilde{f}\colon\tilde{X}\to \tilde{E}$ the natural projections. Since for each $l\geq 1$, $\eta\colon S^l(f_*\mathcal{O}_X(mD)) \to f_*\mathcal{O}_X(lmD)$ is surjective, the homomorphism
$$\varphi_l: S^l(\tilde{f}_*\mathcal{O}_{\tilde{X}}(m\tilde{\pi}^*D)) \cong \tau^*S^l(f_*\mathcal{O}_X(mD)) \cong  \bigoplus \mathcal{O}_{\tilde{E}}  \to  \tilde{f}_*\mathcal{O}_{\tilde{X}}(lm\tilde{\pi}^*D) \cong  \tau^*f_*\mathcal{O}_X(lmD)$$
is surjective. Combining this with the numerically flatness of $f_*\mathcal{O}_X(lmD)$, $l\ge1$, we conclude that $\tilde{f}_*\mathcal{O}_{\tilde{X}}(lm\tilde{\pi}^*D) \cong  \bigoplus \mathcal{O}_{\tilde{E}}$, and $\varphi_l$ is determined by the corresponding homomorphism of the global sections. From this, we conclude that $\tilde{X}=X\times_E \tilde{E}\cong \tilde{E}\times F$.
\qed\smallskip 

We regard the morphism $\tau\colon \tilde{E} \to E$ as a morphism of abelian varieties and write $\tilde{H} = \mathrm{ker}(\tau)$. The natural action of $\tilde{H}$ on $\tilde{E}$ induces an action of $\tilde{H}$ on the base change $\tilde{X}=X\times_E \tilde{E}$. To summarize, we have a commutative diagram:
\begin{equation}\label{eq:VHBR}
  \vcenter{\xymatrix@R=2ex@C=3.5ex{
   &&& F \ar[rr] && Z \\
   \tilde{E}\times F  &&\tilde{X} \ar[ll]^<<<<<<\cong_<<<<<<{\varphi=(\tilde f, \tilde g)} \ar[ru]^{\tilde g} \ar[rr]^<<<<<<<{\tilde{\pi}}\ar[dd]^{\tilde{f}} && X\ar[dd]^f\ar[ru]^<<g &\\
   \\
   &&\tilde{E}\ar[rr]^{\tau} &&E\rlap{\,.}  \\
}}\end{equation}

\medskip
Through the isomorphism $\varphi\colon\tilde{X}\buildrel\sim\over\to \tilde{E}\times F$, $\tilde{H}$ acts on $\tilde{E}\times F$.
Our next step is to show that this action of $\tilde H$ on $\tilde E\times F$ is diagonal.
For this purpose, we consider the second fibration $g\colon X\to Z$.
As observed before, for a general closed point $z\in Z$, the fiber $X_z$ of $g$ is an elliptic curve, so $X_z\to E$ is an isogenous of a fixed degree. 
Let $Z^\circ \subset Z$ be an open subset such that $Z^\circ$ is regular and for each closed point $z\in Z^\circ$, the fiber $X_z$ of $g$ is an elliptic curve.
Set $X^\circ := g^{-1}(Z^\circ)$ and $F^\circ := F\cap X^\circ$.
By Proposition~\ref{isotrivial-prod}, we conclude that:
\begin{itemize}
  \item[$(*{*}*)$] There exists an isogeny $\tau'\colon E'\to E$ from an elliptic curve $E'$ such that $X^\circ \times_E E' \cong E' \times F^\circ$. Moreover the induced action of $G:=\ker\sigma$ on $X^\circ \times_E E'$ is diagonal on $E' \times F^\circ$, and the following commutative diagram commutes
$$\xymatrix@R=2.5ex@C=3ex{
   &&&  F^\circ \ar[rr]&& Z^{\circ}\cong  F^\circ/G \\
  E'\times F^\circ\ar@/^1pc/[rrru]^{p_2}  && X^{\circ}_1:=X^\circ\times_E E'\ar[ll]_<<<<<\cong \ar[ru]^{g_1} \ar[rr]^>>>>>>>>{\pi_1}\ar[dd]^{f_1} && X^{\circ}\ar[dd]^f\ar[ru]^<<<g \\
   \\
   &&E'\ar[rr]^{\tau'} &&E \rlap{\,.} \\
}$$
\end{itemize}

Remark that if we take a further isogeny $\tau''\colon E'' \to E'$ of elliptic curves, then the action of $\ker\tau''$ on the base change $X^{\circ}_1\times_{E'}E''\cong E''\times F^{\circ}$ induces an action on the product $E''\times F^{\circ}$; precisely, $\ker\tau''$ acts on $E''$ by translation and on $F^{\circ}$ trivially. Thus the base change $\tau'\circ \tau''\colon E'' \to E' \to E$ induced a diagonal action of $\ker\tau'\circ \tau''$ on $E''\times F^{\circ}$.
Therefore, we may choose an isogeny $\tilde{\tau}'\colon \tilde{E}'\to E$ which factors through both $\tau\colon \tilde E\to E$ and $\tau'\colon E'\to E$, and obtain the
following commutative diagram 
\[
  \xymatrix@M=4pt@C6pt@R7pt{
    & F^\circ \ar[rr]\ar@{_{(}-->}[dd]|>>>>>>\hole && F^\circ/G \ar[dd] \\
    X^\circ \times_E \tilde E' \ar@{_{(}->}[dd]\ar[rr]\ar[ur]^{\tilde g_0'} && X^\circ \ar@{_{(}->}[dd]\ar[ur]^{g_0} \\
    & F \ar@{-->}[rr]|<<<<<<<\hole && Z \\
    X\times_E \tilde E' \ar[rr]\ar[dd]\ar@{-->}[ur]^{\tilde g'} && X \ar[dd]\ar[ur]^g \\
    \\
    \tilde E' \ar[rr]^{\tilde{\tau}'} && E \rlap{\,.}
  }
\]
Thus the natural action of $\tilde H'=\ker \tilde\tau'$ on the base change $X\times_E \tilde E' \cong \tilde E' \times F$ is compatible with its action on $X^\circ \times_E E_1$ given by $(*{*}*)$, which is diagonal.
From this we conclude that $\tilde H'$ acts on $\tilde E' \times F$ diagonally.
Finally, let $K$ be the kernel of the action of $\tilde H'$ on $F$. Set $H := \tilde H'/K$ and $\bar E := \tilde E'/K$.
Then $H$ acts on $F$ faithfully and the proof is complete.
\end{proof}

\section{Structure of $K$-trivial irregular varieties with $\dim a_X(X) = \dim X-1$}\label{sec:str-thms-reldim=1}
In this section, we treat $K$-trivial irregular varieties whose Albanese morphism has relative dimension one. We work in the following setting.
\begin{assumption}\label{Ass-C}\rm
Let $X$ be a normal projective $\mathbb{Q}$-factorial variety with $K_X \equiv 0$. Let $a_X\colon X\to A$ be the Albanese morphism of $X$ and assume that $a_X\colon X\to a_X(X)$ is of relative dimension one.
Then by Theorem~\ref{thm:cb-formula} and Corollary~\ref{cor:fibration}, we see that $K_X \simQ 0$ and $a_X\colon X\to A$ is a fibration. In the following we set  $f=a_X\colon X\to A$. We fall into one of the following three cases:
\begin{itemize}
  \item[(C1)] $X_{\eta}$ is smooth over $k(\eta)$, that is, $f$ is an elliptic fibration;
  \item[(C2)] $X_{\eta}$ is non-smooth but geometrically integral, that is, $f$ is a quasi-elliptic fibration;
  \item[(C3)] $X_{\eta}$ is geometrically non-reduced, that is, $f$ is inseparable.
\end{itemize}
\end{assumption}
Note that the latter two cases occur only when the characteristic $p = 2$ or $3$ (see Section \ref{sec:curve-pa1}).

\subsection{Case (C1): $X_{\eta}$ is smooth.}\label{sec:elliptic-3fold}
In this case, we give a thorough description of $X$.

\subsubsection{A rough description} 
We first give a rough description of the structure of $X$. 
\begin{thm}\label{thm:elliptic-3fold}
  Under Assumption~\ref{Ass-C} (C1), there exist an abelian variety $A'$ and an isogeny $\tau\colon A'\to A$ such that $X \times_A A' \cong A'\times F$, where $F$ is a general fiber.
  More precisely, there is a faithful action of $G:=\ker(\tau)$ on $F$, such that $X \cong A'\times F/G$, where $G$ acts on $ A'\times F$ diagonally.
\end{thm}

\begin{proof}
We follow the strategy of \cite{PZ19}.

\smallskip{\bf Step~1.}
Let $d=\dim A$, and let $H$ be a sufficiently ample line bundle on $A$.
For general choices of $H_i \in|H|$, the curve $C := H_1\cap\cdots\cap H_{d-1}$ is smooth, and $f_C\colon X_C= X\times_A C \to C$ is flat and smooth over generic point of $C$.
In this step, we aim to show that {\em there exists a finite flat morphism $C'\to C$ such that $f_{C'}\colon X_{C'}=X\times_A C' \to C'$ is a trivial fibration.}

\smallskip

Set $Z_{i} := H_1 \cap \cdots \cap H_i$ for $i=1,\ldots,d-1$. As $H_1 \in |H|$ is general, we may assume that $X_{Z_1}$ is a prime divisor. Then applying the adjunction formula, we have
\begin{equation}\label{eq:8F7E}
  \begin{aligned}
    0\simQ (K_{X/A}) |_{ X_{Z_1}^\nu }    \simQ K_{X_{Z_1}^\nu/Z_1} + \Delta_{X_{Z_1}^\nu},
  \end{aligned}
\end{equation}
where $\Delta_{X_{Z_1}^\nu}\ge0$.
By Theorem~\ref{rel-can-ellfib-sm}, we have $K_{X_{Z_1}^\nu/Z_1} \geQ 0$. In turn, we deduce that $K_{X_{Z_1}^\nu/Z_1} \simQ 0$ and $\Delta_{X_{Z_1}^\nu} = 0$, which implies that $X_{Z_1}$ is normal in codimension one. Inductively we can prove this holds for each $X_{Z_i}$. In particular for $i=d-1$, the surface $X_{C}$ is normal in codimension one and $K_{X_{C}^\nu/C} \simQ 0$. Since $X$ is Cohen-Macaulay in codimension two, we can show that $X_{C}$ is Cohen-Macaulay by induction. Together with that $X_{C}$ is regular in codimension one, we see that $X_C=(X_{C})^\nu$ by Serre's criterion.
 Finally, since $K_{X_C/C}\simQ 0$, we conclude Step~1 by Proposition~\ref{tri-rel-can-ellfib-isotrivial}.


\medskip

Therefore, there exists a big open subset $A^\circ$ of $A$ covered by such curves $C$ as above. Over $A^\circ$, the fibration $f\colon X \to A$ is smooth, and all the closed fibers are isomorphic to a fixed one $F=f^{-1}(t_0)$ of $f$ for some $t_0 \in A^\circ$. In the following, we equip $F$ with a fixed abelian variety structure.

\medskip
{\bf Step~2.}  
In this step, we aim to show that {\em there is a finite group scheme $G$ and a $G$-torsor $I\to A^\circ$ such that $X^\circ \times_{A^\circ} I \to I$ is a trivial fibration, where $X^\circ := f^{-1}(A^\circ)$.}
\smallskip

  We follow the proof of \cite[Theorem~9.1]{PZ19}.
  Let $L$ be a very ample line bundle on $X$, and set $L_0 :=  L|_{F}$.
  We take $G$ to be the automorphism group scheme $\Aut(F, L_0)$ which is finite by \cite[Proposition~10.1]{PZ19},
  and $I$ to be the quasi-projective scheme over $A^\circ$ representing the Isom functor $\Isom_{A^\circ } \bigl((X^\circ , L|_{X^\circ}), ( A^\circ \times F, \mathrm{pr}_2^* L_0)\bigr)$ as constructed in \cite[Construction~7.5]{PZ19}.
  Following the proof of \cite[Theorem~9.1]{PZ19}, we only need to show that $I\to A^\circ$ is surjective and flat.
  For each curve $C\subset A^\circ$ as in Step~1, there exists a finite flat morphism $C'\to C$ such that $f_{C'}\colon X_{C'}\to C'$ is a trivial fibration.
  By the base change property \cite[Proposition~7.8]{PZ19} and flattening decomposition \cite[Lecture~8]{Mum66}, it suffices to verify that $I_{C'}:= \Isom_{C'} \bigl( ( X_{C'}, L|_{X_{C'}} ),  (C' \times F, \mathrm{pr}_2^*  L_0)\bigr) \to C'$ is surjective and flat.
  To verify this, we apply \cite[Lemma~8.6]{PZ19}, which requires the condition that $-K_{X_{C'}/C'}$ is nef and semi-ample.
  This condition is satisfied because $f_{C'}$ is a trivial fibration.
  This finished the proof of this step.

\smallskip{\bf Step~3.}
We extend the $G$-torsor $I\to A^\circ$ over $A^\circ$ to a $G$-torsor $\bar I\to A$ over $A$ as follows.

Regard $F$ as an abelian variety with identity $0_F$. Since $G=\Aut(F, L_0)$ is a finite group scheme, $G$ is an extension of an \'etale group scheme by an infinitesimal group scheme.
Thus by a ``Purity'' theorem (\cite[Proposition~1.4]{Eke88}), the torsor $I\to A^\circ$ can be extended to a torsor $\bar I\to A$.
  Furthermore, according to \cite[Proposition]{Nor83}, there exist an integer $n$, a homomorphism $\varphi\colon A[n]\to \Aut(F, L_0)$, and a morphism $\eta\colon  A\to\bar I$,
  which is $A[n]$-equivariant with the action of $A[n]$ on $\bar I$ induced by $\varphi$,
  leading to the following commutative diagram:
  \[
  \xymatrix@C=70pt@R=20pt{
      A \ar[dr]_{n_A} \ar[r]^{ \text{ $A[n]$-equivariant}}_{\eta} & \bar I \ar[d] \\
      & A \rlap{\,.}
  }\]
Let $A^{\circ\circ}$ denote the preimage of $A^{\circ}$ under $n_A\colon A\to A$. By the construction, we obtain an isomorphism $X\times_{A,n_A} A^{\circ\circ}\cong A^{\circ\circ} \times F$.

  \medskip
  {\bf Step~4.}
  We show that the birational map $\psi\colon Y :=X\times_{A,n_A}A \dashrightarrow Y':=F\times A$, which is determined by $X\times_{A,n_A} A^{\circ\circ}\cong A^{\circ\circ} \times F$, is an isomorphism.
  \smallskip

  Since $f\colon X\to A$ is equidimensional by \cite[Theorem 4.1]{PZ19}, $Y\to A$ is equidimensional too.
  Since $\psi$ is an isomorphism in codimension one, $Y$ is regular in codimension one, thus $Y$ is  normal by Lemma~\ref{lem:flat-base-change-integral}.

  Let $H_Y$ be an ample Cartier divisor on $Y$, and denote by $H_{Y'}= \psi_*H_Y$ the strict transform of $H_Y$. Observe that $H_{Y'}$ is relatively ample over $A$ since each fiber of $Y' \cong F\times A \to A$ is irreducible.  If necessary adding the pullback of an ample divisor on $A$, we may assume that $H_{Y'}$ is ample. Since $\psi\colon Y \dashrightarrow Y'$ is an isomorphism in codimension one, we have a natural ring isomorphism $\bigoplus_{m\geq 0 }H^0(Y, mH_Y) \cong \bigoplus_{m\geq 0 }H^0(Y', mH_{Y'})$, which implies that $\psi\colon Y \to Y'$ is an isomorphism. 


  \medskip
  {\bf Step~5.}
 Finally if the action of $A[n]$ on $F$ is not faithful, we may set
$H=\mathop{\rm Ker} (A[n] \to \Aut(F, L_0))$ and $G=A[n]/H$. In turn, we get a faithful action of $G$ on $F$ and an action on $A' = A/H$ such that $X \cong (A\times F)/G \cong (A'\times F)/H$.
\end{proof}

\subsubsection{Explicit description of Case (C1)}\label{sec:explicite-of-C1}
Based on the analysis in \cite[pages~36-37]{BM77}, we give all the possibilities of $X$ in the context of Theorem~\ref{thm:elliptic-3fold}.
\smallskip

Recall that $X \cong (A \times F)/G$, where $A$ is an abelian variety, $F$ is an elliptic curve, and $G$ is a finite group scheme acting diagonally on $A \times F$ via injections $G\hookrightarrow A$ and $\alpha\colon G\hookrightarrow \Aut(F)\cong F\rtimes \Aut(F,0_F)$. In particular, $G$ is commutative.
As observed in \cite{BM77}, the commutativity of $G$ severely constrains its possible structure. 
In fact, the argument of \cite[page~36]{BM77} shows that $G$ must have the form
\[ 
  \alpha(G)\cong G_0\times \mathbb Z/n\mathbb Z,
\]
where $G_0$ is a finite subgroup scheme of $F$, and $\mathbb Z/n\mathbb Z\subseteq \Aut(F,0_F)$ is a cyclic subgroup with order $n = 2,3,4$ or $6$.
Moreover, if we denote by $\sigma\in G$ the element corresponding to some generator of $\mathbb Z/n\mathbb Z$, then $G_0 \subseteq \mathfrak F$, where $\mathfrak F\subset F$ is the fixed subscheme of $\sigma$.
There are the following possibilities of the fixed subscheme $\mathfrak F$:
\begin{itemize}
  \item[(a)] $n=2$, (so $\sigma = -1_F$), then $\mathfrak F \cong \mathop{\rm Ker} 2_{F}$.
  \item[(b)] $n=3$, then $\mathop{\rm ord}\mathfrak F = 3$, so \ \
  \begin{minipage}[t]{28em}
  	$\mathfrak F \cong \mathbb Z/3\mathbb Z$, if $\mathop{\rm char} \ne 3$;\par
  	$\mathfrak F \cong \alpha_3$, if $\mathop{\rm char} = 3$ ($j(F) = 0$ and thus $F$ is supersingular);\end{minipage}
  \item[(c)] $n=4$,
    then $\mathop{\rm ord}\mathfrak F = 2$, so \ \
    \begin{minipage}[t]{28em}
    $\mathfrak F \cong \ZZ/2\ZZ$ if $\mathop{\rm char} \ne 2$; \par
    $\mathfrak F \cong \alpha_2$ if $\mathop{\rm char} = 2$ ($F$ is supersingular); \end{minipage}
  \item[(d)] $n=6$, then $\mathfrak F = (e)$.
\end{itemize}

\smallskip
We can now give a complete list of the possible $G$:
\columnratio{0.4}
\begin{paracol}{2}
	a1) $X \cong A \times F /(\ZZ/2\ZZ)$.
	\switchcolumn
	\noindent The action is given by $(x,y) \mapsto (x+a, -y)$ for some $a\in A[2]$.
	Moreover, $A$ has $p$-rank $\ge1$ if $p=2$.
	
    \medskip
	\switchcolumn*
	a2) $X \cong A \times F /(\ZZ/2\ZZ\cdot \ZZ/2\ZZ)$.
	\switchcolumn
	\noindent The action is given by $(x,y) \mapsto (x+a,y+b)$ and $(x,y) \mapsto (x+c,-y)$ for some $a\not=c\in A[2]$ and $b\in F[2]$.
	Moreover, $A$ and $F$ are both ordinary if $p=2$.
	
    \medskip
	\switchcolumn*
	a3) $X \cong A \times F /((\ZZ/2\ZZ)^2\cdot \ZZ/2\ZZ)$.
	\newline\centerline{ ($p \neq 2$)}
	\switchcolumn
	\noindent The factor $\ZZ/2\ZZ$ acts as $(x,y) \mapsto (x+c,-y)$ and $(\ZZ/2\ZZ)^2$ acts by translation on both factors.
	
    \medskip
	\switchcolumn*
	a4) $X \cong A \times F /(\mu_2 \cdot \ZZ/2\ZZ)$. \newline\centerline{ ($p=2$) }
	\switchcolumn
	\noindent The factor $\ZZ/2\ZZ$ acts as $(x,y) \mapsto (x+c,-y)$ and $\mu_2 $ acts by translation on both factors.
	Here $A$ has $p$-rank $\ge1$ and $F$ is ordinary.
	
    \medskip
	\switchcolumn*
	a5) $X \cong A \times F /(\ZZ/2\ZZ\cdot \mu_2 \cdot \ZZ/2\ZZ)$. \newline\centerline{ ($p=2$)}
	\switchcolumn
	\noindent One factor $\ZZ/2\ZZ$ acts as $(x,y) \mapsto (x+c,-y)$ and the residue factor $\ZZ/2\ZZ\cdot \mu_2 $ acts by translation on both factors.
	Here, $A$ and $F$ are both ordinary.
	
    \medskip
	\switchcolumn*
	a6) $X \cong A \times F /(\alpha_2 \cdot \ZZ/2\ZZ)$. \newline\centerline{ ($p=2$)}
	\switchcolumn
	\noindent The factor $\ZZ/2\ZZ$ acts as $(x,y) \mapsto (x+c,-y)$ and $\alpha_2$ acts by translation on both factors.
	Here $A$ has $p$-rank $1$ and $F$ is supersingular.
	
    \medskip
	\switchcolumn*
	a7) $X \cong A \times F /(M_2 \cdot \ZZ/2\ZZ)$,
	 \newline\centerline{ ($p=2$)}
	\switchcolumn
	\noindent where $M_2$ is the non-split extension of $\alpha_p$ by $\alpha_p$.
	The factor $\ZZ/2\ZZ$ acts as $(x,y) \mapsto (x+c,-y)$ and  $M_2$ acts by translation on both factors.
	Here $A$ has $p$-rank $1$, and $F$ is supersingular.
	
    \medskip
	\switchcolumn*
	b1) $X \cong A \times F /(\ZZ/3\ZZ)$.
	\newline\centerline{ ($j(F)=0$)}
	\switchcolumn
	\noindent The group $\ZZ/3\ZZ$ acts as $(x,y) \mapsto (x+a,\omega y)$, where $a\in A[3]$ and  $\omega$ is an automorphism of $F$ of order $3$.
	
    \medskip
	\switchcolumn*
	b2) $X \cong A \times F /(\ZZ/3\ZZ)^2$.
	\newline\centerline{ ($j(F)=0$ and $p \neq 3$)}
	\newline\centerline{\ }
	\switchcolumn
	\noindent The action is given by $(x,y) \mapsto (x+a,\omega y)$ and $(x,y) \mapsto (x+b,y+c)$ for some $a,b\in A[3]$ and $c\in F[3]$, $\omega$ as before and $a\not=b,2b$, $\omega c=c$.
	
    \medskip
	\switchcolumn*
	b3) $X \cong A \times F /(\alpha_3\cdot \ZZ/3\ZZ )$.
	\newline\centerline{ ($j(F)=0, p=3$)}
	\switchcolumn
	\noindent The factor $\ZZ/3\ZZ$ acts as in b1) and $\alpha_3$ acts by translation on both factors.
	Here $A$ has $p$-rank $1$ and $F$ is supersingular.
	
    \medskip
	\switchcolumn*
	c1) $X \cong A \times F /(\ZZ/4\ZZ)$.
	\newline\centerline{ ($j(F)=12^3$)}
	\switchcolumn
	\noindent The group $\ZZ/4\ZZ$ acts as $(x,y) \mapsto (x+a,i y)$, where $a\in A[4]$ and  $i$ is an automorphism of $F$ of order $4$.
	
    \medskip
	\switchcolumn*
	c2) $X \cong A \times F /(\ZZ/2\ZZ \cdot \ZZ/4\ZZ )$.
	\newline\centerline{ ($j(F)=12^3, p \neq 2$)}
	\switchcolumn
	\noindent The group $\ZZ/4\ZZ$ acts as in c1) and $\ZZ/2\ZZ$ acts as $(x,y) \mapsto (x+b, y+c)$ for some $b\in A[2], c\in F[2]$ and $b\not=2a$, $ic=c$.
	
    \medskip
	\switchcolumn*
	c3) $X \cong A \times F /(\alpha_2 \cdot \ZZ/4\ZZ) $.
	\newline\centerline{ ($j(F)=0, p=2$)}
	\switchcolumn
	\noindent The factor $\ZZ/4\ZZ$ acts as in c1) and $\alpha_2$  by translation on both factors.
	Here $A$ has $p$-rank $1$ and $F$ is supersingular.

    \medskip
	\switchcolumn*
	d)  $X \cong A \times F /(\ZZ/6\ZZ)$.
	\newline\centerline{ ($j(F)=0$)}
	\switchcolumn
	\noindent The group $\ZZ/6\ZZ$ acts as $(x,y) \mapsto (x+a,-\omega y)$, where $a\in A[6]$ and  $\omega$ is as in b1).
\end{paracol}

\subsection{Case (C2): $f$ is a quasi-elliptic fibration}\label{sec:C-quasi-ell}
In this case, we prove the following theorem.

\begin{thm}\label{thm:C-quasi-ell}
  Let notation and assumptions be as in Assumption~\ref{Ass-C} {\rm(C2)}. Then
\begin{itemize}
  \item[\rm(i)] The characteristic $p$ is $2$ or $3$.
  \item[\rm(ii)] $X$ admits another fibration $g\colon X \to \mathbb{P}^1$ that is transversal to $f$:
    $$\xymatrix@R=4ex@C=4ex{&X\ar[d]_{f}\ar[r]^{g} & \mathbb{P}^1 \,.\\ &A  &}$$
    Moreover, a fiber of $g$ is either an abelian variety or a multiple of an abelian variety.
  \item[\rm(iii)] Let $X_1$ be the normalization of $X \times_{A} A^{(-1)}$ and denote by $f_1\colon X_1\to A_1 :=A^{(-1)}$ the projection. Then $f_1$ is a smooth fibration fibred by rational curves, which falls into one of the following cases:
  \begin{itemize}
    \item[\rm(1)] $f_1\colon X_1\to A_1$ is a projective bundle described as one of the following:
      \begin{itemize}
        \item[\rm(1.a)] $X_1\cong \mathbb{P}_{A_1}(\mathcal{O}_{A_1}\oplus \mathcal{L})$, where $\mathcal{L}^{\otimes p+1}\cong \mathcal O_X$;
        \item[\rm(1.b)] $X_1\cong \mathbb{P}_{A_1}(\mathcal{E})$, where $\mathcal{E}$ is a unipotent vector bundle of rank two, and
          there exists an \'etale cover $\mu\colon A_2\to A_1$ of degree $p^v$ {\rm(}$0\le v\le d:=\dim A${\rm)}, such that $\mu^*F^{(d-v)*}_{A_1}\mathcal{E}$ is trivial.
      \end{itemize}
    \item[\rm(2)] $p=2$, and there exists a purely inseparable isogeny $A_2 \to A_1$ of degree two, such that $X_2 :=X_1\times_{A_1}A_2\to A_2$ is a projective bundle described as one of the following:
      \begin{itemize}
        \item[\rm(2.a)] $X_2\cong \mathbb{P}_{A_2}(\mathcal{O}_{A_2}\oplus \mathcal{L})$, where $\mathcal{L}^{\otimes 4}\cong \mathcal{O}_{A_2}$;
        \item[\rm(2.b)] $X_2\cong \mathbb{P}_{A_2}(\mathcal{E})$, where $\mathcal{E}$ is a unipotent vector bundle of rank two, and there exists an \'etale cover $\mu\colon A_3\to A_2$ of degree $p^v$ {\rm(}$0\le v\le d${\rm)}, such that $\mu^*F^{(d-v)*}_{A_2}\mathcal{E}$ is trivial.
      \end{itemize}
  \end{itemize}
\end{itemize}
\end{thm}

\begin{proof}
The assertion (i) is well known. Let us prove the remaining ones.

Note that $X \times_A A^{\frac{1}{p}}$ is not normal by Proposition \ref{prop:ga1-geo-reduced}.
Let $X_1:=(X \times_A A^{\frac{1}{p}})^{\nu}$ be the normalization, and let $\pi\colon X_1\to X$, $f_1\colon X_1\to A^{\frac1p}$ be the natural morphisms.
We have the following commutative diagram
\begin{equation*}
 \vcenter{\xymatrix{
      &X_1\ar[r]\ar[rd]_{f_1}\ar@/^1pc/[rr]^{\pi} &X \times_A A^{\frac{1}{p}} \ar[r]\ar[d]^{f'}&X\ar[d]^{f}\\
      &  &A_1:=A^{\frac{1}{p}}\ar[r]^>>>>>{\sigma}&A \rlap{\,.}
  }}
\end{equation*}
By results of Section~\ref{sec:base-change}, we can write that
  \begin{equation}
    \pi^* K_X \sim K_{X_1} + (p-1)\mathcal C,  \quad\text{with } \mathcal C\ge0,
  \end{equation}
where $\mathcal C$ can be chosen such that $\mathcal C|_{X_{1,\eta}}$ coincides with the conductor of the normalization of the generic fiber of $f'$.  Hereafter, we fix such $\mathcal C$.
According to Proposition~\ref{prop:ga0} (1), $\deg_{K(A_1)}(p-1)\mathcal C =2$.
Write $(p-1)\mathcal C = H + V$, where $H$ is the $f_1$-horizontal part and $V$ the vertical part of $\mathcal C$.
Note that $H$ is irreducible by Proposition~\ref{prop:ga1-geo-reduced} (1), thus $H=nC$, where $C$ is reduced and $n=1$ or $2$.
Let $D$ be the reduced divisor supported on $\pi(C)$. We have $\pi^*D = C$ or $pC$ by Proposition~\ref{prop:ga1-geo-reduced}.

\medskip
{\bf Step~1.}
{\it We prove that and both $C$ and $D$ are semi-ample with numerical dimension one.
}

\smallskip
Since $X_1\to X$ is a finite purely inseparable morphism, the semi-ampleness of one of the divisors $C$ or $D$ implies the semi-ampleness of the other.
We shall show that $D$ is semi-ample with numerical dimension one.
  By Theorem~\ref{thm:hor-map-1}, it suffices to verify that $D|_{D^{\nu}}\simQ 0$, which is equivalent to that $C|_{C^{\nu}}\simQ 0$.

\smallskip
{\it {\sc Claim}.  Let $T$ be a $f_1$-horizontal prime divisor of $X_1$. Then $T|_{T^\nu} \geQ 0$.}
\smallskip

To prove this claim, set $\bar{T} := \pi(T)$.
By Lemma~\ref{lem:adj-charA} (1), $\bar{T}|_{\bar{T}^{\nu}} \simQ (K_X+ \bar{T})|_{\bar{T}^{\nu}} \geQ 0$.
Then, since $X_1 \to X$ is finite and purely inseparable, we have $T|_{T^{\nu}} \geQ 0$.
\qed\smallskip

Consequently, since $C$ is $f_1$-horizontal, we have $C|_{C^{\nu}}\geQ 0$.
Thus, applying Lemma~\ref{lem:adj-charA} to
  \begin{equation}
    0\simQ (K_{X_1} + \mathcal C)|_{C^\nu} \simQ (K_{X_1} + nC+V)|_{C^\nu}
  \end{equation}
 we see that $C$ is an abelian variety, $V|_{C^\nu}\simQ 0$, and moreover, 
in case $n=2$, we have $C|_{C^\nu} =C|_{C}\simQ 0$. 
Thus we only need to show $C|_{C} \simQ 0$ in case $n=1$.
In this case, we have $p=2$, and $C\to A_1$ is purely inseparable of degree $2$ since $C|_{X_{K(A_1)}}$ is a point which is inseparable of degree $2$ over $\Spec K(A_1)$ by Proposition \ref{prop:ga1-geo-reduced}.
By doing the base change $C \to A_1$ and setting $Z=(X_1\times_{A_1} C)^{\nu}$, we have the following commutative diagram:
  $$\xymatrix{
    Z\ar[r]_<<<<{\nu}\ar[dr]_{f_2}\ar@/^1pc/[rr]^>>>>>{\phi}\ar[dr]\ar@/^2pc/[rrr]^{\pi'}
      &X_1\times_{A_1} C\ar[r]\ar[d]&X_1\ar[r]_{\pi}\ar[d]_{f_1}&X\ar[d]^{f} \\
      &C\ar[r]&A_1:=A^{\frac{1}{p}}\ar[r]&A \rlap{\,.}
  } $$
where $\phi, \pi',f_2$ denote the natural morphisms fitting into the above diagram. Since $C\times_{A_1} C$ is non-reduced and $f_1\colon X_1\to A_1$ is smooth over the generic point of $A_1$, we see that $\phi^*C=2E$ for some $\ZZ$-divisor $E$ on $Z$.
  Then we have \begin{equation}
    \pi'^*K_X = K_Z + 2E + V',
  \end{equation}
  where $V'\ge0$ is vertical over $C$.
  Consequently, $(K_Z + 2E + V')|_{E^\nu} \simQ 0$.
  Since $C|_{C}\geQ0$, we have $E|_{E^\nu}\geQ 0$.
  It follows from Lemma~\ref{lem:adj-charA} that $E|_{E^\nu} \simQ 0$, which is equivalent to that $C|_{C} \simQ 0$.

  In summary, the semi-ample divisor $D$ (resp.\ $C$) induces a fibration $g\colon X\to B$ (resp.\ $g_1\colon X_1\to B_1$) to a curve.
  Since general fibers of $f$ and $f_1$ are rational curves, we have $g(B) = g(B_1) =0$.
 In summary, there is a commutative diagram:
  $$\xymatrix@C=2ex@R=2ex{
    &\PP^1 \ar[rr] && \PP^1\\
    X_1 \ar[rr]\ar[dd]_{f_1}\ar[ru]^{g_1} && X\ar[dd]_{f}\ar[ru]^<<{g} \\
    \\
    A_1 \ar[rr]   && A \,.
  }$$

\medskip

  \medskip
  {\bf Step~2.}  We prove the following statements:
  \begin{itemize}\it
    \item[\rm(a)] $D$ is isomorphic to an abelian variety (we have shown this for $C$ in the last step).
    \item[\rm(b)] $X$ (resp.\ $X_1$) is regular at codimension one points of $D$ (resp.\ $C$).
    \item[\rm(c)] $V = 0$.
    \item[\rm(d)] A general fiber of $g$ (resp.\ $g_1$) is an abelian variety, and a special fiber of $g$ (resp.\ $g_1$) is a multiple of an abelian variety.
  \end{itemize}

  First, since $C$ and $D$ are irreducible and $C|_{C^\nu} \simQ D|_{D^\nu} \simQ 0$, we have $(K_X + D)|_{D^\nu} \simQ (K_{X_1} + C)|_{C^\nu} \simQ 0$, thus the statements (a,\thinspace b) follow from Lemma~\ref{lem:adj-charA}.

  To show the last two statements, we denote by $G_t$ (resp.\ $G^1_t$) the fiber of $g$ (resp.\ $g_1$) over $t\in \PP^1$.
  We first consider the fibration $g_1\colon X_1\to \PP^1$.
  Write $G^1_t=mT+V'$, where $T$ is an $f_1$-horizontal component and $V'$ is the remaining part with $T\not\subset \mathop{\rm Supp}V'$. 
  By the claim in Step~1, we have $T|_{T^\nu}\geQ 0$. 
  Since $C$ is irreducible, we see that $C\simQ rG^1_t$ for some positive rational number $r$, and it follows that $C|_{T^\nu}\simQ  rG^1_t|_{T^\nu}\simQ 0$.
  Thus \[
    0 \simQ (K_{X_1} + nC + V+G^1_t)|_{T^\nu}\simQ  (K_{X_1} +  mT+V'+ V)|_{T^\nu}.\]
  By Lemma~\ref{lem:adj-charA} (2), we see that $T$ is an abelian variety, $V|_{T^\nu}=V'|_{T^\nu}=0$, which implies that $\mathop{\rm Supp}V' \cap \mathop{\rm Supp}T= \mathop{\rm Supp}V \cap \mathop{\rm Supp} T = \emptyset$.
  Since the fiber $G^1_t$ is connected, we have $V'=0$, and consequently $G^1_t=mT$ is (a multiple of) an abelian variety.
  It follows that $\mathop{\rm Supp}G^1_t \cap \mathop{\rm Supp}V =\emptyset$, and thus $V=0$, which is the statement (c).
  Moreover, since $g_1$ is a fibration to a curve, by \cite[Corollary~7.3]{BadescuAS}, a general fiber of $g$ is integral, so we obtain the statement (d) for $g_1$.

  Finally, by writing $G_t = mT + V'$ and using $(K_X + G_t)|_{T^\nu} \simQ (K_X + mT + V')|_{T^\nu} \simQ 0$, Lemma~\ref{lem:adj-charA} implies that $G_t$ is also (a multiple of) an abelian variety.
Thus we obtain the statement (d) for $g$.

  \medskip
   {\bf Step~3.} {\it
   We show that there exists an isogeny $\tau \colon B\to A_1$ of abelian varieties such that $X_1 \times _{A_1} B \cong B \times \PP^1$.
   In turn, $f_1\colon X_1 \to A_1$ is a smooth morphism.
   }
\smallskip

Take a general fiber $G^1_t$ of $g_1\colon X_1\to \PP^1$ over $t\in \PP^1$, which is an abelian variety as established in the previous step.  
We denote by $n$ the degree $\deg(G^1_t\to A_1)$, which is independent of $t\in \PP^1$.
  Then the morphism $\times n\colon B:=A_1 \to A_1$ factors through the isogeny $G^1_t\to A_1$ for a general $t$ (see \cite[page~169, Remark]{MumfordAV}).

  By Lemma~\ref{lem:flat-base-change-integral}, the fiber product $X_{1,B} := X_1\times_{A_1}B$ is integral.
Let $W$ be the normalization of $X_{1,B}$.
  We have the following commutative diagram
$$
    \xymatrix{
      & &\mathbb P^1 \ar[rd] \\
      W \ar@/^4mm/[urr]^q \ar[r]^{\nu} \ar[rd]^p \ar@/^6mm/[rr]|{\,\pi\,} & X_{1,B} \ar[r] \ar[d] & X_1 \ar[d]^{f_1} \ar[r]^{g_1} & \mathbb P^1 \\
      & B \ar[r]^{\tau} & A_1 \rlap{,} 
    }
$$
  where $q$ is the fibration resulting from the Stein factorization of $W\to X_1 \buildrel g_1\over\to\PP^1$.
  By our choice of the base change $B\to A_1$, a general fiber $Q_{t}$ of $q$ is a birational section of $p$.
  Since each fiber of $g_1$ contains no vertical components over $A_1$ and $\nu$ is finite, each fiber of $q$ contains no vertical components over $B$.
  Thus, for each fiber $Q_t$ of $q$, the morphism $Q_t\to B$ is birational.
  Denote by $\mathcal N$ the conductor divisor of $\nu$.
  Remark that, since $X_1\to A_1$ is generically smooth, $\mathcal N$ is vertical over $B$.
  It follows that $K_W + (p-1)\pi^*\mathcal C + \mathcal N \simQ \pi^*(K_{X_1} + (p-1)\mathcal C) \simQ 0$.
  Thus \[
       ( K_W + (p-1)\pi^*\mathcal C + \mathcal N + Q_{t}) |_{Q^\nu_t}   \simQ 0.
  \]
Applying Lemma~\ref{lem:adj-charA} we see that $Q_t = Q^\nu_t$ is an abelian variety and $\mathcal N|_{Q_t} \simQ 0$. In turn, we conclude that
\begin{itemize}
\item $\mathcal N = 0$, thus $X_{1,B}$ is normal by Lemma~\ref{lem:flat-base-change-integral}; and
\item for each $t\in \PP^1$, the projection $Q_t\to B$ is an isomorphism, hence the morphism $X_{1,B}\cong W\to B\times\PP^1$ is an isomorphism by Zariski's main theorem.
\end{itemize}

\medskip{\bf Step~4.}
 We consider the case $(p-1)\mathcal C=2C$ and prove the statement (iii-1).
\smallskip

  We first show that $(p+1) C|_{C} \sim 0$.
  Denote by $\tilde\pi\colon C\to D$ the induced finite morphism of abelian varieties and note that $\pi^*D = pC$ by Proposition~\ref{prop:ga1-geo-reduced}.
  By Step~2 (c), $K_X+D$ is Cartier at codimension-one points of $D$.
  Therefore, the adjunction formula gives
  \begin{equation}
    (K_X + D)|^w_{D} \sim K_D \sim 0,
  \end{equation}
  where $|^w$ is the restriction on $D$ by first considering a big open subset $D^\circ \subset D$ over which $K_X + D$ is Cartier and then extending it (see \cite[pages~173-174]{Kol92}).
  Thus $\pi^* (K_X + D) |_{C} = \tilde\pi^* ( (K_X+D)|^w_{D} )\sim 0$ (here the pullback of $\pi$ makes sense since $\pi$ is finite).
  On the other hand, we have
  \begin{equation*}
      \pi^* (K_X + D) |_{C} \sim (K_{X_1}+ 2C +\pi^*D)|_{C} \sim (p+1)C|_{C}.
  \end{equation*}
  It follows that $(p+1) C|_{C} \sim 0$.

  Next, we equip the smooth morphism $f_1\colon X_1\to A_1$ with a projective bundle structure over $A_1$.
  Note that since $C\to A_1$ is birational and $C$ is isomorphic to an abelian variety, $C\to A_1$ is an isomorphism; this gives a section $s\colon A_1\to C$ of $f_1$ which is $f_1$-ample.
    Set $\mathcal{E}=f_{1*}\mathcal{O}_{X_1}(C)$. 
    Since each fiber of $f_1$ is isomorphic to $\mathbb{P}^1$, by exactly the same proof of \cite[Proposition~V.2.2]{Hartshorne-AG}, we can show that $f_1^*\mathcal E\to \mathcal O_{X_1}(C)$ is surjective and induces an isomorphism $X_1 \buildrel\sim\over\to \PP(\mathcal{E})$.
  Under this isomorphism, the section $s\colon A_1\to C$ corresponds to the exact sequence
  \begin{equation}\label{eq:G78V}
  	0\to \mathcal{O}_{A_1}\to \mathcal{E}\to s^*(\mathcal{O}_{X_1}(C)|_C)\to 0.
  \end{equation}
  Since the divisor $C|_C$ is torsion,  $\mathcal{L}:=s^*(\mathcal{O}_{X_1}(C)|_C)$ is a torsion line bundle with the same order.

  If the above exact sequence (\ref{eq:G78V}) splits, then $\mathcal{E}\cong \mathcal{O}_{A_1}\oplus \mathcal{L}$,  and the torsion order of $\mathcal{L}$ divides $p+1$.
  If (\ref{eq:G78V}) does not split, then $\mathrm{Ext}^1(\mathcal{L},\mathcal{O}_{A_1})\cong H^1({A_1},\mathcal{L}^{-1}) \ne 0$, which implies that $\mathcal{L}\cong\mathcal{O}_{A_1}$ by \cite[Lemma~7.19]{E-vdG-M}.
Thus $\mathcal{E}$ is an extension of $\mathcal{O}_{A_1}$ by $\mathcal{O}_{A_1}$, which corresponds to a nonzero element $\xi \in H^1(A_1, \mathcal{O}_{A_1}) \cong \mathrm{Ext}^1(\mathcal{O}_{A_1},\mathcal{O}_{A_1})$.
With respect to the Frobenius action $F^*$ on $H^1(A_1, \mathcal{O}_{A_1})$, we have a decomposition (\cite[pages~143-148]{MumfordAV})
$$H^1(A_1, \mathcal{O}_{A_1}) =  V_n \oplus V_s, $$
where $V_n$ is the nilpotent part and $V_s$ is the semisimple part. Moreover, the semisimple part $V_s$ admits a basis $\alpha_1,\ldots,\alpha_r$, where $r = \dim V_s$, such that $F^*(\alpha_i) = \alpha_i$ for $i=1,\ldots,r$.
Write $\xi = \xi_n + \xi_s$.
Since $\dim V_n = \dim A-r = d-r$, we see that $F^{(d-r)*} \xi_n = 0$.
By \cite[Satz~1.4]{Lange-Stuhler77}, there exist \'etale covers $\pi_i\colon A_i\to A$ ($i=1,\ldots,r$) of degree $p$ such that $\pi_i^* \alpha_i = 0$.
Define $\mu$ as the fiber products $(\cdots((A_1 \times_ A A_2) \times_A A_3) \times_A \cdots)\times_A A_r \to A$. Then $\mu^* \xi_s = 0$.
Therefore $\mu^*(F^{(d-r)*}\xi) = 0$, which is equivalent to $\mu^*(F^{(d-r)*}E)$ being a trivial extension.

\medskip{\bf Step~5.}
 We consider the case $(p-1)\mathcal C=C$ and prove the statement (iii-2).
\smallskip

In this case, $p=2$, and $C\to A_1$ is purely inseparable of degree 2. We do a further base change $A_2:=C \to A_1$.
Let $X_2= X_1\times_{A_1} A_2$. We obtain the following commutative diagram
 \begin{align*}\label{diag:qs-ell-bc2}
 \xymatrix{
 &X_2 \ar[r]^{\pi_2}\ar[d]_{f_2} &X_1\ar[r]^{\pi_1}\ar[d]_{f_1} &X\ar[d]^{f}\ar[r]^{g} &\mathbb{P}^1 \\
 &A_2\ar[r]&A_1 \ar[r]  &A \rlap{\,.}
  } \end{align*}
We see that $\pi_2^*C = 2C_2$, where $C_2$ is a section of $X_2\to A_2$. Similar to Step~4, we have $X_2 \cong \mathbb{P}_{A_2}(\mathcal{E}_2)$, where $\mathcal{E}_2 = f_{2*}\mathcal{O}_{X_2}(C_2)$.

If  $\mathcal{E}_2$  does not split, then it is as described as in Step~4 accordingly.
Otherwise $\mathcal{E}_2 \cong \mathcal{O}_{A_1}\oplus \mathcal{L}_2$ and we conclude the proof by determining the torsion order of $C_2|_{C_2}$.
By Proposition~\ref{prop:ga1-geo-reduced}, we have $\pi^*D = C$ if $\deg_{K(A)} D = 2$ and $\pi^*D = 2C$ if $\deg_{K(A)} D = 4$.
It follows that
\begin{equation*}
  \pi^* (K_X + D) |_{C} \sim (K_{X_1}+\mathcal C+\pi^*D)|_{C} \sim
  \begin{cases}
     C|_{C} &\text{if $\deg_{K(A)}D = 2$}\\
    2C|_{C} &\text{if $\deg_{K(A)}D = 4$.}
  \end{cases}
\end{equation*}
As in Step~4, we have $\pi^*(K_X + D)|_{C}\sim0$, thus $2 C|_{C} \sim 0$.
Since $\pi_2^*C = 2C_2$, we see that $4C_2|_{C_2} \sim 0$.
\end{proof}

\subsection{Case (C3): $f$ is an inseparable fibration}
\begin{thm}\label{thm:C-nonreduced}
Let notation and assumptions be as in Assumption~\ref{Ass-C} (C3). Then
\begin{itemize}
  \item[\rm(i)] The characteristic $p$ is $2$ or $3$.
  \item[\rm(ii)] $X$ is equipped with another fibration $g\colon X \to \mathbb{P}^1$ transversal to $f$, whose fibers are abelian varieties or multiples of abelian varieties:
    $$\xymatrix@R=4ex@C=4ex{&X\ar[d]_{f}\ar[r]^{g} & \mathbb{P}^1 \,.\\ &A  & }$$

  \item[\rm(iii)] Let $A_1= A^{(-1)}$, and denote $X_1 = (X\times_A A_1)_{\rm red}^\nu$.
    Then one of the following happens.
    \begin{itemize}
      \item[\rm (1)]  $X\cong A_1 \times\mathbb P^1/\mathcal F$, where $\mathcal F$ is a foliation on $A_1\times\PP^1$ of height one and $\mathop{\rm rank} \mathcal F < \dim A$.
      \item[\rm(2)] $p=2$ and $X \cong X_1/\mathcal F_1$ for some height one foliation $\mathcal F_1$ with $\mathop{\rm rank} \mathcal F_1 < \dim A$, where $X_1$ is one of the following:
        \begin{itemize}
          \item[\rm (2.a)] $X_1\to A_1$ is separable and can be trivialized by an \'etale isogeny $\tau_2\colon A_2\to A_1$ of degree $2$, namely $X_1 \times_{A_1} A_2 \cong A_2 \times \PP^1$.
          \item[\rm (2.b)] $X_1\to A_1$ is separable and can be trivialized by the Frobenius base change $\tau_2\colon A_2:=(A_1)^{(-1)}\to A_1$, namely $X_1 \times_{A_1} A_2 \cong A_2 \times \PP^1$.
          \item[\rm (2.c)] $X_1\to A_1$ is inseparable, and $X_1\cong A_2\times\PP^1/\mathcal F_2$, where $A_2$ is an abelian variety and $\mathcal F_2$ is a (height-one) foliation with $\mathop{\rm rank} \mathcal F_2 <\dim A$.
        \end{itemize}
    \end{itemize}
\end{itemize}
\end{thm}

\begin{proof}
In this case $X\times_A A_1$ is non-reduced. Let $X_1=(X\times_A A_1)_{\rm red}^{\nu}$. Denote by $\pi_1\colon X_1\to X$ the induced morphism, which is a purely inseparable morphism of height one with $\deg \pi_1 < p^{\dim A}$.
We have $\pi_1^*K_X \sim K_{X_1} - (p -1)\det\mathcal{F}_{X_1/X}$, where $\lvert-\det\mathcal{F}_{X_1/X}\rvert$ has non-trivial movable part generated by $f_1^*H^0(\Omega_{A^{(-1)}/A})$ (see Section~\ref{sec:base-change}).
We may write that
\begin{equation}
  \lvert -\det \mathcal F_{X_1/X} \rvert= \mathfrak M + \mathfrak F,
\end{equation}
where $\mathfrak M$ is the movable part and $\mathfrak F$ is the fixed part.
In turn we obtain that $\pi_1^*K_X \sim K_{X_1} + (p -1)( \mathfrak M + \mathfrak F)$ where $\deg_{K(A_1)}\mathfrak M >0$, and $\deg_{K(A_1)} (p-1)(\mathfrak M + \mathfrak F) = 2$ by results of Section~\ref{sec:curve-pa1}.

\smallskip
By the claim in Step~1 of Theorem~\ref{thm:C-quasi-ell}, we have $T|_{T^\nu} \geQ 0$ for any $f_1$-horizontal prime divisor $T$ of $X_1$. 
Then, since $\deg_{K(A_1)}(p-1)\mathfrak F \leq 1$, no matter $p=2$ or $p=3$, we can apply
 \cite[Proposition~5.3]{CWZ23} to the pair $(X_1, (p-1) \mathfrak M +  (p-1) \mathfrak F)$ and obtain the following:
\begin{itemize}
  \item $\mathfrak M$ is semi-ample with $\nu(\mathfrak M) =1$;
  \item Denote by $g_1\colon X_1\to \PP^1$ the induced fibration by $\mathfrak M$. Then a general fiber of $g_1$ is an abelian variety and a special fiber is a multiple of an abelian variety.
\end{itemize}

From this, by a similar argument of Step~2 (c) in Theorem~\ref{thm:C-quasi-ell}, we see that there is no $f_1$-vertical part in $\mathfrak M+\mathfrak F$.


Moreover, since $X_1 \to X$ is purely inseparable there exists a fibration $g\colon X\to \PP^1$ fitting into the following commutative diagram:
  $$\xymatrix@C=2ex@R=2ex{
    &\PP^1 \ar[rr] && \PP^1\\
    X_1 \ar[rr]\ar[dd]_{f_1}\ar[ru]^{g_1} && X\ar[dd]_{f}\ar[ru]^{g} \\
    \\
    \llap{$A^{(-1)}=:\;$}A_1 \ar[rr]   && A \rlap{\,.}
  }$$

\smallskip
Let $G_1$ be a general fiber of $g_1$.
Then $\deg _{K(A_1)} G_1 =1$ or $2$ since $\deg_{K(A_1)} \mathfrak M \le 2$.
We distinguish between the following two cases:
\begin{itemize}
  \item[\rm (1)] $\mathrm{deg}_{K(A_1)} G_1 = 1$;
  \item[\rm (2)] $\mathrm{deg}_{K(A_1)} G_1 = 2$, which happens only when $p=2$.
\end{itemize}

In case~(1), the general fiber $G_1$ of $g_1$ maps to $A_1$ isomorphically.
Thus, $g_1$ has no multiple fibers, and it follows that every fiber of $g_1$ is isomorphic to $A_1$.
Then, the induced morphism $X_1 \xrightarrow{(f_1,g_1)} A_1\times\PP^1$ is an isomorphism by Zariski's main theorem.
Therefore, $X\cong A_1\times\PP^1/\mathcal F$ for some height one foliation $\mathcal F$.
By $0\simQ\pi_1^*K_{X} \sim K_{A_1\times\PP^1} - (p-1)\det\mathcal F$, we have
\begin{equation}\label{eq:9388}
  \det \mathcal F\sim \mathrm{pr}_2^*\mathcal O_{\PP^1}(-2/(p-1)).\end{equation}
Moreover, since $\deg (X_1\to X) < \dim A$, we have $\mathop{\rm rank} \mathcal F < \dim A$.

\smallskip
In case~(2), we fall into one of the following three subcases:
\begin{itemize}
  \item[(2.a)] $f_1$ is separable and $G_1\to A_1$ is \'etale.
  \item[(2.b)] $f_1$ is separable and $G_1\to A_1$ is purely inseparable.
  \item[(2.c)] $f_1$ is inseparable.
\end{itemize}

In case (2.a), since the general fiber $G_1$ of $g_1$ are \'etale of degree $2$ over $A_1$, they are isomorphic to each other.
Do the \'etale base change $A_2:= G_1 \to A_1$. Then $X_2=X_1\times_{A_1} A_2$ is normal, and the Stein factorization of $X_2 \to X_1\xrightarrow {g_1} \PP^1$ gives a fibration $g_2\colon X_2 \to \PP^1$. To summarize, we have the following commutative diagram:
\begin{equation*}
  \xymatrix@C=1.5ex@R=2ex{
    & \PP^1\ar[rr]^{\delta} && \PP^1\\
    \llap{$X_1\times_{A_1} A_2 =: \;$} X_2 \ar[ru]^{g_2} \ar[rr]^<<<<<<{\pi_2} \ar[dd]_{f_2} && X_1 \ar[dd]_{f_1} \ar[ur]^<<{g_1} \\
    \\
  \llap{$G_1 =: \;$}A_2 \ar[rr]^{\tau_2}    && A_1 \rlap{\,.}
}
\end{equation*}
Note that the scheme $G_1\times_{A_1} A_2$ has two disjoint irreducible components $G_2',G_2''$, i.e., $\pi_2^*G_1 \sim G_2' + G_2''$.
Thus, $\delta$ is a separable morphism of degree two.
Therefore, each fiber $G_2$ of $g_2$ is mapped birationally to $A_2$ via $f_2$.
By $(K_{X_2} + G_2)|_{G_2^\nu} \simQ 0$, applying Lemma~\ref{lem:adj-charA}, we see that $G_2$ is an abelian variety, hence $G_2\to A$ is an isomorphism. It follows that the morphism
$X_2 \xrightarrow{(f_2,g_2)}  A_2 \times \PP^1$ is an isomorphism.

\smallskip
We use a similar argument for case (2.b).
Take $\tau_2\colon A_2:=A_1 \buildrel F\over\to A_1$ to be the Frobenius morphism, which factors through $G_1\to A_1$ for a general fiber $G_1$ of $g_1$ (\cite[page~169, Remark]{MumfordAV}).
Let $\nu\colon X_2 \to X_1 \times_{A_1} A_2$ be the normalization morphism. Since $f_1$ is generically smooth, the projection $f_2\colon  X_2 \to A_2$ is in fact a fibration. Notice that $G_1\times_{A_1} A_2$ is non-reduced. Then by $\deg_{K(A_2)}G_1\times_{A_1} A_2=2$, we conclude $\pi_2 ^* G_1 = 2G_2$ for some integral divisor $G_2$ in $X_2$.
Let $g_2\colon  X_2 \to \PP^1$ be the induced fibration from the Stein factorization of $X_2 \to X_1 \xrightarrow{g_1} \PP^1$. Then a general fiber $G_2$ is mapped to $A_2$ birationally via $f_2$.
Let $\mathcal N$ be the conductor divisor of $\nu$, which is vertical over $A_2$ since the generic fiber of $f_1$ is smooth over $K(A_1)$. Now fix a fiber $G_2^0$, and take a horizontal irreducible component $G_2'$, which is unique since $\deg_{K(A_2)}G_2^0 = 1$ and write that $G_2^0 = G_2' + G_2''$.
Then by
 $$0\simQ (\pi_2^* K_{X_1} + G_2)|_{(G_2')^\nu} \simQ (K_{X_2} + G_2' + G_2'' + \mathcal N)|_{(G_2')^\nu},$$
and applying Lemma~\ref{lem:adj-charA}, we see that $G_2\to A_2$ is an isomorphism and $G_2''|_{(G_2')^\nu} =\mathcal N|_{(G_2')^\nu} = 0$. It follows that that $\Supp \mathcal N \cap \Supp G_2^0 = \emptyset$ and $G_2''=0$ since $G_2^0$ is connected. In turn we conclude that $\mathcal N=0$. As before, we show that the morphism $X_2 \xrightarrow{(f_2,g_2)}  A_2 \times \PP^1$ is an isomorphism.

\smallskip
Finally, we deal with the case (2.c).
Take $\tau_2\colon A_2:=A_1 \buildrel F\over\to A_1$ to be the Frobenius morphism, and let $X_2 := (X_1\times_{A_1} A_2)^\nu_{\rm red}$.
We have the following diagram
\[\xymatrix{
    X_2 \ar[r] \ar@/^2pc/[rr]|{\,\pi_2\,} \ar[dr]^{f_2} & X_1\times_{A_1} A_2 \ar[r]\ar[d] & X_1 \ar[d]^{f_1}\\
                                                     &A_2 \ar[r]^{\tau_2}   & A_1 \rlap{\,.}
  }
\]
Since $f_1$ is inseparable, there exist a moving linear system $\mathfrak N$ and $V_2\geq 0$ on $X_2$ such that \begin{equation}
  K_{X_2} + \mathfrak N + V_2\sim \pi_2^* K_{X_1}.
\end{equation}
We have that $K_{X_2} + \mathfrak N + V_2 +\pi_2^*\mathfrak M \simQ \pi_2^* \pi_1^* K_X \simQ 0$.
By $\deg_{K(A_2)} (K_{X_2}) = -2$, we see that $\deg_{K(A_2)} (\mathfrak N) =\deg_{K(A_2)} (\pi_2^*\mathfrak M) =1$. Fix a divisor $N_0\in\mathfrak N$, and take an irreducible $f_2$-horizontal component $N'$ and write that $N_0 = N' +N''$. Arguing as in the case (2.b), we can prove that $N_0 = N'$ and $V_2=0$. In turn we can prove that $X_2 \cong A_2 \times \PP^1$.
\end{proof}

\section{Irregular $K$-trivial threefolds}\label{sec:3-folds}%
In this section we focus on irregular $K$-trivial threefolds.

\subsection{Structure theorem}\label{sec:thm-3fold} 
\begin{thm}\label{thm:3fold}
  Let $X$ be a normal $\QQ$-factorial projective threefolds with $K_X\equiv0$. Denote the Albanese morphism of $X$ by $a_X\colon X\to A$.
  Assume $\dim a_X(X) > 0$. Then $X$ can be described as follows.
  \begin{itemize}
    \item[\rm(A)] If $\dim a_X(X) =3$, then $X$ is an abelian variety.
    \item[\rm(B)] If $\dim a_X(X) =1$, under the condition that 
      \begin{itemize}
        \item either {\rm(i)} $X$ is strongly $F$-regular and $K_X$ is $\ZZ_{(p)}$-Cartier;
        \item or {\rm(ii)} $X$ has at most terminal singularities and $p \geq 5$,
      \end{itemize}
      the Albanese morphism $a_X$ is a fibration and there exists an isogeny of elliptic curves $A'\to A$,
      such that  $X\times_A A' \cong A'\times F$, where $F$ is a general fiber of $a_X$. More precisely, $X\cong A'\times F/H$,
      where $H$ is a finite group subscheme of $A'$ acting diagonally on $A'\times F$.
    \item[\rm(C)] If $\dim a_X(X) =2$, then $a_X$ is a fibration and one of the following holds:
      \begin{itemize}
        \item[\rm(C1)] $a_X$ is a smooth fibration and there exists an isogeny of abelian surfaces $A'\to A$,
          such that  $X\times_A A' \cong A'\times E$, where $E$ is an elliptic curve appearing as a general fiber of $a_X$. More precisely, $X\cong A'\times E/H$,
          where $H$ is a finite group subscheme of $A'$ acting diagonally on $A'\times E$, with a full classification as in Section~\ref{sec:explicite-of-C1}.
        \item[\rm(C2)] $p=2$ or $3$, and $a_X$ is a quasi-elliptic fibration. Set $X_1$ to be the normalization of $X \times_{A} A^{(-1)}$ and $f_1\colon X_1\to A_1 :=A^{(-1)}$ the induced morphism.
          Then  $f_1\colon X_1\to A_1$ is a smooth fibration fibred by rational curves, which falls into one of the following two specific cases:
          \begin{itemize}
            \item[\rm(2.1)] In the first case, $f_1\colon X_1\to A_1 $ has a section, and
              \begin{itemize}
                \item[\rm(2.1a)] either $X_1\cong \mathbb{P}_{A_1}(\mathcal{O}_{A_1}\oplus \mathcal{L})$, where $\mathcal{L}^{\otimes p+1}\cong \mathcal O_X$; or
                \item[\rm(2.1b)] $X_1\cong \mathbb{P}_{A_1}(\mathcal{E})$, where $\mathcal{E}$ is a unipotent vector bundle of rank two, and
                  there exists an \'etale cover $\mu\colon A_2\to A_1$ of degree $p^v$ for some $v\le 2$ such that $\mu^*F^{(2-v)*}_{A_1/k}\mathcal{E}$ is trivial.
              \end{itemize}
            \item[\rm(2.2)]
              In the second case, $p=2$, and there exists a purely inseparable isogeny $A_2 \to A_1$ of degree two, such that $X_2 :=X_1\times_{A_1}A_2$ is a projective bundle over $A_1$ described as follows
              \begin{itemize}
                \item[\rm(2.2a)] $X_2\cong \mathbb{P}_{A_2}(\mathcal{O}_{A_2}\oplus \mathcal{L})$, where $\mathcal{L}^{\otimes 4}\cong \mathcal{O}_{A_2}$; or
                \item[\rm(2.2b)] $X_2\cong \mathbb{P}_{A_2}(\mathcal{E})$, where $\mathcal{E}$ is a unipotent vector bundle of rank two,
                  and there exists an \'etale cover $\mu\colon A_3\to A_2$ of degree $p^v$  for some $v\le 2$ such that $\mu^*F^{(2-v)*}_{A_1/k}\mathcal{E}$ is trivial.
              \end{itemize}
          \end{itemize}
        \item[\rm(C3)] $p=2$ or $3$, and $a_X$ is inseparable.
          Set $X_1=(X \times_{A} A_1(:={A^{(-1)}}))_{\rm red}^{\nu}$. Then the projection $X_1 \to A_1$ is a smooth morphism, and
          \begin{itemize}
            \item[\rm(3.1)] either $X_1=A_1\times\mathbb P^1$ and $X=A_1\times\mathbb P^1/\mathcal F$ for some smooth rank~one foliation $\mathcal F$, which can be described concretely (see Section~\ref{sec:descrip-C31}); or
            \item[\rm(3.2)] $p=2$, and there exists an isogeny of abelian surfaces $\tau\colon A_2 \to A_1$ such that $X_1\times_{A_1}A_2 \cong A_2\times\mathbb P^1$, where either
              \begin{itemize}
                \item[\rm(3.2a)] $\tau\colon A_2 \to A_1$ is an \'etale of degree two; or
                \item[\rm(3.2b)] $\tau=F_{A_1/k}\colon A_2 := {A_1^{(-1)}} \to A_1$ is the relative Frobenius.
              \end{itemize}
          \end{itemize}
      \end{itemize}
  \end{itemize}
\end{thm}
\begin{proof}
When $\dim a_X(X) =3$, we can apply Proposition~\ref{char-abel-var} and show that $X \cong A$. This is Case (A).

When $\dim a_X(X) =1$, we apply Theorem~\ref{thm:q1} and show that $q(X)=1$ and $X$ is described as in Case (B).

At last, we consider the case $\dim a_X(X) = 2$. When  $a_X\colon X \to A$ is separable, we can apply Theorem~\ref{thm:elliptic-3fold} and \ref{thm:C-quasi-ell} to obtain the cases (C1) and (C2) respectively. When $a_X\colon X \to A$ is inseparable, we can apply Theorem~\ref{thm:C-nonreduced} to obtain (C3). Note that  as $\dim A=2$, by Remark~\ref{rmk:pdeg2}, we see that the generic fiber of $X_1\to A_1$ is geometrically normal, and thus Case (2.c) in Theorem~\ref{thm:C-nonreduced} does not occur here.
\end{proof}


\subsection{Concrete description and examples}%
In this subsection, we first describe concretely the foliation in case (C3.1), then give examples of case (C3.2) and (C2).
\subsubsection{Notation}\label{sec:nota-basis}
(1) On the projective line $\PP^1$, we fix an affine open cover $\PP^1 = \AA^1_{(t)} \cup \AA^1_{(s)}$, with $s=1/t$.
We identify the  line bundle $\mathcal O_{\PP^1}(i)$ as
\[\mathcal O_{\mathbb P^1}(i)|_{\mathbb A^1_{(t)} } = \mathcal O_{\AA^1_{(t)}}\cdot 1,\quad
  \mathcal O_{\mathbb P^1}(i)|_{\mathbb A^1_{(s)} } = \mathcal O_{\AA^1_{(s)}}\cdot \frac1{s^i};
\]
and the twisted tangent bundle $\mathcal T_{\PP^1}(i)$ as \[
  \mathcal T_{\mathbb P^1}(i)|_{\mathbb A^1_{(t)} } = \mathcal O_{\mathbb A^1_{(t)}}\cdot \partial_t,\quad
  \mathcal T_{\mathbb P^1}(i)|_{\mathbb A^1_{(s)} } = \mathcal O_{\mathbb A^1_{(s)}}\cdot \frac{1}{s^{i}}\partial_s,\quad
  \text{with \ } \p_t = s^2\p_s.
\]

(2) Let $A$ be an abelian surface. We can choose a basis $\alpha,\beta$ of Lie algebra $\mathop{\rm Lie} A$ which falls into one of the following cases:
\begin{itemize}
  \item[(i)]   $\alpha^p= \beta^p=0$ (superspecial);
  \item[(ii)]  $\alpha^p = \beta$ and $\beta^p=0$ (supersingular, not superspecial);
  \item[(iii)] $\alpha^p = \alpha$ and $\beta^p=0$ ($p$-rank one);
  \item[(iv)]  $\alpha^p = \alpha$ and $\beta^p=\beta$ (ordinary).
\end{itemize}
 Then $\mathcal T_A = \mathcal O_A \cdot \alpha \oplus \mathcal O_A \cdot \beta$.

\subsubsection{A concrete description of foliation in Case \rm(C3.1)} \label{sec:descrip-C31}%
In this case, $X \cong A\times\PP^1/\mathcal F$ where $\mathcal F \subset \mathcal T_{A\times\PP^1}$ is a rank one foliation.
We shall give a concrete description of $\mathcal F$.

As $\mathcal F$ is reflexive and of rank one on a factorial variety, it is locally free (\cite[Proposition~1.9]{Hartshorne80}).
By Equation~(\ref{eq:9388}), we have $\mathcal F \cong \det\mathcal F \cong \mathrm{pr}_2^*\mathcal O_{\PP^1}(-i)$, where $i=1$ if $p=3$ and $i=2$ if $p=2$, therefore the inclusion $\mathcal F \subset  \mathcal T_{A\times\PP^1}$ is determined by a non-zero element (unique up to scaling)
  \begin{equation}\label{eq:O8Q8}
    \alpha_\mathcal F \in \mathrm{Hom}(\mathrm{pr}_2^*\mathcal O_{\mathbb P^1}(-i), \mathcal T_{A\times\PP^1}) \cong
    H^0( A\times\PP^1, \mathcal T_A(i) \oplus \mathrm{pr}_2^* \mathcal T_{\PP^1}(i) ),
  \end{equation}
  where $\mathcal T_A(i) := \mathrm{pr}_1^*\mathcal T_A \otimes \mathrm{pr}_2^*\mathcal O_{\PP^1}(i)$.
  Using the basis of $\mathcal O_{\PP^1}(i)$, $\mathcal T_{\PP^1}(i)$ and $\mathcal T_A$ given in Section~\ref{sec:nota-basis},
  the element $\alpha_{\mathcal F}$, over the subset $A \times \mathbb A^1_{(t)}$, can be written as
  \begin{equation}\label{eq:XAAM}
    D = u \alpha + v \beta + w\partial_t,
  \end{equation}
  where $u,v,w\in k[t]$ with $\deg u, \deg v\leq i$, $\deg w \leq i+2$.



Since $\mathcal F$ is saturated, we have $\gcd(u,v,w) = 1$, thus the ideal $(u,v,w) = k[t]$.
From this we conclude that on $A \times \mathbb A^1_{(t)}$, the quotient $\mathcal T_{A\times\PP^1}/\mathcal F$ is also locally free, thus the foliation $\mathcal F$ is smooth.
For the same reason, $\mathcal F$ is smooth on  $A \times \mathbb A^1_{(s)}$. To summarize, we have

\begin{prop}\label{non-red-I-smooth}
  The foliation $\mathcal F$ is smooth, and thus $X$ is a smooth variety.\qed
\end{prop}

\paragraph{\underline{The case $p=3$}}
In this case, $\deg u(t)\le1$, $\deg v(t)\le1$, $\deg w(t)\le3$, and we can write
\[
 D= (a_1t + a_0)\alpha + (b_1t + b_0)\beta + (c_3 t^3 + c_2 t^2 + c_1 t + c_0)\partial_t.
\]
The rank one subsheaf $\mathcal F$ is a foliation if and only if $\mathcal F^p \subseteq \mathcal F$, which is equivalent to the condition $D^p = \lambda D$ for some $\lambda \in k[t]$.
By Proposition~\ref{prop:push-foliation}, the Albanese morphism $a_X\colon X\to A_X$ is inseparable if and only if
\[
  \Delta:= a_1b_0-a_0b_1\ne0 .
\]
Therefore, we can characterize the foliation $\mathcal F$ as the following:
\begin{itemize}
  \item[$(\clubsuit)$] $\Delta\ne0$, and $D^p = \lambda D$ for some $\lambda \in k[t]$.
\end{itemize}
By a direct calculation, we have \begin{equation*}\label{eq:T07C}
  \begin{aligned}
    D^3 =\bigl(u(t)\alpha + v(t) \beta + w(t) \partial_t\bigr)^3 = u^3 \alpha^3 + v^3 \beta^3 + ww'u'\alpha + ww'v'\beta+(ww'^2 + w^2w'') \partial_t.
  \end{aligned}
\end{equation*}
Then we can translate the condition ($\clubsuit$) into the following in each of the cases (i-iv) of $\mathop{\rm Lie} A$ in Section~\ref{sec:nota-basis}(2):

(i) $\Delta\ne0$ and $w'=0$.

\smallskip
(ii)
Invalid.

\smallskip
(iii)
$\displaystyle \Delta\ne0,\
c_1 c_3 - c_2^2 = {-b_0 a_1^3 \over \Delta},\
c_2 c_3 = {b_1 a_1^3 \over \Delta},\
c_0 c_2 - c_1^2 = {b_1 a_0^3\over \Delta},\
c_0 c_1 ={-b_0 a_0^3\over \Delta}.
$

\smallskip
(iv)
$\displaystyle \Delta\ne0,\
  c_1c_3 - c_2^2  = {a_0b_1^3 - b_0a_1^3\over \Delta},\
  c_2c_3 = {a_1^3b_1 - a_1b_1^3\over \Delta},\
  c_0 c_2 - c_1^2 = {a_0^3b_1 - b_0^3a_1\over \Delta},\
  c_0c_1 = {a_0b_0^3 - b_0a_0^3\over \Delta}.\
$

\medskip
\paragraph{\underline{The case $p=2$}}
In this case, $\deg u(t)\le2$, $\deg v(t)\le2$, $\deg w(t)\le4$ and we can write \[
    D = (a_2 t^2 + a_1 t + a_0) \alpha + (b_2t^2 + b_1t + b_0) \beta + (c_4t^4 + c_3t^3 + c_2t^2 + c_1t + c_0) \partial_t,
\]
which satisfies the following conditions
\begin{itemize}
  \item[$(\spadesuit)$]
    $
    \begin{cases}
     \hbox{$D^p = \lambda D$ for some $\lambda \in k[t]$;}\\
     \hbox{one of $\Delta_{01}:= a_0b_1 + a_1b_0$, $\Delta_{12} := a_1b_2 + a_2b_1$ and $\Delta_{02} := a_0b_2 + a_2b_0$ is nonzero;}\\
     \hbox{$\gcd(u,v,w) = 1$ and $(a_2,b_2,c_4)\ne0$.}
    \end{cases}
    $
\end{itemize}
The conditions ($\spadesuit$) can be translated into the following in each of the cases (i-iv) of $\mathop{\rm Lie} A$ in Section~\ref{sec:nota-basis}(2):

\smallskip
(i) either (1) $w=0$, $(\Delta_{01},\Delta_{12},\Delta_{02})\ne0$, $\gcd(u,v) = 1\ne0$ and $(a_2,b_2)\ne0$; or
           (2) $u'=v'=w'=0$ and $\Delta_{02}\ne0$.

(ii)
$\gcd(u,v,w)=1$, $(a_2,b_2)\ne0$, $(\Delta_{01},\Delta_{12})\ne0$ and
\[\setbox\strutbox=\hbox{\vrule height15pt depth8pt width0pt}
\begin{tabular}{|c|ccccc|}
  \hline
   & $c_0$ & $c_1$ &$c_2$ &$c_3$ &$c_4$\\
  \hline
  \strut$\Delta_{01}\ne0, \Delta_{12}\ne0$ &
    {$a_0^3\over \Delta_{01}$} &{$a_0^2a_1\over\Delta_{01}$}&{${a_0a_1^2\over \Delta_{01}} = {a_2a_1^2 \over \Delta_{12}}$}&{$a_1a_2^2\over \Delta_{12}$}&{$a_2^3\over \Delta_{12}$}\\
  \hline
  \strut$b_1\ne0, a_1=0$, $a_0$ or $a_2=0$ & $a_0^2\over b_1$ &0&0&0& $a_2^2\over b_1$\\
  \hline
\end{tabular}
\]
By symmetric reason, we omit the case $\Delta_{01}=0$ and $\Delta_{12}\ne0$ in the following.

(iii)
$\gcd(u,v,w)=1$,  $(a_2,b_2)\ne0$, $(\Delta_{01},\Delta_{12})\ne0$ and
\[\setbox\strutbox=\hbox{\vrule height15pt depth8pt width0pt}
\begin{tabular}{|c|ccccc|}
  \hline
   & $c_0$ & $c_1$ &$c_2$ &$c_3$ &$c_4$\\
  \hline
  $\Delta_{01}\ne0, \Delta_{12}\ne0,\atop
    {b_1a_0 \Delta_{02}\Delta_{12} } + {a_2b_1 \Delta_{02}\Delta_{01}} + a_1b_1 \Delta_{01}\Delta_{12} = 0
  $ &
    \multirow{2}*{$b_0a_0^2\over \Delta_{01}$} &
    \multirow{2}*{$b_1a_0^2\over \Delta_{01}$} &
    ${b_0 a_0a_2 \over \Delta_{01}} + {a_0a_2b_2 \over \Delta_{12}} + a_1$  &
    {$b_1a_2^2\over \Delta_{12}$} &
    $b_2a_2^2\over \Delta_{12}$ \\
  $\Delta_{01}\ne0, \Delta_{12} = 0, \atop a_1b_1=a_2b_1=a_2b_2=0$ &
    &&
    ${b_0a_1^2\over \Delta_{01}} + {b_1a_0^2\Delta_{02}\over \Delta_{01}^2}$ &
    0&
    $b_0a_2^2\over \Delta_{01}$ \\
  \hline
\end{tabular}
\]

(iv)
$\gcd(u,v,w) = 1$, $(a_2,b_2)\ne0$, $(\Delta_{01},\Delta_{12})\ne0$, ${(a_1b_0^2 + b_1a_0^2)\Delta_{12}^2} + {(a_1b_2^2 + b_1a_2^2)\Delta_{01}^2} + (b_1a_1^2 + a_1b_1^2)\Delta_{01}\Delta_{12} =0$, and
\[\setbox\strutbox=\hbox{\vrule height15pt depth8pt width0pt}
\begin{tabular}{|c|c@{\hskip8pt}c@{\hskip8pt}c@{\hskip8pt}c@{\hskip8pt}c|}
  \hline
   & $c_0$ & $c_1$ &$c_2$ &$c_3$ &$c_4$\\
  \hline
  $\scriptstyle \Delta_{01}\ne0, \Delta_{12}\ne0, a_1\ne0$
  &
    \multirow{3}*{$a_0b_0^2 + b_0a_0^2\over \Delta_{01}$} &
    \multirow{3}*{$a_1b_0^2 + b_1a_0^2\over \Delta_{01}$} &
    ${a_2b_0^2 + a_2a_0 b_0\over \Delta_{01}} +  {a_0b_2^2 + a_0a_2b_2\over \Delta_{12}} + a_1$ &
    \multirow{2}*{$a_1b_2^2 + b_1a_2^2\over \Delta_{12}$} &
    \multirow{2}*{$a_2b_2^2 + b_2a_2^2\over \Delta_{12}$} \\
  $\scriptstyle \Delta_{01}\ne0, \Delta_{12}\ne0, b_1\ne0$
  & & &
    ${b_0a_0b_2 + a_0^2b_2\over \Delta_{01}} + {b_0b_2a_2 + b_0a_2^2\over \Delta_{12}} + b_1$  &
    & \\
  $\scriptstyle \Delta_{01}\ne0, \Delta_{12} = 0$ &
    &&
    ${(a_1b_0^2 + b_1a_0^2)\Delta_{02} + (a_0b_1^2 + b_0a_1^2)\Delta_{01} \over \Delta_{01}^2}$ &
    0&
    $(a_0b_2^2 + b_0a_2^2)\over\Delta_{01}$ \\
  \hline
\end{tabular}
\]

\begin{rmk}
For each valid case, it is easy to give examples.
For instance, if $p=2$, $\alpha^2=\alpha$ and $\beta^2=\beta$ (so we are in case (iv)), then
\newcommand\smallmatt[6]{\bigl({#1\atop#4}{#2\atop#5}{#3\atop#6}\bigr)}%
$\smallmatt{a_0}{a_1}{a_2}{b_0}{b_1}{b_2} = \smallmatt{1}{0}{0}{0}{1}{1}$ is a solution of $(\spadesuit)$ which corresponds to the following vector field \[
  D = \alpha + (t+t^2)\beta + (t + t^4) \p_t.
\]
Some special cases have appeared in the literature, see \cite{Moret-Bailly79}, \cite[Section~3]{Schroer04} and \cite[Example~14.1]{PZ19}.
\end{rmk}

\subsubsection{Examples of Case \rm(C3.2):}
\begin{examp}[of Case (C3.2a)]\label{exam:C32A}%
  We aim to find a free action of $G=\ZZ/2\ZZ$ on $A\times\PP^1$ and a foliation $\mathcal F$ on $X_0:=A\times\PP^1$ in a compatible way
  such that $\mathcal F$ descents to a foliation $\mathcal F_1$ on $A\times\PP^1/G$ and the quotient $X := (A\times\PP^1/G)/\mathcal F_1$ is an example of Case (C3.2a).

  Let $p=2$, and let $A$ be an ordinary abelian surface with $\alpha^p=\alpha$ and $\beta^p = \beta$ for some bases $\alpha,\beta\in\mathop{\rm Lie} A$.
  Let $G = \ZZ/2$ be the cyclic group of order $2$ with a generator $\sigma$.
  Let $P_0\in A$ be a nontrivial 2-torsion point.
  Then we can define an action of $G$ on $A\times\PP^1$ diagonally by 
  \[
    \sigma (P) = P+P_0 \hbox{ for each } P\in A,
    \hbox{ and } \sigma(t) = t+1  \hbox{ for each } [t:1]\in \PP^1.
  \]
  Let $X_1:= A\times\mathbb P^1/G$.
  Since the action of $G$ is free, $\pi_2\colon A\times\PP^1\to X_1$ is \'etale.
  Let $\mathcal F$ be the foliation on $A\times\mathbb P^1$ generated by
  \[
    \alpha + (t^2 + t)\beta + (t^4 + t)\partial_t.
  \]
  Let $Y_1=(A\times\mathbb P^1)/\mathcal F$.
  Since $K_{A\times\PP^1} = \mathrm{pr}_2^*K_{\PP^1}$ and $\det \mathcal F \cong \mathrm{pr}_2^*\mathcal O_{\mathbb P^1}(-2)$, we have $K_{Y_1}\equiv 0$.

  The foliation $\mathcal F$ is invariant under the action of $\sigma$, thus it descends to a foliation on $X_1$, denoted by $\mathcal F_1$.  Moreover, the action of $\sigma$ descends to an action $\sigma_1$ on $Y_1=X/\mathcal F$ (Lemma~\ref{lem:etaleLift}). Let
  $$X :=X_1/\mathcal F_1 \cong Y_1/\langle\sigma_1\rangle.$$
  By Remark~\ref{rmk:sep-foliation-lift}, we have the following commutative diagram: \[
    \vcenter{\xymatrix@R3ex@C3ex{
        X_0=A \times \PP^1 \ar[dd]\ar[dr]^{\mu}\ar[rr] && X_0/\mathcal F = Y_1\ar[dr]^{\mu'}\ar@{-->}[dd]|\hole^<<<<{a_{Y_1}} \\
                                               & X_1\ar[rr] \ar[dd]^<<<<<{a_{X_1}} && X_1/\mathcal F_1 = X \ar[dd]^{a_X}\\
      A \ar@{-->}[rr]|\hole \ar[rd]^{\tau}  && A^{(1)} \ar@{-->}[dr]^{\tau'}\\
                                                    & A_{X_1} \ar[rr]^{F} && A_{X_1}^{(1)} 
    }}
  \]
  where $\tau\colon A \to A_{X_1}=A/\langle\sigma\rangle$ and $\tau'$ are the étale morphisms, the left and the right squares are Cartesian and thus $\mu,\mu'$ are étale.
  As stated, we have $K_{Y_1} \equiv 0$, thus $K_{X} \equiv 0$.
  Moreover, $a_{Y_1}$ is inseparable, so is $a_X$.
\end{examp}

\begin{examp}[of Case (C3.2b)]\label{exam:two-Foliation}
  In this example, we construct a threefold $X$ by two consecutive quotients of $A\times\PP^1$ by certain foliations:
  \begin{equation}\label{eq:DFX3}
    \vcenter{\xymatrix{
      X_0 = A\times\mathbb P^1\ar[r]^>>>>>{\pi_0}\ar[d]^{f_0} & X_1 = X_0/\mathcal F_0\ar[r]^>>>>>{\pi_1}\ar[d]^{f_1} &X =X_2 = X_1/\mathcal F_1\ar[d]^{f_2}\\
      A\ar[r]^{\tau_0} &A_1\ar[r]^{\tau_1} & A_1^{(1)} \rlap{,}
    }}
  \end{equation}
   where $\tau_0$ is purely inseparable of degree $p$, $\tau_1$ is the relative Frobenius over $k$, and $f_i$ ($i=0,1,2$) are the Albanese morphisms respectively.
  We work over a base field $k$ of characteristic $2$.

  Let $E$ and $E'$ be elliptic curves with the following defining equations:
  \begin{equation}\label{eq:64Q7}
  \begin{aligned}
    E  &: y^2 + y = x^3,\\
    E' &: y'^2 + x'y' = x'^3 + 1.
  \end{aligned}
  \end{equation}
  Note that $E$ is supersingular which admits a invariant vector field $\alpha\in H^0(E,\mathcal T_E)$ such that $\alpha^p = 0$;
  and $E'$ is ordinary which admits a invariant vector field $\beta\in H^0(E',\mathcal T_{E'})$ such that $\beta^p = \beta$.
  Concretely, we can take $\alpha$ to be the dual of $dx \in H^0(E, \Omega_E^1)$ such that $\alpha(dx) = \alpha(x)=1$ and $\alpha(y) = x^2$; and $\beta$ the dual of $dx'/x' \in H^0(E, \Omega_E^1)$, so $\beta(x') = x'$, $\beta(y')=y'+x'^2$ (see \cite[Section~III.1]{SilvermanAEC}).

  Let $A = E \times E'$.
  In the following we usually describe vector fields on the given affine pieces of $E$ and $E'$ by (\ref{eq:64Q7}).  The reader can check that the vector fields involved extend to the whole variety.

With notation of Section~\ref{sec:nota-basis}, let $X_0=A\times\PP^1$ and $Y_0=E\times\PP^1$  and let $D_0= \alpha + \partial_t \in  H^0(Y_0, \mathcal T_{Y_0})$. We may also regard $D_0$ as an element of $H^0(X_0, \mathcal T_{X_0}) \cong  H^0(Y_0, \mathcal T_{Y_0}) \oplus  H^0(E', \mathcal T_{E'})$.
Then $D_0^p=0$, and $D_0$ extends to a global vector field on $X_0=A\times\PP^1$, which is non-zero everywhere.
Let $\mathcal F_0$ and $\mathcal G_0$ be the foliations generated by $D_0$ on $X_0$ and $Y_0$ respectively.
Let $Y_1=Y_0/\mathcal G_0$ and $X_1=X_0/\mathcal F_0\cong Y_0\times E'$, which are both nonsingular.
By construction  $X_1$ is equipped with two fibrations $f_1\colon X_1 \to A_1= E^{(1)} \times E'$ and $g_1\colon X_1 \to B=\mathbb{P}^1$, which fit into the following commutative diagrams
  \[
    \xymatrix{
      X_0 = A\times\mathbb P^1\ar[r]^>>>>>{\pi_0}\ar[d]^{f_0=\mathrm{pr}_1} & X_1 = X_0/\mathcal F_0\ar[d]^{f_1} &  & X_0\ar[d]^{g_0=\mathrm{pr}_2}\ar[r]  &X_1 \ar[d]^{g_1}\\
      A\ar[r]^{\tau_0} &A_1 &  &\mathbb P^1 \ar[r]^<<<<<{F} &B= (\mathbb P^1)^{(1)} .
    }
  \]
Here we have an open covering $(\mathbb P^1)^{(1)} = \mathbb A^1_{(t^p) }\cup \AA^1_{(s^p)}$. Taking the corresponding inverse images of $g_0, g_1$ we obtain
the coverings
$$X_0 = A\times \PP^1 = (X_{0,t}:=A \times \AA^1_{(s)})\cup (X_{0,s}:=A \times \AA^1_{(s)}),~~X_1 = X_{1,t^p} \cup X_{1,s^p}~\mathrm{and}~Y_1 = Y_{1,t^p} \cup Y_{1,s^p}.$$
We may identify
$$\Gamma(Y_{1,t^p}, \mathcal{O}_{Y_{1,t^p}})= R=k[x_1=x^2,y_1=y^2,t_1=t^2, u=x+t]/(y_1^2 + y_1 - x_1^3, u^2-(x_1+t_1)),$$
then $\Omega_{R/k} = R\cdot d t_1 \bigoplus R\cdot du$. Let $\alpha_1 \in H^0(Y_{1,t^p}, \mathcal T_{Y_{1,t^p}})$ be determined by $\alpha_1(dt_1) = \alpha_1(dx_1) =t_1,   \alpha_1(du) = 0$. On the other piece $Y_{1,s^p}$, we have
$$
\displaylines{
  \Gamma(Y_{1,s^p}, \mathcal{O}_{Y_{1,s^p}})\cong k[x_1=x^2,y_1=y^2, s_1=s^2, v=s^2x+s]/(y_1^2 + y_1 - x_1^3, v^2 - (s_1^2x_1 + s_1)),\cr
  \Omega_{R'/k} = R' \cdot dx_1 \oplus R' d v,\;
  \hbox{and } \alpha_1(x_1)= 1/s_1, \alpha_1(v) = \alpha_1(s_1u) = v.
}
$$
We abuse notation $\alpha_1 \in H^0(X_1, \mathcal T_{X_0}) \cong  H^0(Y_1, \mathcal T_{Y_1}) \oplus  H^0(E', \mathcal T_{E'})$ for the lifting of $\alpha_1 \in H^0(Y_1, \mathcal T_{Y_1})$ and $\beta_1$ for the lifting of $\beta \in H^0(E',\mathcal T_{E'})$.

Set  $D_1 = \alpha_1 + \beta_1 \in H^0(X_1, \mathcal T_{X_1})$.
Clearly, $D_1^2 = D_1$, thus $D_1$ determines a rank one foliation $\mathcal F_1$ on $X_1$. More precisely, on $X_{1,t}$,
applying Jacobi criterion, we see that $(u,t_1,x')$ form a local coordinate for $X_{1,t^p}$ and
$$  D_1  = t_1  \frac{\p}{\p t_1} + x' \frac{\p}{\p x'}.$$
Since the vector field $D_1$ has no zeros on $X_{1,t}$, we have $\mathcal F_1|_{X_{1,t^p}} = \mathcal{O}_{X_{1,t^p}}\cdot D_1$. While on  $X_{1,s}$, the set of functions $\{ v,x_1,x' \}$ forms a local coordinate and correspondingly
$$D_1  = \frac1{s_1} \biggl(\frac{\p}{\p x_1} + s_1 v \frac{\p}{\p v} + s_1x' \frac{\p}{\p x'}\biggr).$$
We see that $\mathcal F_1|_{X_{1,s^p}} = \mathcal{O}_{X_{1,s^p}}\cdot s_1D_1$. Thus $\mathcal F_1 \sim g_1^* \mathcal{O}_{\mathbb{P}^1}(-1)$.

Let $X= X_1/\mathcal F_1$. By Proposition~\ref{prop:push-foliation}, the Albanese morphism of $X$ is $f\colon X \to A_1^{(1)}$, and therefore the diagram (\ref{eq:DFX3}) holds.
Finally $X$ is as required since \[
    \pi_0^*\pi_1^* K_{X}\sim
    \pi_0^*(K_{X_1} - (p-1)\det \mathcal F_1 ) \sim K_{X_0} + \mathrm{pr}_2^* \mathcal{O}_{\mathbb{P}^1}(-1) \sim 0.
  \]
\end{examp}

\subsubsection{Examples of case (C2): $a_X$ being a quasi-elliptic fibration}%
\begin{examp}[of (C2.1a)]\label{exam:quasi-ell-non-isotrivial}
  We give this example to show that, in the structure theorem \ref{thm:C-quasi-ell}, the second fibration $g\colon X\to \PP^1$ is not necessarily isotrivial.
  Recall that, in this case, the general fiber $C$ of the Albanese fibration $f\colon X\to A$ is a rational curve with a cusp.
  Therefore, there does not exists an isogeny $A'\to A$ and a diagonal group action of $G$ on $A'\times C$ such that $X \cong A'\times C/G$.

 Assume $p = 2$.
  Let $E$ be an supersingular elliptic curve with $\alpha\in\mathop{\rm Lie} E$ a nonzero vector field.
  Let $E'$ be a copy of $E$ over which we use $\beta$ to denote the same vector field $\alpha$.
  Let $A = E\times E'$.
  Then $\{\alpha,\beta\}$ form a basis of $\mathop{\rm Lie} A$.
  Let $\mathcal F$ be the rank two foliation on $A\times  B(:=\PP^1)$, which is generated by
  $$
  \begin{cases}
    t^2 \alpha + \beta,\\     \beta + \p_t,
  \end{cases}
  $$
  on $A\times\mathbb A^1_{(t)}$ and  by $\alpha + s^2\beta,     \beta + s^2\p_s$  on $A\times\mathbb A^1_{(s)}$. We see that $\det \mathcal F= \mathrm{pr}_2^*\mathcal{O}_{\PP^1}(-2)$ on $A\times B$.  Let $X = A\times B/\mathcal F$. Then we have $K_X \simQ 0$.
  Thus, the generic fiber of $a_X\colon X\to A^{(1)}$ is of arithmetic genus one, so it is geometrically non-normal; and by Proposition~\ref{prop:push-foliation}, $a_X$ is separable, thus $a_X$ forms a quasi-elliptic fibration.
  Since $t\not\in\mathop{\rm Ann}(\mathcal F)$, we have a fibration $g\colon X\to B'= B^{(1)}\cong \PP^1$.
  For $b=[t^2:1] \in B'$, the fiber $X_b$ of $g$ is isomorphic to $A/(t^2\alpha + \beta)$. In particular, two general closed fibers of $g$ are not isomorphic to each other, which implies that $g\colon X\to \PP^1$ is not an isotrivial fibration.
\end{examp}

\begin{examp}[{of (C2.1b)}]
  Assume $p = 2$.  Let $A$ be an ordinary abelian surface.
  We take a basis $\alpha,\beta$ of $\mathop{\rm Lie} A$ such that $\alpha^2 = \alpha$, $\beta^2 = \beta$.
  Let $G = \langle \sigma \rangle \cong \ZZ/2$ be the cyclic group of order $2$ which acts on $A \times \PP^1$ as in Example~\ref{exam:C32A}:
  \[
    \sigma (P) = P+P_0 \hbox{ for each } P\in A,
    \hbox{ and } \sigma(t) = t+1  \hbox{ for each } [t:1]\in \PP^1,
  \]
  where $P_0\in A$ is a nontrivial 2-torsion point.
  Let $\mathcal F$ be the saturated subsheaf of $\mathcal T_{A\times \PP^1}$ generated by \[
  \begin{cases}
    \alpha + (t^{4} + t) \p_t,\\
    \alpha + \beta.
  \end{cases}
  \]
  The sheaf $\mathcal F$ is $p$-closed so it is a foliation. We have $\det \mathcal F \cong \mathrm{pr}_2^*\mathcal{O}_{\PP^1}(-2)$.
  Also, $\mathcal F$ is invariant under the action of $\sigma$, thus it descends to a foliation $\mathcal G$ on $Y = (A\times\PP^1)/\langle\sigma\rangle$; or equivalently, the action of $\sigma$ descends to an action on $ (A\times\PP^1)/\mathcal F$ (Lemma~\ref{lem:etaleLift}).
  Let $X= Y/ \mathcal G \cong (A\times\PP^1)/\mathcal F/\langle\sigma\rangle$. Then $K_X \equiv 0$.
  Finally, by Proposition~\ref{prop:push-foliation}, $a_X\colon X\to A^{(1)}$ is separable, thus $a_X$ forms a quasi-elliptic fibration; for which the non-smooth locus of $a_X$ is located at ${t=\infty}$.
  This gives an example of Case (C2.1b).
\end{examp}



\subsection{Effectivity of the pluricanonical maps of threefolds in case $q=2$}\label{sec:effectivity}%
\smallskip
{\it Proof of Corollary~\ref{cor:effectivity}.}
We discuss separately the cases (C1), (C2) and (C3) according to Theorem~\ref{thm:3fold}.

(C1) By the classification of Section~\ref{sec:explicite-of-C1}, the torsion order of $K_X$ is exactly the same as \cite[page~37]{BM77}; namely
\begin{align*}
    \hbox{the torsion order of $K_X$} & = 2,3,4,6 \hbox{ in cases a), b), c), d) and $p\ne 2,3$}\\
                          & = 1,3,1,3             \hbox{ in cases a), b), c), d) and $p =  2$}\\
                          & = 2,1,4,2             \hbox{ in cases a), b), c), d) and $p =  3$}.
\end{align*}

(C2) We shall use the following statement.

\begin{lem}\label{lem:norm-sec}
  Let $\pi\colon Y\to X$ be a finite dominant morphism of degree $d$ between normal varieties, and let $L$ be a Weil divisor on $X$. Assume $H^0(Y,\pi^*L) \neq 0$.
  Then
  \begin{itemize}
    \item[\rm(1)] We have $H^0(X, dL) \neq 0$; in particular, if $\pi^*L\sim 0$, then $dL \sim 0$.
    \item[\rm(2)] If $\pi$ is purely inseparable of height $r$, then $H^0(X, p^rL) \neq 0$.
  \end{itemize}
\end{lem}
\begin{proof}
  (1) On a normal variety, the global sections of a reflexive sheaf are the same as the sections over a big open subset.  Thus, by restricting on the regular locus, we may assume that $X$ and $Y$ are regular.  Then given a section $s\in H^0(Y,\pi^*L)$, taking its norm gives a section of $H^0(X, dL)$.

  (2) If $\pi$ is purely inseparable of height $r$, then the $r$-th Frobenius of $X$ factors through $\pi$. So the assertion follows.
\end{proof}

\smallskip
We only show how to treat the case (2.2b), because the other cases are similar and easier. We have the following commutative diagram \[
  \xymatrix{
    \llap{$A_3\times\PP^1 \cong\;$} X_3 \ar[r]^<<<<<<{\pi_3} \ar[d] & X_1 = (X\times_A A_1)^\nu \ar[r]^<<<<<{\pi_1} \ar[d] & X \ar[d] \\
    A_3\ar[r] & A_1 = {A^{(-1)}} \ar[r]^<<<<<<<<F & A \rlap{\,,}
  }
\]
where $A_3\to A_1$ is a composition of three morphisms\[
  \xymatrix@C=18mm{
    A_3 \ar[r]^{\text{\'etale of}}_{\text{degree $p^r$}}
    & A_2 \ar[r]^{F^{2-r}} & A_2\ar[r]^{\text{inseparable}}_{\text{of degree $2$}}
    & A_1 ,
  }\quad r=0,1\text{ or }2.
\]
Thus $\pi_3$ is a composition of three morphisms in the similar way. Therefore, $X_3\to X$ is a composition of an \'etale morphism  of degree $p^r$ and an purely inseparable morphism of height $4-r$.

Since
\[
    \pi_3^*\pi_1^*K_X =  \pi_3^*(K_{X_1} + \mathcal C)
                      = \pi_3^*(K_{X_1/A_1} + \mathcal C)
                      = K_{X_3/A_3} + \pi_3^*\mathcal C \sim 0,
\]
we conclude that $2^{4}K_X\sim 0$ by Lemma~\ref{lem:norm-sec}.

We list the results as follows, leaving the detailed computation of other cases to the reader:
\[
\text{the torsion order of } K_X \text{ is a divisor of \,}
\begin{cases}
	(p+1)p & \text{in case (2.1a) for } p=2 \text{ or } 3, \\
	p^3 & \text{in case (2.1b) for } p=2 \text{ or } 3, \\
	 4\cdot 2^2 & \text{in case (2.2a) for } p=2, \\
	2^4 & \text{in case (2.2b) for } p=2.
\end{cases}
\]

(C3) Since only purely inseparable and \'etale base change are involved in Theorem~\ref{thm:C-nonreduced}, we see that $3K_X \sim 0$ if $p=3$ and $4K_X \sim 0$ if $p=2$.
\qed

\bibliographystyle{amsalpha}

\begin{thebibliography}{CHMS14}

\bibitem[B{\u a}d01]{BadescuAS}
Lucian B{\u a}descu.
\newblock {\em Algebraic surfaces}.
\newblock Universitext. Springer-Verlag, New York, 2001.
\newblock Translated from the 1981 Romanian original by Vladimir Ma\c {s}ek and revised by the author.

\bibitem[BdF07]{B-DF}
G.~Bagnera and M.~de~Franchis.
\newblock Sopra le superficie algebraich che hanno le coordinate del punto generico esprimibili con funzioni meromorfe quadruplamente periodiche di due parametri.
\newblock {\em Rom. Acc. L. Rend. (5)}, 16(1):492--498, 596--603, 1907.

\bibitem[Bea83a]{Beauville83Rmk}
Arnaud Beauville.
\newblock Some remarks on {K}\"ahler manifolds with {$c\sb{1}=0$}.
\newblock In {\em Classification of algebraic and analytic manifolds ({K}atata, 1982)}, volume~39 of {\em Progr. Math.}, pages 1--26. Birkh\"auser Boston, Boston, MA, 1983.

\bibitem[Bea83b]{Beau83}
Arnaud Beauville.
\newblock Vari\'et\'es {K\"ahleriennes} dont la premi\`ere classe de {Chern} est nulle.
\newblock {\em J. Differential Geom.}, 18(4):755--782 (1984), 1983.

\bibitem[Bea96]{Beau96}
Arnaud Beauville.
\newblock {\em Complex algebraic surfaces}, volume~34 of {\em London Mathematical Society Student Texts}.
\newblock Cambridge University Press, Cambridge, 2 edition, 1996.
\newblock Translated from the 1978 French original by R. Barlow, with assistance from N. I. Shepherd-Barron and M. Reid.

\bibitem[Bir16]{Bir16}
C.~Birkar.
\newblock Existence of flips and minimal models for 3-folds in char {$p$}.
\newblock {\em Ann. Sci. \'Ec. Norm. Sup\'er. (4)}, 49(1):169--212, 2016.

\bibitem[BM76]{BM76}
E.~Bombieri and D.~Mumford.
\newblock Enriques' classification of surfaces in char. $p$. {III}.
\newblock {\em Invent. Math.}, 35:197--232, 1976.

\bibitem[BM77]{BM77}
E.~Bombieri and D.~Mumford.
\newblock Enriques' classification of surfaces in char. {$p$}. {II}.
\newblock In {\em Complex analysis and algebraic geometry}, pages 23--42. Iwanami Shoten Publishers, Tokyo, 1977.

\bibitem[Bog74]{Bog75}
F.~A. Bogomolov.
\newblock K\"{a}hler manifolds with trivial canonical class.
\newblock {\em Izv. Akad. Nauk SSSR Ser. Mat.}, 38:11--21, 1974.

\bibitem[CHMS14]{CHMS14}
Paolo Cascini, Christopher Hacon, Mircea Musta{\c{t}}{\u{a}}, and Karl Schwede.
\newblock On the numerical dimension of pseudo-effective divisors in positive characteristic.
\newblock {\em Amer. J. Math.}, 136(6):1609--1628, 2014.

\bibitem[CWZ23]{CWZ23}
Jingshan Chen, Chongning Wang, and Lei Zhang.
\newblock On canonical bundle formula for fibrations of curves with arithmetic genus one in positive characteristic.
\newblock {\em arXiv:2308.08927v3}, 2023.

\bibitem[CZ15]{CZ15}
Yifei Chen and Lei Zhang.
\newblock The subadditivity of the {K}odaira dimension for fibrations of relative dimension one in positive characteristics.
\newblock {\em Math. Res. Lett.}, 22(3):675--696, 2015.

\bibitem[Das15]{Das15}
Omprokash Das.
\newblock On strongly {$F$-regular} inversion of adjunction.
\newblock {\em J. Algebra}, 434:207--226, 2015.

\bibitem[dJ96]{deJ96}
A.~J. de~Jong.
\newblock Smoothness, semi-stability and alterations.
\newblock {\em Publ. Math., Inst. Hautes {\'E}tud. Sci.}, 83:51--93, 1996.

\bibitem[Eji17]{Eji17}
Sho Ejiri.
\newblock Weak positivity theorem and {Frobenius} stable canonical rings of geometric generic fibers.
\newblock {\em J. Algebraic Geom.}, 26(4):691--734, 2017.

\bibitem[Eji19]{Eji19}
Sho Ejiri.
\newblock Positivity of anticanonical divisors and {$F$-purity} of fibers.
\newblock {\em Algebra Number Theory}, 13(9):2057--2080, 2019.

\bibitem[Eji23]{Eji23}
Sho Ejiri.
\newblock Splitting of algebraic fiber spaces with nef relative anti-canonical divisor and decomposition of {$F$-split} varieties.
\newblock {\em arXiv:2308.11145}, 2023.

\bibitem[Eji24]{Eji24}
Sho Ejiri.
\newblock Direct images of pluricanonical bundles and {Frobenius} stable canonical rings of fibers.
\newblock {\em Algebr. Geom.}, 11(1):71--110, 2024.

\bibitem[Eke87]{Eke87}
Torsten Ekedahl.
\newblock Foliations and inseparable morphisms.
\newblock In {\em Algebraic geometry, {B}owdoin, 1985 ({B}runswick, {M}aine, 1985)}, volume 46, Part 2 of {\em Proc. Sympos. Pure Math.}, pages 139--149. Amer. Math. Soc., Providence, RI, 1987.

\bibitem[Eke88]{Eke88}
Torsten Ekedahl.
\newblock Canonical models of surfaces of general type in positive characteristic.
\newblock {\em Inst. Hautes \'{E}tudes Sci. Publ. Math.}, 67:97--144, 1988.

\bibitem[EP23]{EP23}
Sho Ejiri and Zsolt Patakfalvi.
\newblock The {Demailly--Peternell--Schneider} conjecture is true in positive characteristic.
\newblock {\em arXiv:2305.02157}, 2023.

\bibitem[EvdGM]{E-vdG-M}
Bas Edixhoven, Gerard van~der Geer, and Ben Moonen.
\newblock Abelian varieties.
\newblock \url{https://gerard.vdgeer.net/AV.pdf}.
\newblock in preparation.

\bibitem[FL17]{Fulger-Lehmann-2017-Dual}
Mihai Fulger and Brian Lehmann.
\newblock Positive cones of dual cycle classes.
\newblock {\em Algebr. Geom.}, 4(1):1--28, 2017.

\bibitem[FS20]{FS20}
Andrea Fanelli and Stefan Schr\"{o}er.
\newblock Del {Pezzo} surfaces and {Mori} fiber spaces in positive characteristic.
\newblock {\em Trans. Amer. Math. Soc.}, 373(3):1775--1843, 2020.

\bibitem[Gro65]{EGAIV.2}
A.~Grothendieck.
\newblock \'{E}l\'{e}ments de g\'{e}om\'{e}trie alg\'{e}brique. iv. \'{E}tude locale des sch\'{e}mas et des morphismes de sch\'{e}mas. ii.
\newblock {\em Inst. Hautes \'{E}tudes Sci. Publ. Math.}, 24:231, 1965.

\bibitem[Har77]{Hartshorne-AG}
Robin Hartshorne.
\newblock {\em Algebraic geometry}.
\newblock Graduate Texts in Mathematics, No. 52. Springer-Verlag, New York-Heidelberg, 1977.

\bibitem[Har80]{Hartshorne80}
Robin Hartshorne.
\newblock Stable reflexive sheaves.
\newblock {\em Math. Ann.}, 254(2):121--176, 1980.

\bibitem[HPZ19]{HPZ19}
Christopher~D. Hacon, Zsolt Patakfalvi, and Lei Zhang.
\newblock Birational characterization of abelian varieties and ordinary abelian varieties in characteristic $p>0$.
\newblock {\em Duke Math. J.}, 168(9):1723--1736, 2019.

\bibitem[HW22]{HW22}
Christopher Hacon and Jakub Witaszek.
\newblock The minimal model program for threefolds in characteristic 5.
\newblock {\em Duke Math. J.}, 171(11):2193--2231, 2022.

\bibitem[HX15]{HX15}
Christopher~D. Hacon and Chenyang Xu.
\newblock On the three dimensional minimal model program in positive characteristic.
\newblock {\em J. Amer. Math. Soc.}, 28(3):711--744, 2015.

\bibitem[Iit82]{Iit82}
Shigeru Iitaka.
\newblock {\em Algebraic geometry}, volume~24 of {\em North-Holland Mathematical Library}.
\newblock Springer-Verlag, New York-Berlin, 1982.
\newblock An introduction to birational geometry of algebraic varieties.

\bibitem[JW21]{JW21}
Lena Ji and Joe Waldron.
\newblock Structure of geometrically non-reduced varieties.
\newblock {\em Trans. Amer. Math. Soc.}, 374(12):8333--8363, 2021.

\bibitem[Kaw85a]{Kaw85p}
Y.~Kawamata.
\newblock Pluricanonical systems on minimal algebraic varieties.
\newblock {\em Invent. Math.}, 79(3):567--588, 1985.

\bibitem[Kaw85b]{Kaw85}
Yujiro Kawamata.
\newblock Minimal models and the {K}odaira dimension of algebraic fiber spaces.
\newblock {\em J. Reine Angew. Math.}, 363:1--46, 1985.

\bibitem[Kaw86]{Kaw86}
Yujiro Kawamata.
\newblock On the plurigenera of minimal algebraic $3$-folds with {$K\mathop {\protect \rlap {\raise 3.25pt\hbox {$\sim $}}\protect \rlap {$\approx $}\protect \phantom {\approx }}0$}.
\newblock {\em Math. Ann.}, 275(4):539--546, 1986.

\bibitem[Kee99]{Keel99}
Se\'{a}n Keel.
\newblock Basepoint freeness for nef and big line bundles in positive characteristic.
\newblock {\em Ann. of Math. (2)}, 149(1):253--286, 1999.

\bibitem[{Kol}92]{Kol92}
{Koll{\'a}r, J{\'a}nos}, editor.
\newblock {\em Flips and abundance for algebraic threefolds}.
\newblock Soci\'{e}t\'{e} Math\'{e}matique de France, Paris, 1992.
\newblock Papers from the Second Summer Seminar on Algebraic Geometry held at the University of Utah, Salt Lake City, Utah, August 1991; Ast\'{e}risque No. 211 (1992).

\bibitem[Kol13]{Kollar-Sing}
J\'{a}nos Koll\'{a}r.
\newblock {\em Singularities of the minimal model program}, volume 200 of {\em Cambridge Tracts in Mathematics}.
\newblock Cambridge University Press, Cambridge, 2013.
\newblock With a collaboration of S\'{a}ndor Kov\'{a}cs.

\bibitem[Liu02]{LiuAGAC}
Qing Liu.
\newblock {\em Algebraic geometry and arithmetic curves}, volume~6 of {\em Oxford Graduate Texts in Mathematics}.
\newblock Oxford University Press, Oxford, 2002.
\newblock Translated from the French by Reinie Ern\'{e}; Oxford Science Publications.

\bibitem[LS77]{Lange-Stuhler77}
Herbert Lange and Ulrich Stuhler.
\newblock Vektorb\"{u}ndel auf {K}urven und {D}arstellungen der algebraischen {F}undamentalgruppe.
\newblock {\em Math. Z.}, 156(1):73--83, 1977.

\bibitem[Mat89]{MatsumuraCRT}
Hideyuki Matsumura.
\newblock {\em Commutative ring theory}, volume~8 of {\em Cambridge Studies in Advanced Mathematics}.
\newblock Cambridge University Press, Cambridge, 2 edition, 1989.
\newblock Translated from the Japanese by M. Reid.

\bibitem[MB79]{Moret-Bailly79}
Laurent Moret-Bailly.
\newblock Polarisations de degr\'{e} $4$ sur les surfaces ab\'{e}liennes.
\newblock {\em C. R. Acad. Sci. Paris S\'{e}r. A-B}, 289(16):A787--A790, 1979.

\bibitem[Mor86]{Morrison86}
David~R. Morrison.
\newblock A remark on {Kawamata}'s paper ``{On} the plurigenera of minimal algebraic 3-folds with {$K\mathop {\protect \rlap {\raise 3.25pt\hbox {$\sim $}}\protect \rlap {$\approx $}\protect \phantom {\approx }}0$}''.
\newblock {\em Math. Ann.}, 275:547--553, 1986.

\bibitem[MP97]{MP97}
Yoichi Miyaoka and Thomas Peternell.
\newblock {\em Geometry of higher dimensional algebraic varieties}, volume~26 of {\em DMV Semin.}
\newblock Basel: Birkh{\"a}user, 1997.

\bibitem[Mum66]{Mum66}
D.~Mumford.
\newblock {\em Lectures on curves on an algebraic surface}, volume~59 of {\em Ann. Math. Stud.}
\newblock Princeton University Press, Princeton, NJ, 1966.

\bibitem[Mum70]{MumfordAV}
David Mumford.
\newblock {\em Abelian varieties}, volume~5 of {\em Tata Institute of Fundamental Research Studies in Mathematics}.
\newblock Published for the Tata Institute of Fundamental Research, Bombay by Oxford University Press, London, 1970.

\bibitem[Nor83]{Nor83}
Madhav~V. Nori.
\newblock The fundamental group-scheme of an abelian variety.
\newblock {\em Math. Ann.}, 263(3):263--266, 1983.

\bibitem[Oda71]{Oda71}
Tadao Oda.
\newblock Vector bundles on an elliptic curve.
\newblock {\em Nagoya Math. J.}, 43:41--72, 1971.

\bibitem[Ogu93]{Oguiso93}
Keiji Oguiso.
\newblock A remark on the global indices of {${\protect \bf Q}$-Calabi-Yau $3$-folds}.
\newblock {\em Math. Proc. Cambridge Philos. Soc.}, 114(3):427--429, 1993.

\bibitem[Pat14]{Pat14}
Zsolt Patakfalvi.
\newblock Semi-positivity in positive characteristics.
\newblock {\em Ann. Sci. \'{E}c. Norm. Sup\'{e}r. (4)}, 47(5):991--1025, 2014.

\bibitem[Pos24]{Posva24}
Quentin Posva.
\newblock Resolution of 1-foliations singularities on surfaces and threefolds.
\newblock {\em arXiv:2405.05735v1}, 2024.

\bibitem[PW22]{PW22}
Zsolt Patakfalvi and Joe Waldron.
\newblock Singularities of general fibers and the {LMMP}.
\newblock {\em Amer. J. Math.}, 144(2):505--540, 2022.

\bibitem[PZ19]{PZ19}
Zsolt Patakfalvi and Maciej Zdanowicz.
\newblock On the {Beauville--Bogomolov} decomposition in characteristic $p\ge 0$.
\newblock {\em arXiv:1912.12742}, 2019.

\bibitem[Sch04]{Schroer04}
Stefan Schr\"{o}er.
\newblock Some {Calabi-Yau} threefolds with obstructed deformations over the {W}itt vectors.
\newblock {\em Compos. Math.}, 140(6):1579--1592, 2004.

\bibitem[Sch22]{Sch22}
Stefan Schröer.
\newblock The structure of regular genus-one curves over imperfect fields.
\newblock {\em arXiv:2211.04073v1}, 2022.

\bibitem[Sil09]{SilvermanAEC}
Joseph~H. Silverman.
\newblock {\em The arithmetic of elliptic curves}, volume 106 of {\em Graduate Texts in Mathematics}.
\newblock Springer, Dordrecht, 2 edition, 2009.

\bibitem[{Sta}]{stacks-project}
The {Stacks {P}roject authors}.
\newblock The {S}tacks {P}roject.
\newblock \url{https://stacks.math.columbia.edu}.

\bibitem[Tan21]{Tan21}
Hiromu Tanaka.
\newblock Invariants of algebraic varieties over imperfect fields.
\newblock {\em Tohoku Math. J. (2)}, 73(4):471--538, 2021.

\bibitem[Tat52]{Tate52}
John Tate.
\newblock Genus change in inseparable extensions of function fields.
\newblock {\em Proc. Am. Math. Soc.}, 3:400--406, 1952.

\bibitem[Zha19]{Zha19a}
Lei Zhang.
\newblock Subadditivity of {K}odaira dimensions for fibrations of three-folds in positive characteristics.
\newblock {\em Adv. Math.}, 354:106741, 29, 2019.

\bibitem[Zha20]{Zh20}
Lei Zhang.
\newblock Abundance for 3-folds with non-trivial {A}lbanese maps in positive characteristic.
\newblock {\em J. Eur. Math. Soc. (JEMS)}, 22(9):2777--2820, 2020.
\end{thebibliography}

\end{document}